\documentclass[11pt]{amsart}

\usepackage{mathrsfs}
\newtheorem*{introproposition}{Reduction Lemma}
\newtheorem*{introtheorem}{Two Universal Properties}
\newtheorem{theorem}{Theorem}
\newtheorem{lemma}{Lemma}
\newtheorem{sublemma}{Sublemma}
\newtheorem{corollary}{Corollary}

\newtheorem{definition}{Definition}
\newtheorem{remark}{Remark}
\newtheorem{convention}{Convention}
\usepackage{appendix}
\numberwithin{equation}{section}
\numberwithin{theorem}{section}
\numberwithin{corollary}{section}
\numberwithin{proposition}{section}
\numberwithin{definition}{section}
\numberwithin{lemma}{section}
\numberwithin{sublemma}{section}
\numberwithin{remark}{section}
\numberwithin{convention}{section}

\begin{document}
\title[Isoparametric Hypersurface]{Isoparametric hypersurfaces with four principal curvatures, IV}
\author{Quo-Shin Chi}

\address{Department of Mathematics, Washington University, St. Louis, MO 63130}
\email{chi@math.wustl.edu}

\begin{abstract} We prove that an isoparametric hypersurface with four principal curvatures and multiplicity pair $(7,8)$ is either the one constructed by Ozeki and Takeuchi, or one of the two constructed by Ferus, Karcher, and M\"{u}nzner. This completes the classification of isoparametric hypersurfaces in spheres that \'{E}. Cartan initiated in the late 1930s.
\end{abstract}
\maketitle
\pagestyle{myheadings}
\markboth{QUO-SHIN CHI}{ISOPARAMETRIC HYPERSURFACES}

\section{Introduction} 
The class of isoparametric hypersurfaces with four principal curvatures and multiplicity pair $(7,8)$ in $S^{31}$ is the only one that has remained unclassified~\cite{CCJ},~\cite{Chiq},~\cite{Chiq4},~\cite{DN},~\cite{Mi2},~\cite{Mi3}. The subtlety of a possible classification suggests itself when one looks into the three existing examples that are all inhomogeneous, where the octonion algebra is in full force to interplay with the underlying geometric structure, in contrast to the three other anomalous classes of respective multiplicity pairs $(3,4), (4,5),$ and $(6,9)$, where one category (out of at most two) of each class is homogeneous that carries more manageable structural data for the classification~\cite{CCJ},~\cite{Chiq},~\cite{Chiq4}. 

From an algebraic point of view, a classification must begin with classifying the orthogonal multiplications of type $[7,8,15]$, i.e., classifying those bilinear maps 
$$
F: {\mathbb R}^7\times {\mathbb R}^8\rightarrow {\mathbb R}^{15}
$$
satisfying $|F(x,y)|=|F(x)||F(y)|$, or more conveniently for our setup, classifying the following quadratic composition formula of type $[7,8,15]$
$$
(x_1^2+\cdots+x_7^2)(y_1^2+\cdots+y_8^2)=z_1^2+\cdots+z_{15}^2,
$$
where $z_1,\cdots,z_{15}$ are bilinear in $x_1,\cdots,x_7$ and $y_1,\cdots,y_8$, as can be seen by a glance at the first two identities in~\eqref{conv} below. Indeed, the composition formula is equivalent to the Hurwitz matrix equations
$$
F_a F_b^{tr}+F_b F_a^{tr}=2\delta_{ab} I_8,\quad 1\leq a,b\leq 7,
$$
where 
$$
F_a:= \begin{pmatrix} A_a&\sqrt{2}B_a\end{pmatrix}
$$
for $A_a$ of size 8-by-8 and $B_a$ of sixe 8-by-7. With $F_a$ in place one next solves the same problem for another set of seven matrices 
$$
G_a:=\begin{pmatrix} A_a^{tr}&\sqrt{2}C_a\end{pmatrix}
$$
for some $C_a$ of size 8-by-7. Then $A_a,B_a,C_a$ are candidates to form the shape operator $S_a$, in the normal $a$-direction, of the shape operator of the focal manifold $M_{+}$ of the isoparametric hypersurface of the smaller codimension ($=8$) in the sphere $S^{31}$, given by
$$
S_0=\begin{pmatrix}Id&0&0\\0&-Id&0\\0&0&0\end{pmatrix},\quad
S_a=\begin{pmatrix}0&A_a&B_a\\A_a^{tr}&0&C_a\\B_a^{tr}&C_a^{tr}
&0\end{pmatrix},\quad 1\leq a\leq 7.
$$
The possible choices of $A_a,B_a,C_a$ are further restricted because they must verify that the eigenvalues of $S_n$ are 0 and $\pm 1$ in all normal directions $n$ so that $(S_n)^3=S_n$. Algebraically, this says 
$$
(\sum_{a=0}^7c_aS_a)^3=(\sum_{a=0}^7c_a^2)(\sum_{a=0}^7 c_aS_a),\quad\forall c_0\cdots,c_7\in{\mathbb R},
$$
that an isoparametric hypersurface with four principal curvatures and multiplicity pair $(7,8)$ enjoys, which simplifies to those equations in~\eqref{conv} below, plus a few more not listed (see~\cite[II, p. 45]{OT}). This accounts for the possible second fundamental form of the focal manifold and constitutes the first three of the ten defining identities of an isoparametric hypersurface~\cite[I, p. 523]{OT}. One must then pin down the third fundamental form of the focal manifold that is convoluted with the second fundamental form in the seven remaining identities.

 For instance, one can take $B_a=C_a=0$ in all $F_a$ and $G_a$, which is equivalent to Condition A of Ozeki and Takeuchi~\cite[I]{OT} to the effect that there is a point $p\in M_{+}$ at which the shape operators in all normal directions share the same kernel. Then $A_a$ arise from the left or right multiplication of the octonion algebra. Since the two octonion multiplications are inequivalent, it results in two distinct second fundamental forms and three distinct third fundamental forms that give rise to the three inhomogeneous examples in the case when the multiplicity pair is $(7,8)$. This is the approach taken in~\cite{Chi} to give a different proof of a result in~\cite{DN1} that states that the existence of a point of Condition A implies that the isoparametric hypersurface is one of the three inhomogeneous ones.
 
In general, however, there is no known classification of the above quadratic composition formula. 

Algebraic geometry comes to the rescue. In this paper, we shall refer to our fairly detailed survey articles~\cite{Chi5},~\cite{Chi6} and the references therein for all the background material
that we employed in~\cite{CCJ},~\cite{Chiq},~\cite{Chiq4} without dwelling much on it, unless necessarily, except to remark that
the unified theme in the classification is the  
notion of normal varieties and Serre's criterion for
verifying the normality of a variety, in terms of a subtle
codimension 2 test on the generating functions of the ideal of the
variety. Its technical side we developed in~\cite{CCJ},~\cite{Chiq},~\cite{Chiq4} enabled
us to harness the components $p_0,\cdots,p_{m_{+}}$ of the second fundamental form of the focal manifold $M_{+}$ 
of the smaller codimension $1+m_{+}$ in the sphere, to gain a good global control over the codimension 2
estimate on the variety carved out by $p_0,\cdots,p_{m_{+}}$. In fact, an essential step is to study the {\it singular locus} ${\mathscr S}$
of the (complex) linear system of cones ${\mathcal C}_\lambda$
$$
c_0p_0+\cdots+c_{m_{+}}p_{m_{+}}=0
$$
as $\lambda:=[c_0:\cdots:c_{m_{+}}]$ sweeps out ${\mathbb C}P^{m_{+}}$. The codimension 2 estimate gets sharper
when we understand better how $p_0,\cdots,p_{m_{+}}$ cut the singular locus ${\mathscr S}_\lambda$ of the cone ${\mathcal C}_\lambda$,
remarking that ${\mathscr S}=\cup_{\lambda} {\mathscr S}_\lambda$. 

In~\cite{CCJ},~\cite{Chiq},~\cite{Chiq4}, we were able to classify all isoparametric hypersurfaces with four principal
curvatures, except for the case when the principal multiplicity pair is
$(m_{+},m_{-})=(7,8)$, essentially by exploring the cut between $p_0=p_2=0$ and ${\mathscr S}_\lambda$, remarking that, by symmetry, $p_0=0$ and $p_1=0$
produce the same cut into ${\mathscr S}_\lambda$. Intersection of more varieties needs to be considered for a
global codimension 2 estimate in the case when the multiplicity pair is $(7,8)$, which, however, gets untamed without an effective cutting
strategy.

To overcome this obstacle, we introduce in this paper (see Section~\ref{a}) a notion called {\bf $r$-nullity}, which generalizes
Condition A that is 0-null of Ozeki and Takeuchi, remarking that Condition A is important in the classification   
of the anomalous cases when the multiplicity pair is $(m_{+},m_{-})=(3,4),(4,5),$ or $(6,9)$. 

In fact, for Serre's codimension 2 test
it suffices to consider only those ${\mathscr S}_\lambda$ for which $\lambda=[c_0:\cdots:c_{m_{+}}]$ live in the complex hyperquadric 
$$
c_0^2+\cdots+c_{m_{+}}^2=0,
$$
so that each $\lambda$ is a 2-plane spanned by an (oriented) orthonormal pair $(n_0,n_1)$ of a normal basis
$n_0,n_1,\cdots,n_{m_{+}}$ with the corresponding $p_0,p_1,\cdots,p_{m_{+}}$~\cite{Chiq4}. 
Let $r$ be the number 
$$
r:= m_{+}-\dim(\text{kernel}(S_{n_0})\cap\text{kernel}(S_{n_1})).
$$
We say a normal basis element $n_l,l\geq 2,$ is 
{\bf $r$-null} if $p_l$ is identically zero when it is restricted to ${\mathscr S}_\lambda$. We say the normal basis $n_0,n_1,\cdots,n_{m_{+}}$
is $r$-null if $n_l$ is $r$-null for {\bf all} $l\geq 2$. 

As we shall see, a normal basis being $r$-null is the worst case scenario one can encounter in the codimension 2 estimate, since the intersection
between each $p_l=0,l\geq 2,$ and ${\mathscr S}_\lambda$ is trivial, and hence contributes nothing to the codimension 2 estimate. 

At a first glance, this algebro-geometric definition of $r$-nullity seems to lack of differential-geometric content. However, we show in Section~\ref{a} (see Lemma~\ref{rnull}) that $r$-nullity is equivalent to that all the upper left $(m_{-}-r)$-by-$(m_{+}-r)$ blocks of $B_a$ and $C_a$ vanish for $1\leq a\leq m_{+}$, so that in particular $r$-nullity holds if the generic rank of linear combinations of $B_1,\cdots,B_{m_+}$ is $r$.
It is clear now that 
Condition A is equivalent to that the normal basis is $0$-null.

We may assume the isoparametric hypersurface $M$ with multiplicity pair $(m_+,m_{-})=(7,8)$ is not the one constructed by Ozeki and Takeuchi~\cite[I]{OT}. Then we can conclude in Sections~\ref{4n} and~\ref{mirror} (see
Lemma~\ref{r} and Corollary~\ref{CA}), after a long technical preparation of placing constraints on 1-, 2-, and 3-nullity in Section~\ref{constraints} (with the help of certain codimension 2 estimates given in Appendix I) that the focal manifold $M_{+}$ is generically  
$4$-null when we are away from points of Condition A. This enables us to prove in Section~\ref{mirror} the following 

\begin{introproposition} Let $M$ be an isoparametric hypersurface with multiplicity pair $(m_{+},m_{-})=(7,8)$ not constructed by Ozeki
and Takeuchi. Given any point $p\in M$ with its unit normal 
$n$ and any vector $v$ at $p$
tangent to a curvature surface $($which is a sphere$)$ of dimension $7$, there is a $16$-dimensional Euclidean space passing through $p,n$ and $v$ such that it cuts
$M$ in a {\bf homogeneous} isoparametric hypersurface with multiplicity pair $(m_{+},m_{-})=(3,4)$. 
\end{introproposition}  

The key ingredient in establishing the reduction lemma is to look back and forth at the ``mirror'' points~\cite{Chi} of a point $(x,n)$ on the unit normal bundle of $M_{+}$ and $M_{-}$, where $M_{-}$ is the other focal manifold with larger codimension $1+m_{-}$ in the sphere. Here, by the mirror point $(x^\#,n^\#)$ of $(x,n)$ on the unit normal bundle of $M_{+}$, and the mirror point $(x^*,n^*)$ of $(x,n)$ on the unit normal bundle of $M_{-}$, we mean they are the points
$$
(x^\#,n^\#):=(n,x),\quad (x^*,n^*):=((x+n)/\sqrt{2},(x-n)/\sqrt{2}).
$$
Suffices it to say that the shape operators $S_n,S_{n^\#},$ and $S_{n^*}$ are interlocked (see~\eqref{gOOd},~\eqref{gooD},~\eqref{Good}), so that generic 4-nullity at both $x$ and $x^\#$ enables us to read off many zero blocks of $S_n, S_{n^\#},$ and $S_{n^*}$, which, when viewed at $x^*$, fits exactly in the quaternionic framework in~\cite{Chi}. Indeed, we have (see~\eqref{good}, all counterpart quantities at $x^*$ will be denoted with an extra superscript *)
$$
\aligned
&A_\alpha^*=\begin{pmatrix} 0&0\\0&\cdot\end{pmatrix},\quad B_\alpha^*=\begin{pmatrix} \cdot&0\\0&\cdot\end{pmatrix},\quad
C_\alpha^*=\begin{pmatrix} \cdot&0\\0&\cdot\end{pmatrix},\quad 1\leq \alpha\leq 4;\\
&A_\alpha^*=\begin{pmatrix} 0&\cdot\\\cdot&\cdot\end{pmatrix},\quad B_\alpha^*=\begin{pmatrix} 0&\cdot\\\cdot&\cdot\end{pmatrix},\quad
C_\alpha^*=\begin{pmatrix} 0&\cdot\\\cdot&\cdot\end{pmatrix}, \quad5\leq \alpha\leq 8,
\endaligned
$$
where the lower right blocks are all of size 4-by-4, from which the above reduction lemma follows by investigating how the upper left blocks interact with the remaining blocks through the third fundamental form of $M_{-}$.

We are half way home. To determine the remaining blocks of $S_n^*$, it is more convenient to convert the data to $M_{+}$, where now (see~\eqref{mtx})
$$
\aligned
&A_a=\begin{pmatrix}z_a&0\\0&w_a\end{pmatrix},\quad B_a=\begin{pmatrix}0&0\\0&c_a\end{pmatrix},\quad C_a=\begin{pmatrix}0&0\\0&f_a\end{pmatrix},\quad 1\leq a\leq 3,\\
&A_a=\begin{pmatrix} 0&\beta_a\\\gamma_a&\delta_a\end{pmatrix},\quad B_a=\begin{pmatrix}0&d_a\\b_a&c_a\end{pmatrix},\quad C_a=\begin{pmatrix} 0&g_a\\b_a&f_a\end{pmatrix},\quad 4\leq a\leq 7.
\endaligned
$$
An important observation to make is that $(\sqrt{2}c_a,w_a),1\leq a\leq 3,$ generate a quadratic composition formula of type $[3,4,8]$. In~\cite{CW}, the moduli space of orthogonal multiplications of type $[3,4,p], p\leq 12,$ is studied; when
it is incorporated with the data conversion between $x$ and $x^\#$, we are finally able to specify decisive characteristics of the $b_a,c_a,f_a,d_a,g_a$ blocks, to be presented in Section~\ref{sec7}. The driving force for all this to happen is the crucial step that shows the $b_a$ matrices, $4\leq a\leq 7,$ are generically of rank $\leq 2$, so that when we consider the linear combination 
$$
b(x):=x_1b_4+\cdots+x_4 b_7
$$ 
over the polynomial ring ${\mathbb R}[x_1,\cdots,x_4]$, it perfectly fits in the Koszul complex~\cite[p. 423]{Ei} to let us arrive at the important conclusion that all $b_a,1\leq a\leq 7,$ have a common zero column (see Lemma~\ref{ALG}, Corollary~\ref{cccor}, and Corollary~\ref{corollary10}). We phrase it in the following context.

\begin{introtheorem} If the isoparametric hypersurface with multiplicity pair $(m_{+},m_{-})=(7,8)$ is not the one constructed by Ozeki and Takeuchi, then at each point of $M_{+}$ the intersection of kernels of shape operators in all normal directions, or equivalently, of kernels of all $B_a,1\leq a\leq 7,$ is at least $1$-dimensional, and moreover, it is $1$-dimensional at a generic point. Furthermore, the intersection of kernels of all $B_a^{tr},1\leq a\leq 7,$ is generically $2$-dimensional. The statement also holds for $C_a, 1\leq a\leq 7.$
\end{introtheorem}

These two properties, pivotal for the classification in this paper, can be seen to hold true for the two isoparametric hypersurfaces constructed by Ferus, Karcher and M\"{u}nzner through straightforward calculations in Section~\ref{subsec}
to be given as motivation for subsequent development.

Without plunging into technical details, we point out that, with the characteristic features of $A_a, B_a, C_a,1\leq a\leq 7,$ pinpointed, we shall be able to demonstrate in Section 7 that we can come up with a {\bf Clifford frame} over $M_{-}$ (see~\eqref{cliff}) in which the second universal property above plays a vital role. In essence, a Clifford frame~\cite{CCJ},~\cite{Ch} gives rise to an 8-dimensional sphere worth of intrinsic isometries of $M_{-}$ which can be extended to ambient $Spin(9)$ isometries, and hence the hypersurface is one of the two constructed by Ferus, Karcher, and M\"{u}nzner, if it is not the one constructed by Ozeki and Takeuchi.

It is noteworthy that in recent years there has been much effort to investigate isoparametric foliations on Riemannian manifolds other than the standard spheres, such as exotic spheres~\cite{GT},~\cite{QT}, compact manifolds of positive  scalar curvature~\cite{TX}, complex and quaternionic projective spaces~\cite{DV}, ~\cite{DG}, Damek-Ricci spaces~\cite{DD}, and more generally singular foliations on Riemannian manifolds~\cite{GG},~\cite{GR} (and the references therein). Moreover, since isoparametric hypersurfaces form an ideal testing ground to furnish examples and counterexamples, the Yau conjecture on the first eigenvalues of minimal submanifolds in spheres has been mostly established on such hypersurfaces and their focal manifolds~\cite{TY},~\cite{TXY}, metrics of positive constant scalar curvature have been constructed on products of Riemannian manifolds~\cite{HP}, and moreover, many more stable and unstable examples of Lagrangian submanifolds in the complex hyperquadrics have been given through such (homogeneous) hypersurfaces~\cite{MO},~\cite{MO1}. (The references are by no means exhaustive.) It is hoped that the completed classification of isoparametric hypersurfaces would spur even more advances far beyond the standard sphere.

\section{The basics}
\subsection{Second fundamental form of a focal manifold} Let $M$ be an isoparametric hypersurfaces with four
principal curvatures in the sphere. Let $F$ be its Cartan-M\"{u}nzner
polynomial of degree $g$ that satisfies~\cite[I]{Mu}
\begin{equation}\label{CM}
|\nabla F|^2(x)=g^2|x|^{2g-2},\quad (\Delta F)(x)=(m_{-}-m_{+})g^2|x|^{g-2}/2,
\end{equation}
and let $f$ be the restriction of $F$ to the sphere. 

To fix notation, we make the convention
that its two focal manifolds are $M_{+}:=f^{-1}(1)$ and
$M_{-}:=f^{-1}(-1)$ with respective codimensions $m_{+}+1\leq m_{-}+1$ in
the ambient sphere $S^{2(m_{+}+m_{-})+1}$ by changing $F$ to $-F$ if necessary. 
The principal curvatures
of the shape operator $S_n$ of $M_{+}$ (respectively, $M_{-}$) with respect to any unit
normal $n$ 
are $0,1$ and $-1$, whose multiplicities are, respectively,
$m_{+},m_{-}$ and $m_{-}$ (respectively, $m_{-},m_{+}$ and $m_{+}$).

On the unit normal sphere bundle $UN_{+}$ of $M_{+}$, let $(x,n_0)\in UN_{+}$
be points in a small open set; here $x\in M_{+}$ and $n_0$ is normal to the tangents
of $M_{+}$ at $x$. We define a smooth orthonormal frame
$n_a,e_p,e_\alpha,e_\mu$, where $1\leq a,p\leq m_{+}$ and
$1\leq\alpha,\mu\leq m_{-}$,
in such a way that $n_a$ are tangent to
the unit normal sphere at $n_0$, and $e_p,e_\alpha$
and $e_\mu$, respectively, are basis vectors of the eigenspaces $E_0,E_{+}$ and $E_{-}$
of the shape operator $S_{n_0}$. 
The symmetric matrices $S_a:=S_{n_a}$ relative to $E_{+},E_{-}$ and $E_0$ are
\begin{equation}\label{0a}
S_0=\begin{pmatrix}Id&0&0\\0&-Id&0\\0&0&0\end{pmatrix},\quad
S_a=\begin{pmatrix}0&A_a&B_a\\A_a^{tr}&0&C_a\\B_a^{tr}&C_a^{tr}
&0\end{pmatrix},\quad 1\leq a\leq m_{+},
\end{equation}
where $A_a:E_{-}\rightarrow E_{+}$,
$B_a:E_0\rightarrow E_{+}$ and $C_a:E_0\rightarrow E_{-}$.

Given the second fundamental form $S(X,Y)$, the third fundamental form of $M_{+}$ is the symmetric tensor
$$
q(X,Y,Z):=(\nabla^{\perp}_X S)(Y,Z)/3,
$$
where $\nabla^{\perp}$ is the normal connection. Write 
$$
p_a(X,Y):=\langle S(X,Y),n_a\rangle,\quad q^a(X,Y,Z)=\langle q(X,Y,Z),n_a\rangle,\quad 0\leq a\leq m_{+}.
$$
The Cartan-M\"{u}nzner polynomial $F$ is related to $p_a$ and $q^a$ by
the expansion formula of Ozeki and Takeuchi~\cite[I, p. 523]{OT}
\begin{eqnarray}\nonumber
\aligned
&F(tx+y+w)=t^4+(2|y|^2-6|w|^2)t^2+8(\sum_{i=0}^{m_{+}}p_{i}w_{i})t\\
&+|y|^4-6|y|^2|w|^2+|w|^4-2\sum_{i=0}^{m_{+}}p_{i}^2-8\sum_{i=0}^{m_{+}}q^{i}w_{i}
\\
&+2\sum_{i,j=0}^{m_{+}}\langle \nabla p_{i},\nabla p_{j}\rangle w_{i}w_{j},
\endaligned
\end{eqnarray}
where $w:=\sum_{i=0}^{m_{+}} w_i n_i$, $y$ is tangential to $M_{+}$ at $x$,
$p_i:=p_i(y,y)$ and $q^i:=q^i(y,y,y)$. Note that our definition of $q^i$
differs from that of Ozeki and Takeuchi by a sign. It follows
that the second and third fundamental forms at a single point of $M_{+}$
(or $M_{-}$) determine the isoparametric family, where the two forms are related by
ten rather convoluted equations of Ozeki and Takeuchi~\cite[I, p. 530]{OT}, of which the first
three is a rephrase of the fact that the shape operator $S_n$ in any normal
direction $n$ satisfies $(S_n)^3=S_n$, which implies the following identities,
among others~\cite[II, p.45]{OT}:
\begin{equation}\label{conv}
\aligned
&A_iA_j^{tr}+A_jA_i^{tr}+2(B_iB_j^{tr}+B_jB_i^{tr})=2\delta_{ij}Id;\\
&A_i^{tr}A_j+A_j^{tr}A_i+2(C_iC_j^{tr}+C_jC_i^{tr})=2\delta_{ij}Id;\\
&A_iC_jB_j^{tr}+B_iC_j^{tr}A_j^{tr}+A_jC_iB_j^{tr}\quad\text{is skew-symmetric};\\
&C_jB_j^{tr}A_i+A_j^{tr}B_iC_j^{tr}+C_iB_j^{tr}A_j\quad\text{is skew-symmetric};\\
&B_j^{tr}A_iC_j+C_j^{tr}A_j^{tr}B_i+B_j^{tr}A_jC_i\quad\text{is skew-symmetric};\\
&B_j^{tr}B_i+B_i^{tr}B_j=C_j^{tr}C_i+C_i^{tr}C_j;\\
&(A_iA_i^{tr}+B_iB_i^{tr})B_j+B_j(B_i^{tr}B_i+C_i^{tr}C_i)+B_iB_j^{tr}B_i+\\&
A_jA_i^{tr}B_i+A_iA_j^{tr}B_i+B_iC_i^{tr}C_j+B_iC_j^{tr}C_i=B_j;\\
&C_i^{tr}A_i^{tr}B_i+B_i^{tr}A_iC_i=0.\\
\endaligned
\end{equation}

Lemma 49~\cite[p. 64]{CCJ} ensures
that we can assume 
\begin{equation}\label{BC}
B_1=C_1=\begin{pmatrix}0&0\\0&\sigma\end{pmatrix},
\end{equation}
where $\sigma$ is a nonsingular diagonal matrix of size
$r$-by-$r$ with $r$ the rank of $B_1$,
and $A_1$ is of the form
\begin{equation}\label{A}
A_1=\begin{pmatrix}I&0\\0&\Delta\end{pmatrix},
\end{equation}
where $\Delta=\text{diag}(\Delta_1,\Delta_2,\Delta_3,\cdots)$
is of size $r$-by-$r$, in which $\Delta_1=0$
and $\Delta_i,i\geq 2,$ are nonzero skew-symmetric matrices expressed in the
block form
$\Delta_i=\text{diag}(\Theta_i,\Theta_i,\Theta_i,\cdots)$ with $\Theta_i$ a
2-by-2 matrix of the form
$$
\begin{pmatrix}0&f_i\\-f_i&0\end{pmatrix}
$$
for some $0<f_i<1$, where the block of $\sigma$ corresponding to $\Delta_1=0$
is\, $I/\sqrt{2}$.

\begin{definition} We call a normal basis $n_0,n_1,n_2,\cdots,n_{m_{+}}$ {\rm(}or simply the pair $(n_0,n_1)$\rm{)} {\bf normalized} with {\bf spectral data} $(\sigma,\Delta)$ if $S_0$ and $S_a,1\leq a\leq m_{+},$ are given in~\eqref{0a} satisfying~\eqref{BC} and~\eqref{A}.
\end{definition}

\begin{remark}\label{RK} The geometric meaning of the rank $r$ of $B_1$ is
that $m_{+}-r$
is the dimension of the intersection of the kernels of the two shape operators
$S_0$ and $S_1$. 
\end{remark}

\begin{corollary}\label{inde} Let $(m_{+},m_{-})=(7,8)$. Let an integer $0\leq r\leq 7$ be the rank
of $B_1$ of size $8$-by-$7$, which is normalized as in~\eqref{BC}. 
\begin{description}
\item[(1)] The
first $8-r$ rows of $B_a$ and $C_a$ are
zero for at most one index $a$ between $2$ and $7$ when $r=2$, and at most three indexes $a$ when $r=4$.
\item[(2)]
Assume $r=0$. Away from points of Condition A on $M_{+}$,  no other index $a$ between $2$ and $7$ can make 
the first six rows of $B_a$ and $C_a$ zero if there is an index $c$ between $2$ and $7$ for which $B_c$ is of rank $2$, and at 
most two other such indexes $a$ to make the first four rows of $B_a$ and $C_a$ zero if there is a $B_c$ of rank $4$.
\end{description}
\end{corollary}

\begin{proof} $A_1$ and $B_1=C_1$ are normalized. Assume without loss of generality
that the first $8-r$ rows of $B_i$ and $C_i$ are zero. Write
\begin{equation}\label{extra}
A_i=\begin{pmatrix}\alpha_i&\beta_i\\\gamma_i&\delta_i\end{pmatrix},\quad
B_i=\begin{pmatrix}0&0\\b_i&c_i\end{pmatrix},\quad C_i=\begin{pmatrix}0&0\\e_i&f_i\end{pmatrix},
\end{equation}
where $\delta_i, c_i, f_i$ are of size $r$-by-$r$. 

The first identity of~\eqref{conv} applied to $A_i$ and $j=1$
gives
$$
\alpha_i=-\alpha_i^{tr},\quad\gamma_i^{tr}=\beta_i\Delta,
$$
while the third identity gives
$$
\beta_i\sigma^2=0,
$$
so that $\beta_i=\gamma_i=0$. 
 
Suppose there are $k$ indexes $i_1,\cdots,i_k$ between 2 and 7 satisfying~\eqref{extra}. Then it follows from the first identity of~\eqref{conv} applied to
$A_i,A_j,2\leq i,j,$ that
\begin{equation}\label{clifford}
\alpha_{i_s}\alpha_{i_t}+\alpha_{i_t}\alpha_{i_s}=-2\delta_{st}I.
\end{equation}
Meanwhile, $\alpha_{i_1},\cdots,\alpha_{i_k}$ are linearly independent; or else a suitable
linear combination of them will make, say, $\alpha_{i_1}=0$ after a basis change, which
contradicts~\eqref{clifford}. Therefore, the $k$ $(8- r)$-by-$(8-r)$
matrices $\alpha_{i_1},\cdots,\alpha_{i_k}$ make ${\mathbb R}^{8-r}$ into a Clifford
$C_k$-module, so that $\dim(C_k)$ divides $8-r$. We conclude by the classification table of $C_k$ that $k=1$, i.e., there is only one index $a$ between 2 and 7 when $r=2$
because only $\dim(C_1)=2$ divides $6=8-r$. Likewise, $k\leq 3$ when $r=4$, i.e., there are at most three indexes $a$ between 2 and 7 when $r=4$ because $\dim(C_3)=4$ divides $4=8-r$ while $\dim(C_4)=8$. This proves item (1).

When $r=0$, one of the pairs $(B_2,C_2),\cdots,(B_7,C_7)$
is nonzero, say $(B_2,C_2)\neq 0$, for lack of Condition A. We may swap
$n_1$ and $n_2$ so that the old $n_2$ is now the new $n_1^{'}$ with the new $r'\neq 0$, while
the old $n_1$ is now the new $n_2^{'}$ with the new
$B_{2'}=C_{2'}=0$. We apply item (1) to this new indexing to conclude that
there is at most one index $a'$ between $2'$ and $7'$ for which the first six rows of $B_{a'}$ and $C_{a'}$ are zero when $r'=2$, namely,
$a'=2$ itself. That is, in terms of the old indexing, no $a$ between $3$ and $7$ can make the first $6$ rows of $B_a$ and $C_a$ zero when the old $B_2$ is of rank 2.

Meanwhile, the same argument applies to the new indexing to give at most three indexes
$a'\geq 2$ to make the first four rows of $B_{a'}$ and $C_{a'}$ zero when $r'=4$, namely, $a'=2$ and two other indexes. 
That is, in terms of the old indexing, at most two indexes $a$ between $3$ and $7$ can make the first four rows of $B_a$ and $C_a$ zero when the old $B_2$ is of rank 4. This proves the second statement.

\end{proof}

\subsection{A motivational calculation} Let $\rho_1,\cdots,\rho_7$ be a representation of the (anti-symmetric)\label{subsec}
Clifford algebra $C_7$ on ${\mathbb R}^{16}$.
Set
\begin{eqnarray}\nonumber
\aligned
P_0&:(c,d)\mapsto (c,-d),\\
P_1&:(c,d)\mapsto (d,c),\\
P_{1+i}&:(c,d)\mapsto (\rho_i(d),-\rho_i(c)),\quad 1\leq i\leq 7,
\endaligned
\end{eqnarray}
over ${\mathbb R}^{32}={\mathbb R}^{16}\oplus{\mathbb R}^{16}.$
$P_0,P_1,\cdots,P_8$ form a representation of the (symmetric) Clifford
algebra $C_9'$
on ${\mathbb R}^{32}$.

Following our convention, we denote by $M_{-}$ the focal manifold in each of the two examples constructed by Ferus, Karcher, and M\"{n}zner on which the Clifford action acts. It is well known~\cite{FKM} that $M_{-}$ can be realized as the Clifford-Stiefel manifold. Namely,
\begin{eqnarray}\nonumber
\aligned
M_{-}=\{&(\zeta,\eta)\in S^{31}\subset{\mathbb R}^{16}\times{\mathbb R}^{16}:\\
&|\zeta|=|\eta|=1/\sqrt{2},\zeta\perp\eta,\rho_i(\zeta)\perp\eta,i=1,\cdots,7\}.
\endaligned
\end{eqnarray}



At $(\zeta,\eta)\in M_{-}$, the normal space is
$$
N^*=\text{span}(\epsilon_0:=P_0((\zeta,\eta)),\cdots,\epsilon_8:=P_8((\zeta,\eta))).
$$
$E_0^*$, the 0-eigenspace of the shape operator $S_0^*:=S^*_{\epsilon_0}$, is
$$
E_0^*=\text{span}(\epsilon_9:=P_1P_0((\zeta,\eta)),\cdots,\epsilon_{16}:=P_8P_0((\zeta,\eta))).
$$
$E_{\pm}^*$, the $\pm 1$-eigenspaces of $S^*_0$, are
$$
E_{\pm}^*:=\{X:P_0(X)=\mp X,X\perp N^*\}.
$$
Since $E_{+}^*$ (respectively, $E_{-}^*$) consists of vectors of the form $(0,d)\in{\mathbb R}^{32}$ (respectively, $(f,0)\in{\mathbb R}^{32}$), we obtain
$$
\aligned
&E_{+}^*=\{(0,d): d\perp \zeta,d\perp \eta,d\perp \rho_i(\zeta),\forall i\},\\
&E_{-}^*=\{(f,0): f\perp\zeta,f\perp\eta,f\perp\rho_i(\eta),\forall i\}.
\endaligned
$$
The shape operator $S^*_a$ at $(\zeta,\eta)\in M_{-}$ in the normal direction $\epsilon_a\in N^*$ is
$$
S^*_a(X,Y)=-\langle P_a(X),Y\rangle, \quad0\leq a\leq 8.
$$

For illustrating purpose, let us look at the representation
$$
\rho_i: {\mathbb O}\oplus {\mathbb O}\rightarrow {\mathbb O}\oplus {\mathbb O},\quad \rho_i:(x,y)\mapsto (xe_i,ye_i),\quad 1\leq i\leq 7,
$$
where 
$$
(e_0,e_1,\cdots,e_7):=(1,i,j,k,\epsilon,\epsilon i,\epsilon j,\epsilon k)
$$
are the standard basis elements of the octonion algebra ${\mathbb O}$.

Let us choose 
$$
\zeta=(e_0,e_1)/2,\quad \eta=(e_3,e_4)/2.
$$
We calculate to see 
$$
\aligned 
&E_{+}^*=\{((0,0),(\alpha,\beta)):\alpha=e_1\beta, \beta\perp e_2\},\\
&E_{-}^*=\{((x,y),(0,0)):x=e_3(e_2y),y\perp e_1\}.
\endaligned
$$
Therefore, the 7-by-7 $A_a^*$-block of $S_a^*$ reads
$$
A_\alpha^*=\begin{pmatrix} S^*_\alpha(X_a,Y_p)\end{pmatrix}=\begin{pmatrix}-\langle P_\alpha(X_a,Y_p\rangle\end{pmatrix},
$$
where $X_a,Y_p$ are orthonormal basis elements in $E_{+}^*$ and $E_{-}^*$, respectively, which can be chosen to be
$$
X_a=((0,0),(e_1e_a, e_a))/\sqrt{2},\;\; a\neq 2,\quad Y_p=((e_3(e_2e_p),e_p),(0,0))/\sqrt{2},\;\; p\neq 1.
$$

As said in the introduction, this calculation is conducted at $(x^*,n^*):=((\zeta,\eta),\epsilon_0)$ on the unit normal bundle of $M_{-}$, and we can convert it to its mirror point $(x,n)$ on the unit normal bundle of $M_{+}$, so that in fact the data
$A_\alpha^*$ are converted to the seven 8-by-7 matrices
$$
B_a:=\begin{pmatrix}S^*_\alpha(X_a,X_p)\end{pmatrix},\quad 1\leq a\leq 7.
$$
(See~\eqref{gOOd},~\eqref{Good} for the conversion formulae.) The upshot is the following:
$$
\aligned
&B_1=\begin{pmatrix} 0&0&0&0\\0&0&0&0\\0&0&I&0\\0&0&0&I\end{pmatrix},B_2=\begin{pmatrix}0&0&0&0\\0&0&0&0\\0&0&J&0\\0&0&0&-J\end{pmatrix}, B_3=0, B_4=\begin{pmatrix}0&0&L&0\\0&0&0&0\\0&0&0&0\\0&I&0&0\end{pmatrix}\\
&B_5=\begin{pmatrix}0&0&K&0\\0&0&0&0\\0&0&0&0\\0&-J&0&0\end{pmatrix}, B_6=\begin{pmatrix}0&0&0&I\\0&0&0&0\\0&-L&0&0\\0&0&0&0\end{pmatrix}, B_7=\begin{pmatrix}0&0&0&J\\0&0&0&0\\0&-K&0&0\\0&0&0&0\end{pmatrix},
\endaligned
$$
where 
$$
I:=\begin{pmatrix}1&0\\0&1\end{pmatrix},\quad J:=\begin{pmatrix}0&1\\-1&0\end{pmatrix},\quad 
\quad K=\begin{pmatrix}0&1\\1&0\end{pmatrix},\quad L:=\begin{pmatrix}1&0\\0&-1\end{pmatrix}.
$$
Here, each row is of size 2, and the first column is of size 1 and the remaining columns are of size 2.

Note that $x$ is not of Condition A, and all $B_a$ have a common zero column and all $B_a^{tr}$ have two common zero columns. This is the content mentioned in the two universal properties in the introduction. We shall see in the next section that the basis associated with $B_1,\cdots, B_7$ is 4-null, a notion briefly introduced in the introduction. This example shall be our prototype to keep in mind.

\section{$r$-nullity}\label{a}
\subsection{The layout} 
To fix notation and for the reader's convenience, let us first summarize the layout in~\cite{Chiq},~\cite{Chiq4} of the crucial codimension 2 estimate
in the case when the principal multiplicity pair of the isoparametric
hypersurface is not $(7,8)$. We then point out the insufficiency of this
approach and the need for a notion more general than Condition A of Ozeki
and Takeuchi,
when the principal multiplicity pair
of the isoparametric hypersurface is $(m_{+},m_{-})=(7,8)$. 

Recall that on $M_{+}$ we denote by $S_0,\cdots,S_{m_{+}}$ the shape operators
in the normal directions $n_0,\cdots,n_{m_{+}}$, and by
$p_0,\cdots,p_{m_{+}}$ the corresponding components of the second fundamental form. 

We agree that ${\mathbb C}^{2m_{-}+m_{+}}$ consists of
points $(u,v,w)$
with coordinates $u_\alpha,v_\mu$ and $w_p$, where $1\leq\alpha,\mu\leq m_{-}$
and $1\leq p\leq m_{+}$. For $0\leq k\leq m_{+}$, let
$$
V_k:=\{(u,v,w)\in{\mathbb C}^{2m_{-}+m_{+}}:p_0(u,v,w)=\cdots=p_k(u,v,w)=0\}
$$
be the variety carved out by $p_0,\cdots,p_k$.
We want to estimate the dimension of the subvariety ${\mathscr J}_k$ of
${\mathbb C}^{2m_{-}+m_{+}}$, where
$$
{\mathscr J}_k:=\{(u,v,w)\in{\mathbb C}^{2m_{-}+m_{+}}
:\text{rank of Jacobian of}\; p_0,\cdots,p_k<k+1\}.
$$

$p_0,\cdots,p_{k}$ give rise to a linear system of
cones ${\mathcal C}_\lambda$ in ${\mathbb C}^{2m_{-}+m_+}$ defined by
$$
c_0p_0+\cdots+c_{k}p_{k}=0
$$
with
\begin{equation}\nonumber
\lambda:=[c_0:\cdots:c_{k}]\in{\mathbb C}P^{k}.
\end{equation}
The singular subvariety of ${\mathcal C}_\lambda$ is
\begin{equation}\label{sing}
{\mathscr S}_\lambda:=\{(u,v,w)\in{\mathbb C}^{2m_{-}+m_{+}}:
(c_0S_0+\cdots+c_kS_k)\cdot (u,v,w)^{tr}=0\}.
\end{equation}
We have 
\begin{equation}\nonumber
{\mathscr J}_k=\bigcup_{\lambda\in{\mathbb C}P^k}{\mathscr S}_\lambda.
\end{equation}
Set
$$
J_k:=V_k\cap {\mathscr J}_k=\bigcup_{\lambda\in{\mathbb C}P^k} (V_k\cap{\mathscr S}_\lambda).
$$
$J_k$ is where the Jacobian of $p_0,\cdots,p_k$ fails to be of
rank $k+1$ on the variety $V_k$.

We wish to establish the codimension 2 estimate
\begin{equation}\label{prime}
\dim(J_k)\leq\dim(V_k)-2,
\end{equation}
for all $k\leq m_{+}-1$, to verify that $p_0,p_1,\cdots,p_{m_{+}}$ form a regular
sequence.
                                      
We first estimate the dimension of ${\mathscr S}_\lambda$. We established
in~\cite{Chiq4} that
it suffices to consider those $\lambda$ sitting in the hyperquadric
\begin{equation}\label{quadric}
{\mathcal Q}_{k-1}:=\{[c_0:\cdots:c_k]\in{\mathbb C}P^k:c_0^2+\cdots+c_k^2=0\}.
\end{equation}

Recall the following~\cite[Remark 2, p. 484]{Chiq4}.

\begin{convention}\label{convention}
For each $\lambda=[c_0:\cdots:c_k]\in{\mathcal Q}_{k-1}$, we choose $\tilde{n}_0$ and $\tilde{n}_1$ as follows.
Decompose $n:=c_0n_0+\cdots+c_kn_k$ into its real and imaginary parts 
$n=\alpha+\sqrt{-1}\beta$. Define $\tilde{n}_0$ and $\tilde{n}_1$ by performing the Gram-Schmidt
process on $\alpha$ and $\beta$. Then normalize the shape operators 
$S_{\tilde{n}_0},S_{\tilde{n}_1}$ as in~\eqref{BC} and~\eqref{A},
which results in a $2$-frame $(\tilde{n}_0,\tilde{n}_1)$ that varies smoothly with
$\lambda$. Note that $\lambda$ can be interpreted as the oriented real $2$-plane spanned by $n_{\tilde{0}}$ and $n_{\tilde{1}}$. 

We denote
the rank of the matrix $B_{\tilde{1}}$ associated with $S_{\tilde{n}_1}$ by $r_\lambda$. Recall from Remark~$\ref{RK}$ that
$m_{+}-r_\lambda$ is the dimension of the intersection of the kernel spaces
of $S_{\tilde{n}_0}$ and $S_{\tilde{n}_1}$. 

When it is necessary, we will extend $\tilde{n}_0$ and $\tilde{n}_1$ to
an orthonormal basis $\tilde{n}_0,\tilde{n}_1,\cdots,\tilde{n}_{m_{+}}$ with the corresponding shape
operators $S_{\tilde{0}}:=S_{\tilde{n}_0},S_{\tilde{1}}:=S_{\tilde{n}_1},\cdots,S_{\tilde{m}_{+}}:=S_{\tilde{n}_{m_{+}}}$ and components $p_{\tilde{0}},p_{\tilde{1}},\cdots,
p_{\tilde{m}_{+}}$ of the second fundamental form.
\end{convention}

The convention facilitates the dimension estimate for ${\mathscr S}_\lambda$. Indeed,
the defining equation of ${\mathscr S}_\lambda$ can now be written as
\begin{equation}\label{SY}
(S_{\tilde{1}}- \iota_\lambda S_{\tilde{0}})\cdot (x,y,z)^{tr}=0
\end{equation}
after a basis change for some complex number $\iota_\lambda$.
We decompose $x,y,z$ into $x=(x_1,x_2),y=(y_1,y_2),z=(z_1,z_2)$
with $x_2,y_2,z_2\in{\mathbb C}^{r_{\lambda}}$. We have
\begin{eqnarray}\label{estimate}
\aligned
x_1=-\iota_\lambda y_1,&\qquad y_1=\iota_\lambda x_1,\\
-\Delta x_2+\sigma z_2=-\iota_\lambda y_2,&\qquad \Delta y_2+\sigma z_2=\iota_\lambda x_2,\\
\Delta(x_2&+y_2)=0.
\endaligned
\end{eqnarray}
It follows from the first pair of equations in~\eqref{estimate} that either $x_1=y_1=0$, or both are nonzero
with $\iota_\lambda=\pm\sqrt{-1}$. In both cases, by the second pair of equations
in~\eqref{estimate}, we have
\begin{equation}\label{esti}
(\Delta^2-\iota_\lambda^2I)x_2=(\Delta-\iota_\lambda I)\sigma z_2,\qquad (\Delta^2-\iota_\lambda^2I)y_2=-(\Delta-\iota_\lambda I)\sigma z_2,
\end{equation}
which together with the third equation in~\eqref{estimate} imply that 
$x_2=-y_2$, and so $z_2$ can be solved in terms of $x_2$ by the second pair of
equations in~\eqref{estimate}. (Note that conversely $x_2=-y_2$
can be solved in terms of $z_2$ when $\iota_\lambda\neq\pm f_i\sqrt{-1}$ for all $i$ and any real $0<f_i<1$, so that $z$
can be chosen to be a free variable.) Thus
either $x_1=y_1=0$, in which case
$$
\dim({\mathscr S}_\lambda)=m_{+},
$$
or both $x_1$ and $y_1$ are nonzero,
in which case $y_1=\pm\sqrt{-1}x_1$, where $x_1$ is a free variable, $x_2$ and $y_2$ depend linearly on $z_2$ and 
$z$ is a free variable. Hence,
\begin{equation}\label{EST}
\dim({\mathscr S}_\lambda)=m_{+}+m_{-}-r_{\lambda}.
\end{equation}

Since eventually we must estimate the dimension of 
$$
\bigcup_{\lambda\in{\mathcal Q}_{k-1}}(V_k\cap{\mathscr S}_\lambda),
$$
the essential part of $J_k$ for the codimension 2 test, we introduced the first cut of $V_k$ into
${\mathscr S}_\lambda$ by
\begin{equation}\label{p0cut}
0=p_{\tilde{0}}=\sum_{\alpha}(x_\alpha)^2-\sum_{\mu}(y_\mu)^2.
\end{equation}
We substitute $y_1=\pm \sqrt{-1}x_1$ and $x_2$
and $y_2$ in terms of $z_2$ into $p_{\tilde{0}}=0$ to
deduce
$$
0=(x_1)^2+\cdots+(x_{m_{-}-r_\lambda})^2;
$$
hence $p_{\tilde{0}}=0$ cuts ${\mathscr S}_\lambda$ to reduce the dimension by 1, i.e., by~\eqref{EST},
\begin{equation}\label{sub}
\dim(V_{k}\cap{\mathscr S}_\lambda)\leq m_{+}+m_{-}-r_\lambda-1.
\end{equation}
Consider the incidence space
\begin{equation}\label{incidence}
{\mathcal I}_k:=\{(x,\lambda)\in{\mathbb C}^{2m_{-}+m_{+}}\times
{\mathcal Q}_{k-1}:x\in{\mathscr S}_\lambda\cap V_k\}.
\end{equation}
Let $\pi_1$ and $\pi_2$ be the restriction to ${\mathcal I}_k$ of 
the standard projections from
${\mathbb C}^{2m_{-}+m_{+}}\times {\mathcal Q}_{k-1}$ onto the first and
second factors. We see
$$
\pi_1({\mathcal I}_k)=\bigcup_{\lambda\in{\mathcal Q}_{k-1}}(V_k\cap{\mathscr S}_\lambda).
$$
Moreover, if we stratify ${\mathcal Q}_{k-1}$ 
into locally closed sets (i.e.,
Zariski open sets in their respective closures) 
\begin{equation}\label{L}
{\mathcal L}_j:=\{\lambda\in{\mathcal Q}_{k-1}:r_\lambda=j\},
\end{equation}
then 
$$
{\mathcal W}_j:=\pi_1\pi_2^{-1}({\mathcal L}_j)
$$
stratify 
$$
\bigcup_{\lambda\in{\mathcal Q}_{k-1}}(V_k\cap{\mathscr S}_\lambda).
$$
We thus obtain, by~\eqref{sub},
\begin{equation}\label{W}
\aligned
&\dim({\mathcal W}_j)\leq \dim(\pi_2^{-1}({\mathcal L}_j))
\leq\max_{\lambda\in{\mathcal L}_j}(\dim(V_{k}\cap{\mathscr S}_\lambda))
+\dim({\mathcal L}_j)\\
&\leq (m_{+}+m_{-}-1-j)+\dim({\mathcal L}_j)
\endaligned
\end{equation}
On the other hand, since $V_k$ is cut out by $k+1$ equations, we have
\begin{equation}\label{lower-bound}
\dim(V_k)\geq m_{+}+2m_{-}-k-1.
\end{equation}
Therefore, the {\em a priori} codimension 2 estimate holds true over ${\mathcal L}_j$ when
\begin{equation}\label{est}
m_{-}\geq 2k+1-j-c_j,
\end{equation}
where 
\begin{equation}\label{cod}
c_j:=\text{the codimension of}\;{\mathcal L}_j\;\text{in}\;{\mathcal Q}_{k-1}.
\end{equation}

\subsection{$r$-nullity}
Note that we only utilized cutting ${\mathscr S}_\lambda$ by $p_{\tilde{0}}=0$ to
derive the coarse upper bound in~\eqref{sub} and lower bound in~\eqref{est}.
The lower bound is too rough to be effective when the multiplicity pair is
$(7,8)$.  A better
upper or lower bound will be achieved if we can obtain further nontrivial cuts
into ${\mathscr S}_\lambda$ by other $p_{\tilde{a}}=0,a\geq 1$.

As a matter of fact, $p_{\tilde{1}}=0$ results in the same cut on ${\mathscr S}_\lambda$
as $p_{\tilde{0}}=0$. This follows by the symmetry of~\eqref{SY} so that we can switch the roles of $S_{\tilde{0}}$
and $S_{\tilde{1}}$. Therefore, nontrivial new cuts can only be obtained by $p_{\tilde{a}}=0$ for $a\geq 2$.

On the other hand, the worst case scenario is that $p_{\tilde{a}}$ annihilate ${\mathscr S}_\lambda$
for all $a\geq 2$, in which case no more cuts other than $p_{\tilde{0}}=0$ can be
introduced and~\eqref{est} is the best possible lower bound. We categorize this
worst case in the following definition in the language of~\eqref{estimate}
and~\eqref{esti}.

\begin{definition} Given a normal basis $n_0,\cdots,n_{m_{+}}$
at a point of $M_{+}$ with the usual $A_i,B_i,C_i,1\leq i\leq m_{+}$,
and the normalization
as in~\eqref{0a},~\eqref{BC} and~\eqref{A}
with $r$ the rank of both $B_1$
and $C_1$, let $p_0,\cdots,p_{m_{+}}$ be the associated components of the second fundamental
form. 

Let ${\mathbb C}^{m_{-}}\simeq {\mathbb C}E_{+}$,
${\mathbb C}^{m_{-}}\simeq {\mathbb C}E_{-}$ and
${\mathbb C}^{m_{+}}\simeq {\mathbb C}E_{0}$ be parametrized by $x,y$ and $z$ respectively,
where $E_{+},E_{-}$ and $E_0$ are the eigenspaces of $S_0$ with eigenvalues
$1.-1$ and $0$, respectively. Let $x:=(x_1,x_2),y:=(y_1,y_2)$ and $z:=(z_1,z_2)$
with $x_2,y_2,z_2\in{\mathbb C}^r$. 

We say a normal basis element $n_l,l\geq 2,$ is {\bf $r$-null} if $p_l$
is identically zero when we restrict it to the linear constraints
\begin{equation}\label{cons}
y_1=\iota x_1,\quad y_2=-x_2,\quad z_2=\sigma^{-1}(\Delta+\iota I)x_2,
\quad \iota=\pm \sqrt{-1}.
\end{equation}

We say the normal basis, always understood to be 
with the normalization
~\eqref{0a},~\eqref{BC} and~\eqref{A},
is $r$-null if\, {\bf all} its basis
elements $n_l,l\geq 2,$ are $r$-null. 
\end{definition}

\begin{lemma}\label{rnull} Conditions as given in the above definition,
a normal basis element $n_l$ is $r$-null if and
only if
the upper left $(m_{-}-r)$-by-$(m_{+}-r)$ block of $B_l$ and $C_l$
of $S_{l}$ are zero.
\end{lemma}

\begin{proof} Suppose $n_l$ is $r$-null. Then $p_l$
restricted to the linear constraint in the definition is
\begin{equation}\label{substi}
p_l=\sum_{\alpha=1, p=1}^{m_{-}\,-r,\,m_{+}\,-r}
(S^l_{\alpha p}+\iota T^l_{\alpha p})x_\alpha z_p +\; {\rm other\; terms},
\end{equation}
where
\begin{equation}\label{100}
S^l_{\alpha p}:=\langle S(X_\alpha,Z_p),n_l\rangle,\quad
T^l_{\alpha p}:=\langle S(Y_\alpha,Z_p),n_l\rangle
\end{equation}
for some orthonormal basis $X_\alpha,Y_\alpha,Z_p$ of $E_{+},E_{-},E_0$,
respectively. Therefore, 
$$
S^l_{\alpha p}=T^l_{\alpha p}=0
$$
for $1\leq\alpha\leq m_{-}-r$ and $1\leq p\leq m_{+}-r$.

Conversely, suppose~\eqref{100} is true, from which we first derive
some identities. Let $A_1,B_1,C_1$ be normalized as in~\eqref{BC} and~\eqref{A}.
Write
$$
A_l:=\begin{pmatrix}\alpha&\beta\\\gamma&\delta\end{pmatrix},
\quad B_l:=\begin{pmatrix}0&d\\b&c\end{pmatrix},
\quad C_l:=\begin{pmatrix}0&g\\e&f\end{pmatrix},
$$
where $\delta,c,f$ are of size $r$-by-$r$. The third identity
of~\eqref{conv} applied to $i=l$ and $j=1$, with the property
$$
\sigma\Delta=\Delta\sigma,
$$
gives
\begin{equation}\label{b}
\beta-d\sigma^{-1}\Delta+g\sigma^{-1}=0,
\end{equation}
while the fourth identity gives
\begin{equation}\label{d}
d\sigma^{-1}+\gamma^{tr}+g\sigma^{-1}\Delta=0.
\end{equation}
Meanwhile, the sixth identity arrives at
\begin{equation}\label{e}
b=e,\quad c^{tr}\sigma+\sigma c=f^{tr}\sigma+\sigma f.
\end{equation} 
In particular, writing
$$
h:=c-f,
$$
we obtain
\begin{equation}\label{h}
\sigma h+h^{tr}\sigma=0.
\end{equation}
Now, we can rewrite~\eqref{h} as
$$
\sigma(h\sigma^{-1}+\sigma^{-1}h^{tr})\sigma=0,
$$
from which we see
\begin{equation}\label{hs}
h\sigma^{-1}+\sigma^{-1}h^{tr}=0.
\end{equation}

Next, the fifth identity of~\eqref{conv} asserts
\begin{equation}\label{ee}
\sigma(\delta+\delta^{tr})\sigma-\sigma\Delta h+h^{tr}\sigma\Delta=0,
\end{equation}
or equivalently,
$$
\delta+\delta^{tr}-\Delta h\sigma^{-1}+\sigma^{-1}h^{tr}\Delta=0,
$$
or if we employ~\eqref{hs}, which is $h\sigma^{-1}=-\sigma^{-1}h^{tr}$,
we can rewrite it as
\begin{equation}\label{del}
\delta+\delta^{tr}-h\sigma^{-1}\Delta+\sigma^{-1}\Delta h^{tr}=0.
\end{equation}

In general,
$$
p_l/2= x^{tr}A_l y+x^{tr}B_l z+y^{tr}C_l z;
$$
setting $x=x_1+x_2,y=y_1+y_2,z=z_1+z_2$, and employing~\eqref{cons},
we can rewrite it in terms of the independent variables $x_1,x_2$ and $z_1$ as
\begin{equation}\label{p2}
\aligned
&p_l/2\\
&=x_1^{tr}(-\beta+\tau\gamma^{tr}+d\sigma^{-1}(\Delta+\tau I)+\tau g\sigma^{-1}(\Delta+\tau I))x_2
+x_2^{tr}(b-e)z_1\\
&+x_2^{tr}(-\delta-\delta^{tr}+(c-f)\sigma^{-1}(\Delta+\tau I)+ ((c-f)\sigma^{-1}(\Delta+\tau I))^{tr})x_2/2\\
&=x_1^{tr}((-\beta+d\sigma^{-1}\Delta-g\sigma^{-1})+\tau(d\sigma^{-1}+\gamma^{tr}+g\sigma^{-1}\Delta))x_2
+x_2^{tr}(b-e)z_1\\
&+x_2^{tr}((-(\delta+\delta^{tr})+h\sigma^{-1}\Delta-\sigma^{-1}\Delta h^{tr})+\tau(h\sigma^{-1}+\sigma^{-1}h^{tr}))x_2/2\\
&=0
\endaligned
\end{equation}
by~\eqref{b},~\eqref{d},~\eqref{e},~\eqref{hs},~\eqref{del}.
\end{proof}

\begin{corollary}\label{grk} Condition A of Ozeki and Takeuchi is equivalent to that
all normal bases are $0$-null at a point of Condition A.
\end{corollary}

\begin{proof} The statement follows immediately from Lemma~\ref{rnull}.
\end{proof}

\begin{remark} The calculation in Section~$\ref{subsec}$ shows that there the normal basis associated with the displayed $B_1,\cdots,B_7$ is $4$-null.
\end{remark}

\begin{corollary}\label{COR} Let $(m_{+},m_{-})=(7,8)$.
Let $\lambda\in{\mathcal Q}_6$ be given in~\eqref{quadric}
with $S_{\tilde{0}}$ and $S_{\tilde{1}}$ normalized as in Convention~$\ref{convention}$
and~\eqref{BC} and~\eqref{A}.
Suppose
$$
r:=\sup_{\lambda\in{\mathcal Q}_6} r_\lambda.
$$
Then
the upper left $(m_{-}-r)$-by-$(m_{+}-r)$ corner of $B_{\tilde{l}}$ and $C_{\tilde{l}}$
of $S_{\tilde{l}}$ are zero, $2\leq l\leq 7$,
for all $\lambda\in{\mathcal Q}_6$. That is, the basis elements
$\tilde{n}_l,\geq 2,$ are $r$-null.
\end{corollary}

\begin{proof} Pick a generic $\lambda_0\in{\mathcal Q}_6$ at which
$r_{\lambda_0}=r$. Without loss of generality, at $\lambda_0,$ the 2-plane spanned by the frame $(\tilde{n}_0,\tilde{n}_1),$   let us consider $\tilde{n}_2$ with $S_{\tilde{0}}$
and $S_{\tilde{1}}$ normalized as usual by~\eqref{BC} and~\eqref{A}. Set
$$
B_{\tilde{2}}=\begin{pmatrix}a&d\\b&c\end{pmatrix},\quad
C_{\tilde{2}}=\begin{pmatrix}h&g\\e&f\end{pmatrix},
$$
where $c$ and $f$ are of size $r$-by-$r$. We show $a=h=0$.

Let $e_1,\cdots,e_8$ be the standard
(column) basis vectors of ${\mathbb R}^8$.
Consider the 8-by-7 $B(\theta):=\cos(\theta) B_{\tilde{1}}+\sin(\theta) B_{\tilde{2}}$. We have
$$
B(\theta)=\begin{pmatrix}\sin(\theta) a&\sin(\theta) d\\\sin(\theta) b&\cos(\theta)\sigma+\sin(\theta) c\end{pmatrix}.
$$
For a generic choice of $\theta$, the last $r$ columns of $B(\theta)$ are
linearly independent, as is so for those of $\sigma$ at $\theta=0$, which span the column space $V^\theta$ of $B(\theta)$ of dimension $r$.
Note that, dividing out by $\sin(\theta)$, each of the first $7-r$ column vectors of $B_{\tilde{2}}$ belongs to
$V^\theta$. Letting $\theta$ approach zero, we see these $7-r$ vectors also belong
to $V^0$, which is spanned by $e_{9-r},\cdots,e_8$. It follows that $a=0$.
Likewise, $h=0$. This shows that the statement is true for all generic $\lambda\in{\mathcal Q}_6$.
Hence, it is true for all $\lambda\in{\mathcal Q}_6$ by passing to the limit.
\end{proof}

\begin{remark}\label{importantremark} The arguments in Corollary~${\ref{COR}}$
can be strengthened
as follows. Notation as in Corollary~$\ref{COR}$, suppose
$\lambda(\theta),0\leq \theta\leq 1,$ is an analytic curve in ${\mathcal Q}_6$ with $\lambda$ spanned by an oriented frame
$(\tilde{n}_0,\tilde{n}(\theta)),$ where $\tilde{n}(\theta)\perp \tilde{n}_0$ with $\tilde{n}(0)=\tilde{n}_{1}$. Denote by $B(\theta)$
the $B$-block of the shape operator $S_{\tilde{n}(\theta)}$ and suppose $B(0)$ is
normalized as in~\eqref{BC} with rank $r$.

Assume the rank of $B(\theta)=r$ for generic $\theta$.
Then generic $B(\theta)$ has the property that the last $r$ columns are independent
as is the case for $B(0)$. Let us denote the matrix of
the first
$7-r$ columns of $B(\theta)$ by
$$
\begin{pmatrix}a(\theta)\\b(\theta)\end{pmatrix},
$$
where $a(\theta)$ is of size $(8-r)$-by-$(7-r)$. 

Suppose $a(\theta)\neq 0$. It is well-known in analytic curve theory that we can choose the 
Frenet frame $\tilde{n}_2,\cdots,\tilde{n}_7$ such that

\begin{equation}\label{ntheta}
\tilde{n}(\theta)=c_1(\theta){\tilde n}_1+\cdots+c_7(\theta){\tilde n}_7
\end{equation}
for some analytic functions $c_1,\cdots,c_7$, where 
$$
c_1(0)=1,\quad c_l(\theta)=\theta^{k_l}d_l(\theta),\quad d_l(0)\neq 0,\quad k_2<\cdots<k_7, \quad l\geq 2,
$$
where ${\tilde n}_2$ is tangent to $\tilde{n}(\theta)$ with contact order $k_2$ at\, $\theta=0$.

Dividing through by $\theta^{k_2}$, it follows that each column of
$$
\begin{pmatrix}a(\theta)\\b(\theta)\end{pmatrix}/\theta^{k_2}
$$
lives in the vector space $V^\theta$ that converges to
$V^0$ spanned by $e_{9-r},\cdots,e_8$, as $\theta$ approaches zero.
This implies that $\tilde{n}_2$ is $r$-null as in the preceding corollary. Note that the rank of the matrix 

\begin{equation}\label{c1} 
A(\theta):=c_1(\theta) B_{\tilde{n}_1}+c_2(\theta) B_{\tilde{n}_2}
\end{equation}
is $=r$ generically. 

For simplicity of exposition, we assume $r=2$ now, though it is true for any $r$. Dividing through by $c_1(\theta)$, we may assume $c_1(\theta)=1$ in~\eqref{c1},
as far as the rank of $A(\theta)$ is concerned. Then the matrix $A(\theta)$, of rank $2$, takes the form
$$
A(\theta)= \begin{pmatrix}0&0&0&0&0&a_1&b_1\\\cdot&\cdot&\cdot&\cdot&\cdot&\cdot&\cdot\\0&0&0&0&0&a_6&b_6
\\c_1&c_2&c_3&c_4&c_5&1+\alpha&\beta\\d_1&d_2&d_3&d_4&d_5&\gamma&1+\delta\end{pmatrix},
$$
where all the variables $a,b,c,d,\alpha,\beta,\gamma,\delta$ are Taylor series with initial terms of the form $\theta^{k_2}$ about $\theta=0$. We leave it as a simple observation to see that if either the lower left or the upper right block of the matrix
is of rank $2$, then the other is zero; thus $B_{{\tilde n}_2}$, being $A(\theta)$ with the two diagonal\, $1$s removed, is of rank $\leq 2$. Otherwise, 
the upper and lower blocks are both of rank $\leq 1$, in which case we may assume $a_i=b_i=c_j=d_j=0$ for 
$1\leq i\leq 5,1\leq j\leq 4,$ via row and column reductions. 
Then with $c_5,d_5,a_6,b_6,\alpha$ and $\beta$ all essentially being constant multiples of\, $\theta^{k_2}$, it is readily seen that the $3$-by-$3$ lower right diagonal 
determinant being zero {\rm (}because $A(\theta)$ is of rank $2${\rm )} implies that the $3$-by-$3$ determinant without the two diagonal $1$s vanishes as well, i.e., 
that $B_{{\tilde n}_2}$ is again of rank $\leq 2$. $($For instance, we may divide by\, $\theta^{k_2}$ and let $\theta$ go to infinity.$)$
Consequently, the analytic 
\begin{equation}\label{generic}
\cos(\theta)B_{{\tilde n}_1}+\sin(\theta) B_{{\tilde n}_2}
\end{equation} 
is of rank $2$ for generic $\theta$ .

As an application, let ${\mathcal C}$ be an irreducible component 
of ${\mathcal L}_2$ {\rm (}see~\eqref{L} for definition{\rm )}
containing a point $\lambda$ spanned by ${\tilde n}_0$
and $\tilde{n}_1$, for which $r_{\lambda}=2$. Let $S^6$ be the standard unit sphere in ${\tilde n}_0^\perp$, the Euclidean space spanned by ${\tilde n}_1,\cdots,{\tilde n}_7$, and
let ${\mathcal C}_0$ be the connected component of the {\rm (}real{\rm )} variety 
$$
{\mathcal C}_0:=\{\tilde{n}\in S^6\subset {\tilde n}_0^\perp: \text{oriented $2$-plane spanned by}\;({\tilde n}_0,\tilde{n})\in{\mathcal C}\}
$$
containing ${\tilde n}_1$. The circle 
$$
\gamma(\theta):=\cos(\theta) {\tilde n}_1 +\sin(\theta) {\tilde n}_2
$$
spans the so-called {\bf tangent cone} ${\mathcal T}$ of ${\mathcal C}_0$ at ${\tilde n}_1$ in $S^6$, as ${\tilde n}_2$ by our construction above are tangents to
all possible analytic curves through ${\tilde n}_1$ in ${\mathcal C}_0$. By~\eqref{generic}, for a generic ${\tilde n}$ on ${\mathcal T}$, the $2$-plane spanned
by $({\tilde n}_0,\tilde{n})$ belongs to 
${\mathcal L}_2$.

Note, in particular, that when ${\tilde n}_1$ is a generic smooth point on ${\mathcal C}_0$, the tangent cone ${\mathcal T}\subset {\mathcal L}_2$ is just the standard unit sphere 
in the linear space spanned by ${\tilde n}_1$ and the tangent space of ${\mathcal C}_0$ at ${\tilde n}_1$.
\end{remark}

\vspace{2mm}

$r$-nullity turns out to be crucial for understanding the structure of an isoparametric hypersurface
when its multiplicity pair is $(7,8)$. As an immediate application, let us
sharpen the lower bound in~\eqref{est}.

\begin{lemma}\label{sharper} Let $(m_{+},m_{-})=(7,8)$. Fix $\lambda_0$ in an irreducible component ${\mathcal C}$ of 
${\mathcal L}_j$. Let $\lambda_0$ be spanned by the frame $(\tilde{n}_0,\tilde{n}_1)$ and extend it to the normal basis
$\tilde{n}_0,\tilde{n}_1,\tilde{n}_2,\cdots,\tilde{n}_7$,
with $S_{\tilde{0}}$ and $S_{\tilde{1}}$ normalized as in Convention~$\ref{convention}$,~\eqref{BC}, and~\eqref{A}.
Suppose no normal basis elements $\tilde{n}_2,\cdots,\tilde{n}_7$ are $j$-null.
Then over ${\mathcal C}$ we have
\begin{equation}\label{eee}
m_{-}\geq 2k-j-c_j,
\end{equation}
where $c_j$ is the codimension of ${\mathcal C}$ in ${\mathcal Q}_{k-1}$ {\rm (}see~\eqref{cod}{\rm )}.
\end{lemma}

\begin{proof} $r_\lambda=j$ for each $\lambda\in{\mathcal L}_j$ by definition.
By~\eqref{p0cut} and~\eqref{cons}, $p_{\tilde{0}}=0$ cuts ${\mathscr S}_{\lambda_0}$ in the variety
$$
\{(X_1,X_2,Y_1,Y_2,Z_1,Z_2)\},
$$
where $X_1=(x_1,\cdots,x_{8-j}), X_2=(x_{9-j},\cdots,x_8), Z_1=(z_1,\cdots,z_{7-j}),Z_2=(z_{8-j},\cdots,z_7),$ satisfy ($j$ is $r$ in~\eqref{cons}) 
\begin{equation}\label{I}
\sum_{\alpha=1}^{8-j} x_\alpha^2=0,
\end{equation}
$X_1=\pm\sqrt{-1}Y_1$, $X_2=-Y_2$, and $Z_2$ depends linearly on $X_2$ (and vice versa). Since
no bases are $j$-null, we may assume some $p_{\tilde{l}},l\geq 2,$ does not annihilate ${\mathscr S}_{\lambda_0}$,
so that Lemma~\ref{rnull} implies that in the expression (see~\eqref{substi})
\begin{equation}\label{II}
p_l=\sum_{\alpha=1, p=1}^{8-j,7-j}
(S^l_{\alpha p}+\pm\sqrt{-1} T^l_{\alpha p})x_\alpha z_p +\; {\rm other\; terms},
\end{equation}
the displayed sum is nontrivial.~\eqref{I} and~\eqref{II} imply that $p_{\tilde{0}}=p_{\tilde{l}}=0$
cuts down one more dimension in ${\mathscr S}_{\lambda_0}$, which remains true for a generic $\lambda\in{\mathcal C}$, so that the lower
bound in~\eqref{est} is reduced further by 1 to yield~\eqref{eee} for a generic $\lambda$. 

On the other hand, since those nongeneric $\lambda\in{\mathcal C}$ constitute a subvariety of codimension at least 1, the lower bound in~\eqref{eee} still holds ture over this subvariety.
\end{proof}

\section{Constraints on $1,2,3$-nullity}\label{constraints}

\begin{lemma}\label{1null} Let $(m_{+},m_{-})=(7,8)$.
Away from
points of Condition A on $M_{+}$,
no element of a normal basis can be $1$-null.
\end{lemma}

\begin{proof} Set $(B_j,C_j),j\geq 2,$ to be of the form
\begin{equation}\label{e1}
B_j=\begin{pmatrix}0&d_j\\b_j&c_j\end{pmatrix},\quad
C_j=\begin{pmatrix}0&g_j\\e_j&f_j\end{pmatrix},
\end{equation}
for some real numbers $c_j$ and $f_j$. We show $d_j=g_j=c_j=f_j=0$.
 
Indeed, with

\begin{equation}\label{e2}
A_j=\begin{pmatrix}\alpha_j&\beta_j\\\gamma_j&\delta_j\end{pmatrix},\quad
A_1=\begin{pmatrix}I&0\\0&0\end{pmatrix},\quad
B_1=C_1=\begin{pmatrix}0&0\\0&1/\sqrt{2}\end{pmatrix},
\end{equation}
one derives easily $($we suppress the index for notational ease$)$, by setting $i=1,j\geq 2$,
in~\eqref{conv}, that $c=f=\delta=0$, and



\begin{eqnarray}\label{set}
\aligned
&\beta=-\sqrt{2}g,\;\gamma=-\sqrt{2}d^{tr},\;\alpha\gamma^{tr}=\alpha\beta=0,\; |d|=|g|,\\
&\alpha\alpha^{tr}+\beta\beta^{tr}+2dd^{tr}=I,\; b=e,\; |\gamma|^2+2|b|^2=1.
\endaligned
\end{eqnarray}
Suppose $d\neq 0$. By a basis change we may assume
$$
d=(t,0,0,\cdots,0)^{tr}
$$
for some positive number $t$. The skew-symmetry of $\alpha$ and the
second and third identities of~\eqref{set} ensure that the first row and column of $\alpha$
are zero. 

If the first entry of $g$ is zero, by a basis change we may assume
$$
g=(0,s,0,0,\cdots,0)
$$
for some positive $s$, so that the third identity implies that the first two rows and columns of $\alpha$
are zero. Ignoring these trivial rows and columns of $\alpha$, we see that
the remainder of it, denoted $\tilde{\alpha}$ of size 5 by 5, is skew-symmetric, orthogonal and satisfies
$$
\tilde{\alpha}^2=-Id.
$$
That is ${\mathbb R}^5$ is acted on by $\tilde{\alpha}$ as a Clifford $C_1$-module,
so that 5 is divisible by 2, a contradiction. Therefore, the first entry of $g$
is not zero. In particular, the fifth identity implies that all the other entries
of $g$ are zero. Meanwhile, the first, fourth and fifth identities derive
$|d|=1/2=t$, so that
$$
\gamma=(-\sqrt{2}/2,0,0,\cdots,0),\quad \beta=\pm\gamma, \quad d=\pm g,\quad |b|=1/2.
$$
But now the eighth identity of~\eqref{conv} for $i=j$ gives
\begin{equation}\label{bdgtr}
b^{tr}(d^{tr}g+g^{tr}d)=0.
\end{equation}
Consequently, we obtain $d=g=0$, which is contradictory.



\end{proof}

\begin{lemma}\label{2null} Let $(m_{+},m_{-})=(7,8).$ 
Away from points of Condition A on $M_{+}$,   
notation is as in~\eqref{e1}
and~\eqref{e2} with the spectral data change that now
$$
A_1=\begin{pmatrix}I&0\\0&\Delta\end{pmatrix},\quad
B_1=C_1=\begin{pmatrix}0&0\\0&\sigma\end{pmatrix},\quad
\sigma=sId,\quad \Delta=\begin{pmatrix}0&t\\-t&0\end{pmatrix}
$$
with $t=\sqrt{1-2s^2}$. 
\begin{description}
\item[(1)] If a normalized basis $n_0,n_1,\cdots,n_7$ is such that the generic rank of the
linear combinations of 
$B_2,\cdots,B_7\geq 5$, then it cannot be $2$-null.
\item[(2)] If the basis elements
$n_2,n_3,n_4$ are $2$-null, and generic linear combinations of $B_1,B_2,B_3,B_4$ are of rank $\leq 2$, 
 then $r_\lambda\leq 2$ for any $\lambda$ in the $3$-quadric of oriented $2$-planes of ${\mathbb R}^5$ linearly spanned by $n_0,\cdots n_4$. 
\end{description}
\end{lemma}

\begin{proof} To prove the first statement, let $n_0,n_1,\cdots,n_7$ be a 2-null basis.
Let $n=c_2n_2+\cdots+c_7n_7$ be a unit normal vector. Then
$$
B_n=\begin{pmatrix}0&d_n\\b_n&c_n\end{pmatrix},
$$
where $b_n,c_n,d_n$ are the corresponding linear combinations of $b_j,c_j,d_j,2\leq j\leq 7.$
It follows that the rank of $B_n$ is $\leq 4$ by a dimension count, a contradiction. 

To prove the second item, supposing first that all $b_1,\cdots,b_4$ are zero. We employ Remark~\ref{RK} below~\eqref{A} to calculate $r_\lambda$.

Since $B_1,\cdots,B_4$ and $C_1,\cdots,C_4$ are of the form
$$
B_i:=\begin{pmatrix}0&d_i\\0&c_i\end{pmatrix},\quad C_i=\begin{pmatrix}0&g_i\\0&f_i\end{pmatrix},
$$
where $c_i,f_i$ are of size $r$-by-$r$ ($r=2$; we are doing a general argument), a linear combination of $S_c:=c_0S_0+\cdots+c_4S_4$ assumes the form
\begin{equation}\label{sc}
S_c:=\begin{pmatrix}c_0I&A_c&B_c\\A_c^{tr}&-c_0I&C_c\\B_c^{tr}&C_c^{tr}&0\end{pmatrix},
\end{equation}
where $A_c,B_c,C_c$ are the linear combinations of $A_i,B_i,C_i$  with coefficients $c_i,1\leq i\leq 4.$ It follows that  the vector
$$
\begin{pmatrix}0\\0\\z\end{pmatrix},\quad z=\begin{pmatrix}z\\0\end{pmatrix},
$$
where $z$ is of size $(m_{+}-r)$-by-$1$, belongs to the kernel of $S_c$ for all $c$. Therefore, the kernels of any two $S_{c}$ and $S_{c'}$ intersect in a space of dimension at least $m_{+}-r$, so that by Remark~\ref{RK}
$$
r_\lambda\leq m_{+}-(m_{+}-r)=r,
$$ 
where $\lambda$ is the $2$-plane spanned by the two vectors 
$$
c_0n_0+\cdots+c_4n_4,\quad c_0'n_0+\cdots+c_4'n_4.
$$
Consequently, generic $r_\lambda$ for $\lambda\in{\mathcal Q}_3$ is $r$, where ${\mathcal Q}_3$ is the set of oriented $2$-planes in the Euclidean space spanned by $n_0,\cdots,n_4$.

Otherwise, we may assume  $b_2\neq 0$ now. The sixth identity of~\eqref{conv} for $i=1,j\geq 2$ gives
$$
b_j=e_j, \quad c_j-f_j=-(c_j-f_j)^{tr}.
$$
Meanwhile, the same identity for 
$i=j\geq 2$ derives  
\begin{equation}\label{c-f}
b_j^{tr}(c_j-f_j)=0,\quad d_j^{tr}d_j+c_j^{tr}c_j=g_j^{tr}g_j+f_j^{tr}f_j.
\end{equation}

Since $c_2-f_2$ is 2-by-2 and skew-symmetric, it
follows by~\eqref{c-f} that $c_2=f_2.$ Since
a generic linear combination of $b_2,b_3,b_4$ can be renamed to be $b_2$, it furthermore follows that
$$
c_j=f_j,\quad 2\leq j\leq 4,
$$
and so by the second identity of~\eqref{c-f}, we obtain

\begin{equation}\label{ddgg}
d_j^{tr}d_j=g_j^{tr}g_j, \quad 2\leq j\leq 4.
\end{equation}

The fifth identity of~\eqref{conv} for $i\geq 2,j=1$ asserts
$$
\sigma\delta_i\sigma-\sigma\Delta (c_i-f_i)\quad\text{is skew-symmetric},
$$
so that $\sigma\delta_i\sigma$ is skew-symmetric as $c_i=f_i$. Thus we deduce
$$
\delta_i=\begin{pmatrix}0&a_i\\-a_i&0\end{pmatrix}
$$
for some number $a_i$. This imposes one linear constraint. Hence we may assume $\delta_2=0$
in the linear span of $B_2,B_3,B_4$. (Note that with this frame change $b_2$ need not be nonzero anymore.)  

Now, the first and second identities of~\eqref{conv} for $i\geq 2,j=1$ result in
\begin{equation}\label{delta}
\delta^{tr}_i\Delta-\Delta\delta_i=-2s(f_i+f_i^{tr})=-2s(c_i+c_i^{tr})=-\delta_i\Delta+\Delta\delta_i^{tr}.
\end{equation}
It follows that
\begin{equation}\label{cskew}
c_i=f_i=\begin{pmatrix}0&p_i\\-p_i&0\end{pmatrix}, \quad 2\leq i\leq 3,
\end{equation}
for some numbers $p_i$, which imposes another linear constraint. We may
therefore assume
\begin{equation}\label{cfd}
c_2=f_2=\delta_2=0
\end{equation}
in the span of $B_2,B_3,B_4$. With~\eqref{cfd} the first and second identities of~\eqref{conv}
for $i=j=2$ give

\begin{eqnarray}\label{many}
\aligned
&\alpha_2\alpha_2^{tr}+\beta_2\beta_2^{tr}+2d_2d_2^{tr}=Id,\quad 
\alpha_2^{tr}\alpha_2+\gamma_2^{tr}\gamma_2+2g_2g_2^{tr}=Id,\\
&\gamma_2\gamma_2^{tr}+2b_2b_2^{tr}=Id,\quad
\beta_2^{tr}\beta_2+2b_2b_2^{tr}=Id,\\
&\alpha_2\gamma_2^{tr}=0,\quad \alpha_2^{tr}\beta_2=0,\\
&\alpha_2=-\alpha_2^{tr},
\endaligned
\end{eqnarray}
where we remark that the last identity comes from setting $i=2,j=1$ in the
first identity of~\eqref{conv}.

Suppose $d_2$ is of rank 2. Then $b_2=0$; or else $B_2$ written as in~\eqref{e1} would be of rank 3 by row reduction. 
Now, since generic linear combination of $b_2,b_3,b_4$ is nonzero, we may
assume $b_3\neq 0$. It follows that
$$
\cos(\theta)B_2+\sin(\theta)B_3
=\begin{pmatrix}0&\cos(\theta) d_2+\sin(\theta)d_3\\ \sin(\theta)b_3&\sin(\theta)c_3\end{pmatrix}
$$
is of rank at least 3 for a small angle $\theta$, because $d_2$ is of rank 2
and $b_3$ is of rank at least 1. Therefore, the generic
linear combination of $B_2,B_3,B_4$ is of rank $\geq 3$, a contradiction.

$d_2$ cannot be of rank 0. This is because otherwise
from~\eqref{ddgg} we obtain 
$$
d_2=g_2=0.
$$ 
Now~\eqref{b} and~\eqref{d} are just
\begin{equation}\label{bega}
\beta_j=(d_j\Delta-g_j)\sigma^{-1},\quad \gamma_j^{tr}=-(d_j+g_j\Delta)\sigma^{-1},j\geq 2;
\end{equation}
in particular, 
$$
\beta_2=\gamma_2=0.
$$
With~\eqref{many}, we arrive at
$$
A_2=\begin{pmatrix}\alpha_2&0\\0&0\end{pmatrix},\quad B_2=\begin{pmatrix}0&0\\b_2&0\end{pmatrix},\quad \alpha_2\alpha_2^{tr}=Id,\quad b_2b_2^{tr}=Id.
$$
The first and second identities of~\eqref{conv} for $i=2, j=3,4$ give
$$
\alpha_2\alpha_j=-\alpha_j\alpha_2,\quad \alpha_2\beta_j=0,\quad\alpha_2\gamma_j^{tr}=0, \quad j=3,4,
$$
from which there follows $\beta_j=\gamma_j^{tr}=0$, so that~\eqref{bega} implies $d_j=g_j=0,j=3,4$,
and so the first and the second identities of~\eqref{conv} derive
$$
\alpha_i\alpha_j=-\alpha_j\alpha_i,\quad \alpha_i^2=-Id,\quad 2\leq i\neq j\leq 4.
$$
However, this says $\alpha_2,\alpha_3,\alpha_4$ induce a Clifford
$C_3$-action on ${\mathbb R}^6$, so that 4 divides 6, a contradiction.

Thus, $d_2$ must be of rank 1. Now $\alpha_2$ cannot be of rank 6; otherwise, the fifth and sixth identities of`\eqref{many} force $\beta_2=\gamma_2=0$ and so~\eqref{bega} gives $d_2=g_2=0$, which is impossible.
Being skew-symmetric, $\alpha_2$ must then be of even rank $\leq 4$. We may thus write
\begin{equation}\label{abc2}
\alpha_2=\begin{pmatrix}\alpha&0\\0&0\end{pmatrix},
\quad\beta_2=\begin{pmatrix}0\\\beta\end{pmatrix},
\quad\gamma_2=\begin{pmatrix}0&\gamma\end{pmatrix},
\end{equation}
where $\alpha$ is of rank 0, 2, 4. $\beta$ is of size
6-by-2, 4-by-2, 2-by-2, and $\gamma$ is of size 2-by-6, 3-by-4,
2-by-2, respectively.

$\alpha$ cannot be of rank 0.
Suppose the contrary. $\beta$ and
$\gamma^{tr}$
are both of size 6-by-2. In particular, $d_2$ and $g_2$ are of the same form as $\beta_2$
and $\gamma_2$, respectively. 
The first identity of~\eqref{many} gives
$$
\beta_2\beta_2^{tr}=I-2d_2d_2^{tr}.
$$
Since the 6-by-6 $d_2d_2^{tr}$ is of rank at most 2 (because $d_2$ is of size 6-by-2), it has eigenvalue 0 counted
at least four times, so that $I-2d_2d_2^{tr}$ has eigenvalue 1 counted at least
4 times and so its rank is at least 4, which contradicts the fact that
$\beta_2\beta_2^{tr}$ is of rank at most 2 ( because $\beta_2$ is of size 6-by-2).

$\alpha$ cannot be of rank 2. Suppose the contrary. We remark that in general any $A_j$ and $B_j$
can be brought to the normalized form of $A_1$ and $B_1$ as in~\eqref{BC} and~\eqref{A}. 
That is, with an appropriate basis change we have
\begin{equation}\label{extra1}
B_j=\begin{pmatrix}0&0\\0&\sigma_j\end{pmatrix},\quad
A_j=\begin{pmatrix}I&0\\0&\Delta_j\end{pmatrix},
\end{equation}
where $\sigma_j$ is diagonal and the nonzero part of $\Delta_j$ is
skew-symmetric in the same form as $\sigma$ and $\Delta$ in~\eqref{BC} and~\eqref{A}. 
In particular, suppose $\sigma_j$
is of size 3-by-3, then $\Delta_j\Delta_j^{tr}$ has a zero eigenvalue so that one of the eigenvalues of $\sigma_j$ is $1/\sqrt{2}$. 

Now, as a consequence
of~\eqref{abc2} and~\eqref{bega}, we obtain
$$
d_2=\begin{pmatrix}0\\d\end{pmatrix},\quad g_2=\begin{pmatrix}0\\g\end{pmatrix},
$$
where $\beta,\gamma^{tr},d$ and $g$ are all of size 4-by-2. The first two
identities of~\eqref{many} give
\begin{equation}\label{bebe}
\beta\beta^{tr}+2dd^{tr}=Id,\quad \gamma^{tr}\gamma+2gg^{tr}=Id,
\end{equation}
from which it follows that the 4-by-4 $\beta\beta^{tr}$,
being of rank $\leq 2$, has eigenvalue 0 counted at least twice, so that $dd^{tr}$ has
eigenvalue $1/\sqrt{2}$ counted at least twice. That is, the 2-by-2 $d^{tr}d$ has eigenvalue $1/\sqrt{2}$ counted exactly twice, so that $d_2$ is of rank 2. The same argument in the paragraph below~\eqref{many} to exclude $d_2$ being of rank 2 then yields a contradiction.




So now the rank of $\alpha$ is $4$ (and we are assuming $d_2$ is of rank 1).

We may assume
\begin{equation}\label{gd}
d=\begin{pmatrix}p&0\\0&0\end{pmatrix},\quad
g=\begin{pmatrix}u&v\\w&z\end{pmatrix}.
\end{equation}

It is important to remark that $d$ can be put in the above diagonal form without changing the values of the normalized $A_1$ and $B_1$ in~\eqref{BC} and~\eqref{A}. In fact, we can first perform a row operation to bring $d$ to an upper triangular form without changing $\sigma$ in $B_1=C_1$. Now due to the fact that $\sigma=sI$, we can then perform a row operation to bring $d$ to the diagonal form. By doing so, we do have to conduct a row operation also on the rows of $\sigma$
to let $\sigma$ continue to be $sI$. 

We employ~\eqref{ddgg} to conclude that 
$$
v=z=0,\quad\quad u^2+w^2=p^2.
$$
 Moreover,~\eqref{bega} gives
\begin{equation}\label{gdgd}
\beta=s^{-1}\begin{pmatrix} -u&pt\\-w&0\end{pmatrix},\quad
\gamma^{tr}=s^{-1}\begin{pmatrix} -p&-tu\\0&-tw\end{pmatrix}.
\end{equation}
Substituting them into~\eqref{bebe} we obtain
\begin{equation}\label{gdgdgd}
u=0,\quad w^2=p^2=s^2.
\end{equation}

We leave it as a simple exercise to conclude the following

\begin{sublemma} $c_2=c_3=c_4=0$. Moreover,  either 
$$
b_i=\begin{pmatrix}0&0&0&0&w_i\\0&0&0&0&z_i\end{pmatrix},\quad d_i=\begin{pmatrix}0&0\\0&0\\0&0\\y_i&-x_i\\0&0\end{pmatrix},
$$ 
where $w_i=\frac{\sqrt{(1-t^2)}}{s}x_i$ and, moreover, $z_i=\frac{\sqrt{(1-t^2)}}{s}y_i$ if $x_i\neq0$, for all $2\leq i\leq 4$, or
$$
b_i=\begin{pmatrix}0&0&0&0&0\\t_{i1}&t_{i2}&t_{i3}&t_{i4}&t_{i5}\end{pmatrix},\quad d_i=\begin{pmatrix}u_{i1}&0\\u_{i2}&0\\u_{i3}&0\\u_{i4}&0\\u_{i5}&0\end{pmatrix}
$$
for all\; $2\leq i\leq 4.$
\end{sublemma}

\begin{proof} (sketch) We know $c_2=0$ by~\eqref{cfd}. By the third identity of~\eqref{many},~\eqref{gd},~\eqref{gdgd}, and~\eqref{gdgdgd}, we obtain
$$
b_2=\begin{pmatrix}0&0\\0&\sqrt{(1-t^2)/2)}\end{pmatrix},\quad d_2=\begin{pmatrix}0&0\\p&0\\0&0\end{pmatrix}, \quad t\neq 1\;\,\text{as}\;\,s\neq 0,
$$
with an appropriate column operation on $b_2$ (note that $b_2$ is of size 2-by-5 and $d_2$ of size 6-by-2). The sublemma follows by the fact that any linear combination of
$B_1,\cdots,B_4$ is of rank $\leq 2$ and so all its $3$-by-$3$ minors are zero while invoking~\eqref{cskew}.
\end{proof}

To finish the proof of the lemma, we shall find the intersection of the kernel spaces of two neighboring $S_c$ and $S_{c'}$  given in~\eqref{sc} for generic choices of $c$ and $c'$. Let $(x,y,z)^{tr}, x,y\in{\mathbb R}^8,z\in{\mathbb R}^7,$ be in the kernel space of $S_c$, which amounts to the following 
\begin{equation}\label{AbC}
c_0x+A_cy+B_cz=0,\quad A_c^{tr}x-c_0y+C_cz=0,\quad B_c^{tr}x+C_c^{tr}y=0.
\end{equation}
Since the choice of $c$ is generic, $B_c$ is of rank 2, so that we can change frame in which~\eqref{BC} and~\eqref{A} hold for $B_c$ with
$$
B_c=C_c=\begin{pmatrix}0&0\\0&\sigma_c\end{pmatrix}, \quad A_c=\begin{pmatrix}I&0\\0&\Delta_c\end{pmatrix}.
$$
The  point is that then the third identity of~\eqref{AbC} implies that if we decompose $x,y,z$, relative to the new frame, into
$$
x=(X_1,X_2),\quad y=(Y_1,Y_2),\quad z=(Z_1,Z_2), \quad X_2,Y_2,Z_2\in{\mathbb R}^2,
$$
then $X_2=-Y_2$ in the space $V_c$ perpendicular to the kernel of $B_c^{tr} $ ($V_c$ is the image of $B_c$). Meanwhile, the first and second identity result in 
$$
Z_2=-c_0X_2+\Delta_cX_2, \quad X_1=Y_1=0,
$$
so that the kernel of $S_c$ is parametrized by $Z_1$ in the kernel of $B_c$ and $X_2$ in the image of $B_c$ (=$V_c$), which is 7-dimensional.

In both cases of the above sublemma, the two generic $c$ and $c'$ introduce a 1-dimensional reduction to the the 5-dimensional kernel of $B_c$, whereas the image of $B_c$ ($=V_c$) retains a common space for the kernels of $S_c$ and $S_{c'}$. In fact, in the former case, 
$$
\text{kernel}\,(B_c^{tr})={\mathbb R}^5\oplus L_c\subset{\mathbb R}^5\oplus{\mathbb  R}^3,
$$
where $L_c$ is a line.  Therefore, $V_c$ is a 2-plane perpendicular to $L_c$ in ${\mathbb R}^3$, so that $V_c$ and $V_{c'}$ intersect in a line in ${\mathbb R}^3$. In the latter case, $V_c\cap V_{c'}$ is the last ($8^{th}$) coordinate line of $x$. 

In any event, the kernels of $S_c$ and $S_{c'}$ intersect in a space of dimension 5, 4 dimensions from the intersection of the kernels of $B_c$ and $B_{c'}$ and 1 dimension from $V_c\cap V_{c'}$. Thus, generically $r_\lambda=7-5=2$.

\end{proof}


\begin{lemma}\label{3null} Let $(m_{+},m_{-})=(7,8)$. Away from points of Condition A on $M_{+}$, let $n_0,\cdots,n_7$ be 
a normal basis such that the frame $(n_0,n_1)$ is normalized with the given spectral data $(\sigma,\Delta)$ as in~\eqref{BC} and~\eqref{A}. Assume $\sigma$ is 
of size $3$-by-$3$ and the generic rank of linear combinations of $B_1,\cdots,
B_7$ is $\geq 5$.
\begin{description} 
\item[(1)] If $\sigma\neq I/\sqrt{2}$, then the normal basis cannot be $3$-null.
\item[(2)] Suppose $n_2,n_3,n_4$ are $3$-null, and moreover, suppose the spectral data of all linear combinations of $B_1$ through $B_4$ are $(\sigma,\Delta)=(I/\sqrt{2},0)$. Then $b_2=b_3=b_4=0$ if all linear combinations of $B_1$ through $B_4$ are of rank $\leq 3$. In particular, under the same condition,  
$r_\lambda\leq 3$ for any $\lambda$ in the $3$-quadric of oriented $2$-planes of ${\mathbb R}^5$ linearly spanned by $n_0,\cdots n_4$. 
\end{description} 
\end{lemma}

\begin{proof} First note that the 3-by-3 matrices
$\sigma$ in $B_1=C_1$ and $\Delta$ in $A_1$ are now
$$
\sigma:=\begin{pmatrix}1/\sqrt{2}&0\\0&s Id\end{pmatrix},\;s\neq 0;\quad
\Delta:=\begin{pmatrix}0&0\\0&t J\end{pmatrix},\quad J:=\begin{pmatrix}0&1\\-1&0
\end{pmatrix},\;t=\sqrt{1-2s^2},
$$
with $A_j,B_j,C_j$ and the associated notation
given in~\eqref{e1},~\eqref{e2}. Since $\sigma$ is of size 3-by-3, the skew-symmetric $\Delta$ must have a zero eigenvalue, which accounts for the eigenvalue $1/\sqrt{2}$ for $\sigma$. 

We prove item (1). $t\neq 0$ in this case. Suppose the normal basis $n_0,n_1,n_2,\cdots,n_7$ is 3-null.

For $i=1,j\geq 2$, we see
\begin{equation}~\label{ee3}
b_j=e_j, \quad\sigma (c_j-f_j)=-(c_j-f_j)^{tr}\sigma, \quad 2\leq j\leq 7,
\end{equation}
by employing the sixth identity of~\eqref{conv}. 

The second identity of~\eqref{ee3} gives
\begin{equation}\label{ee4}
c_j-f_j=\begin{pmatrix}0&v\\-v^{tr}/s\sqrt{2}&w\end{pmatrix},
\end{equation}
where $w$ is 2-by-2 skew-symmetric. On the other hand, the six identity of~\eqref{conv}
for $i=j\geq 2$ results in
\begin{equation}\label{cf}
b_j^{tr}(c_j-f_j)=0,\quad d_j^{tr}d_j+c_j^{tr}c_j=g_j^{tr}g_j+f_j^{tr}f_j.
\end{equation}
Since the $b$-matrices are of rank at least 2 generically for the generic $B_n$
matrices to be of rank $\geq 5$, we see
from~\eqref{ee4}
and~\eqref{cf} that the $c-f$ matrices are zero generically and hence are zero identically,
so that now
\begin{equation}\label{c=f}
c_j=f_j, \quad 2\leq j\leq 7;
\end{equation}
it follows from the fifth identity for $i=1,j\geq 2$, giving
$$
\sigma\delta_j\sigma-\sigma\Delta c_j+\sigma\Delta f_j\;\;\text{is skew-symmetric},
$$
that, by~\eqref{c=f},
\begin{equation}\label{deltaj}
\delta_j\;\;\text{is skew-symmetric},\;\; 2\leq j\leq 7.
\end{equation}
Meanwhile, the seventh identity of~\eqref{conv} asserts
$$
(-\Delta^2+\sigma^2)c_j+2c_j\sigma^2+\sigma c_j^{tr}\sigma+\sigma^2 c_j+\sigma c_j^{tr}\sigma=c_j,
$$
which comes down to
$$
c_j\sigma=-\sigma c_j^{tr},\quad 2\leq j\leq 7,
$$
which gives that $c_j$ is of the form
$$
c_j=\begin{pmatrix}0&c_{j1}\\-c_{j1}^{tr}/s\sqrt{2}&c_{j2}\end{pmatrix},
\quad c_{j2}=-c_{j2 }^{tr},\quad 2\leq j\leq 7,
$$
where $c_{j2}$ is 2-by-2.

Since the matrix form of $c_j,2\leq j\leq 7,$ imposes three linear constraints, we may
thus assume without loss
of generality that                             
\begin{equation}\label{cj=0}
c_2=f_2=0.
\end{equation}
The second
identity of~\eqref{conv} for $i=2, j=1$ then derives
$$
\delta_2^{tr}\Delta-\Delta\delta_2=-2(f_2+f_2^{tr})\sigma=0,
$$
from which there follows, on account of~\eqref{deltaj} and $t\neq 0$, 
\begin{equation}\label{deltaj=0}
\delta_2=0;
\end{equation}
in particular,~\eqref{many} holds true again.
 
 Now the second identity of~\eqref{cf} and~\eqref{c=f} imply
\begin{equation}\label{dd=gg}
d_2^{tr}d_2=g_2^{tr}g_2,
\end{equation}
and moreover the eighth identity of~\eqref{conv} gives
$$
b_2^{tr}(\beta_2^{tr}d_2+\gamma_2g_2)=0,
$$
which, when incorporated with~\eqref{dd=gg} and~\eqref{bega}, arrives at
\begin{equation}\label{dggd}
b_2^{tr}\sigma^{-1}(d_2^{tr}g_2+g_2^{tr}d_2)=0.
\end{equation}


Now, since the 5-by-5 $\alpha_2$ is skew-symmetric, its rank is either
0, 2, or 4, so that $\beta_2$ and $\gamma_2$ being in the kernel of
$\alpha_2$ imply that we can assume

$$
\alpha_2=\begin{pmatrix}\alpha&0\\0&0\end{pmatrix},
\quad\beta_2=\begin{pmatrix}0\\\beta\end{pmatrix},
\quad\gamma_2=\begin{pmatrix}0&\gamma\end{pmatrix},
$$
where $\alpha$ is of rank 0, 2, or 4 of the same square size, $\beta$ is of size 5-by-3, 3-by-3,
or 1-by-3, and $\gamma$ is of size 3-by-5, 3-by-3, or 3-by-1, respectively.

We first rule out the case when $\alpha=0$. Assume $\alpha=0$.
Now, since
$$
B_2B_2^{tr}=\begin{pmatrix}b_2b_2^{tr}&0\\0&d_2d_2^{tr}\end{pmatrix}
$$
is of rank at most 6 (because both $b_2$ and $d_2$ are of rank at most 3) , we see
$$
A_2A_2^{tr}=\begin{pmatrix}\beta_2\beta_2^{tr}&0\\0&\gamma_2\gamma_2^{tr}\end{pmatrix}
$$
has eigenvalue 1
counted at least twice, which implies by the third and fourth identities of~\eqref{many} that
$b_2b_2^{tr}$ has eigenvalue 0 counted at least once so that, in particular,
$b_2b_2^{tr}$ is of rank at most 2 and hence $B_2B_2^{tr}$ is of rank at most 5 and
so in fact $A_2A_2^{tr}$ has eigenvalue 1 counted at least three times.
Thus, either $\beta_2\beta_2^{tr}$ or $\gamma_2\gamma_2^{tr}$ has eigenvalue 1
counted at least twice, so that $b_2b_2^{tr}$ has eigenvalue 0 counted at least twice
and so $b_2$ is of rank at most 1 and $B_2^{tr}B_2$ is of rank at most 4. This forces
$A_2A_2^{tr}$ to have eigenvalue 1 counted at least four times; we conclude,
by 
\begin{equation}\label{BG}
\gamma_2\gamma_2^{tr}=\beta_2^{tr}\beta_2,
\end{equation}
 a consequence of the third and fourth identities of~\eqref{many},
that each of $\beta_2^{tr}\beta_2$ and
$\gamma_2\gamma_2^{tr}$ has some eigenvalue $1-2\epsilon^2,\epsilon\leq 1/2,$ counted once
and eigenvalue 1 counted twice, whereas $d_2^{tr}d_2$ (and $g_2^{tr}g_2$)
has the eigenvalue $\epsilon^2$
counted once and eigenvalue $1/2$ counted twice and $b_2b_2^{tr}$ is of rank at most 1 with
eigenvalue $\epsilon^2$ counted once and eigenvalue 0 counted twice.

By performing a row operation without changing $A_1,B_1,C_1$, we may assume the 5-by-3 $d_2$ is of the form
\begin{equation}\label{D2}
d_2=\begin{pmatrix}d\\0\end{pmatrix},\quad d=\begin{pmatrix}p&y&z\\0&q&w\\0&0&r\end{pmatrix}.
\end{equation}

Write the 5-by-3 $\beta_2$ as
$$\beta_2=\begin{pmatrix}\theta\\\mu\end{pmatrix},
$$
where  $\theta$ is of size 3-by-3. The first identity of~\eqref{many}, with $\alpha_2=0$, gives 
\begin{equation}\label{tm}
\theta\mu^{tr}=0,\quad\mu\mu^{tr}=I.
\end{equation}
Meanwhile, by~\eqref{bega}
\begin{equation}\label{DG}
g_2=\begin{pmatrix}d\Delta-\theta\sigma\\-\mu\sigma\end{pmatrix},\quad \gamma_2^{tr}=\begin{pmatrix}-d\sigma^{-1}(I-\Delta^{tr}\Delta)+\theta\Delta\\\mu\Delta\end{pmatrix}.
\end{equation}

\noindent {\bf Case 1.} $p\neq 0$. With~\eqref{tm}, the vanishing of the off-block of the second identity of~\eqref{many}
calculates to yield
\begin{equation}\label{0d}
\aligned
0&=
(-d\sigma^{-1}(I-\Delta^{tr}\Delta)+\theta\Delta)\Delta^{tr}\mu^{tr}-2(d\Delta-\theta\sigma)\sigma\mu^{tr}\\
&=d(-\sigma^{-1}(I-\Delta^{tr}\Delta)\Delta^{tr}-2\Delta\sigma)u^{tr}+\theta(\Delta\Delta^{tr}+2\sigma^2)u^{tr}\\
&=d\; \text{diag}(\sqrt{2},0,0)\mu^{tr}+\theta\mu^{tr}\\
&=\text{diag}(\sqrt{2}p,0,0)\mu^{tr},
\endaligned
\end{equation}
where we invoke $t^2+2s^2=1$ and the first identity of~\eqref{tm}, and\, $\text{diag}(a,b,c)$ denotes a diagonal matrix whose diagonal entries are $a,b,c$. 
We conclude 
\begin{equation}\label{MU}
\mu=\begin{pmatrix}0&A\end{pmatrix},\quad AA^{tr}=I, \quad\theta=\begin{pmatrix}\tau&0\end{pmatrix}
\end{equation}
where $A$ is of size 2-by-2 and $\tau$ is of size 3-by-1, when we invoke the second identity of~\eqref{tm}. In particular,
$$
\theta\Delta=0,
$$
 from which $\gamma_2$ is simplified to facilitate the calculation of 
\begin{equation}\label{extra2}
\aligned
\gamma_2\gamma_2^{tr}
&=
(-d\sigma^{-1}(I-\Delta^{tr}\Delta)+\theta\Delta)^{tr}
(-d\sigma^{-1}(I-\Delta^{tr}\Delta)+\theta\Delta)\\&+(\mu\Delta)^{tr}(\mu\Delta)
\endaligned
\end{equation}
to derive, by~\eqref{BG}, 
$$
(I+\Delta^2)^{tr}\sigma^{-1}d^{tr}d\sigma^{-1}(I+\Delta^2)+\Delta^{tr}\mu^{tr}\mu\Delta
=\theta^{tr}\theta+\mu^{tr}\mu=\text{diag}(|\tau|^2,1,1)
$$
whose right hand side gives the eigenvalues of $\beta_2^{tr}\beta_2$, which, as we mentioned above, are 1 counted twice and $|\tau|^2=1-2\epsilon^2$; when we invoke~\eqref{MU}, the equality simplifies to
\begin{equation}\label{DTR}
d^{tr}d=\begin{pmatrix}(1-2\epsilon^2)/2&0&0\\0&1/2&0\\0&0&1/2\end{pmatrix}.
\end{equation}
Therefore, since the eigenvalues of $d_2^{tr}d_2$ are $\epsilon^2$ counted once and $1 /2$ counted twice, it implies
$$
\epsilon^2=(1-2\epsilon^2)/2,\quad \text{so}\;\;\epsilon^2=1/4.
$$
On the other hand, $d_2^{tr}d_2$ can be calculated by~\eqref{D2} to compare with~\eqref{DTR} to obtain
$$
y=z=w=0,\quad p^2=1/4,\; q^2=r^2=1/2.
$$
The fourth identity of~\eqref{many} now gives
$$
b_2b_2^{tr}=(I-\beta_2^{tr}\beta_2)/2=\begin{pmatrix}\epsilon^2&0&0\\0&0&0\\0&0&0\end{pmatrix},\quad \epsilon^2=1/4,
$$
while in~\eqref{dggd}
$$
d_2^{tr}g_2+g_2^{tr}d_2=\begin{pmatrix}-\sqrt{2}p\tau_1&-q\tau_2/\sqrt{2}&-r\tau_3/\sqrt{2}\\-q\tau_2/\sqrt{2}&0&(q^2-r^2)t\\-r\tau_3/\sqrt{2}&(q^2-r^2)t&0\end{pmatrix},\quad \tau=(\tau_1.\tau_2,\tau_3)^{tr}.
$$
In particular,~\eqref{dggd} forces $\tau=0$, which is a contradiction as $|\tau|^2=1-2\epsilon^2=1/2$. 

\noindent {\bf Case 2.} $p=0$. We follow essentially the same reasoning as above, except now 
$$
p=q=r=0
$$
since $d$ is of rank $\leq 2$.
Now substitute the triangular form of $d_2$ into~\eqref{extra2} to observe that it is a matrix whose first row and first column are zero, so that when we look at the $(1,1)$-emtry
of the right hand side of~\eqref{BG}, we see that 
\begin{equation}\label{extra3}
\theta=\begin{pmatrix}0&\tau\end{pmatrix},\quad\mu=\begin{pmatrix}0&A\end{pmatrix},\quad AA^{tr}=I,
\end{equation}
where $\tau$ is of size 3-by-2 and $A$ is of size 2-by-2. But then $\theta\mu^{tr}=0$ implies $\tau=0$, i.e., $\theta=0$ now. 
Once more, $\theta\Delta=0$ so that the same analysis as above goes through to achieve
$$
\beta_2^{tr}\beta_2=\text{diag}(0,1,1),\quad d^{tr}d=\text{diag}(0,1/2,1/2).
$$
This is a contradiction, since it says that $1-2\epsilon^2=\epsilon^2=0$.

\vspace{2mm}

Having disposed of the case $\alpha=0 $,
suppose next that  $\alpha$ is of rank 2, so that $\beta$ and $\gamma$
are both of size 3-by-3; by~\eqref{bega} $d_2$ and $g_2$ are of the same form as $\beta_2$
and $\gamma_2$, respectively. Write
\begin{equation}\label{x1x2x3}
d_2=\begin{pmatrix}0\\X\end{pmatrix},
\quad g_2=\begin{pmatrix}0\\Y\end{pmatrix},
\end{equation}
where $X$ and $Y$ are made up of 3-by-1 column vectors $X_1,X_2,X_3$ and
$Y_1,Y_2,Y_3$, respectively.

If $d_2$ is of rank 3. Then~\eqref{dd=gg} implies that there is a 3-by-3 orthogonal
matrix $T$ such that
$$
TX_i=Y_i,\quad 1\leq i\leq 3.
$$

If $b_2$ is of rank 3,~\eqref{dggd} gives
$$
X_i\cdot TX_j=-X_j\cdot TX_i, \quad 1\leq i,j\leq 3,
$$
where $\cdot$ denotes the standard inner product. Consequently, T is skew-symmetric and orthogonal. This is impossible as
$\det(T)=0$ now.

If $b_2$ is of rank $\leq 2$, then $b_2b_2^{tr}$ has an eigenvalue 0, so that
by the fourth identity of~\eqref{many} $\beta_2^{tr}\beta_2$
has an eigenvalue 1. By the first identity of~\eqref{many}, this forces
$XX^{tr}$,
to have an eigenvalue 0, so that $d_2$ is not of rank 3, a contradiction.

Now that $d_2$ being of rank 3 is excluded, we assume the rank of $d_2$ is $\leq 2$. Note that since the lower right 2-by-2 block of $\sigma$ is a multiple of the identity matrix, 
we can perform column operations between the last two columns of $X$ without changing $A_1,B_1$ and $C_1$, though  
we cannot perform column operations to interchange the first and the remaining 
two columns if we want to retain the values of $A_1,B_1$ and $C_1$, for reason that $s\neq 1/\sqrt{2}$. 

By performing a row operation without changing $A_1,B_1,C_1$, we may assume the 3-by-3 $X$ 
takes the form
\begin{equation}\nonumber
X=\begin{pmatrix}d\\0\end{pmatrix},\quad d=\begin{pmatrix}x&y&z\\0&w&u\end{pmatrix},
\end{equation}
where $X$ and $Y$ are given in~\eqref{x1x2x3}.  For notational consistence, we set
$$
\beta_2=\begin{pmatrix}0\\\beta\end{pmatrix},\quad \beta=\begin{pmatrix}\theta\\\mu\end{pmatrix},\quad \gamma_2=\begin{pmatrix}0\\\gamma\end{pmatrix},\quad  Y=g,
$$
where $\beta, Y$ are of size 3-by-3 and $\theta$ is of size 2-by-3. 

As in the previous case when the rank of $\alpha$ is 0, we have two cases to consider, where when $x=0$ we may perform row and column operations to assume $y\neq 0$ and $w=0$.

When $x\neq 0$,  by~\eqref{tm} and~\eqref{0d}, we derive that the first coordinate of the the unit vector $\mu$ is zero. Therefore, by performing a column operation
between the last two columns we may assume
\begin{equation}\label{extra4}
\mu=(0,0,1), \quad \theta=\begin{pmatrix}p&q&0\\r&l&0\end{pmatrix}.
\end{equation}

When $x=w=0$,~\eqref{extra4} remains true with $p=r=0$.

 With these remarks out of the way,~\eqref{DG} gives
$$
g=\begin{pmatrix}-p/\sqrt{2}& -zt-sq&yt\\-r/\sqrt{2}&-ut-sl&wt\\0&0&-s\end{pmatrix},\quad  \gamma^{tr}=\begin{pmatrix} -\sqrt{2}x&-2sy&-2sz+tq\\0&-2sw&-2su+tl\\0&-t&0\end{pmatrix}.
$$
We calculate to see 
$$
\aligned
&\beta^{tr}\beta=\begin{pmatrix} p^2+r^2&pq+rl&0\\pq+rl&q^2+l^2&0\\0&0&1\end{pmatrix}\\
&\gamma\gamma^{tr}=\\
&\begin{pmatrix}2x^2&2\sqrt{2}sxy&-\sqrt{2}x(-2sz+tq)\\2\sqrt{2}sxy&4s^2y^2+4s^2w^2+t^2&-2sy(-2sz+tq)-2sw(-2su+tl)\\
-\sqrt{2}x(-2sz+tq)&-2sy(-2sz+tq)-2sw(-2su+tl)&(-2sz+tq)^2+(-2su+tl)^2\end{pmatrix}.
\endaligned
$$

By~\eqref{BG}, if  $x\neq 0$,
$$
-2sz+tq=0,\quad -2sw(-2su+tl)=0,\quad (-2su+tl)^2=1,\quad \text{so}\;\; w=0;
$$
on the other hand, if $x=w=0$ and $y\neq 0$, we obtain 
$$
-2sy(-2sz+tq)=0,\quad (-2su+tl)^2=1,\quad\text{so}\;\; -2sz+tq=0.
$$
In any event,
\begin{equation}\label{extra5}
-2sz+tq=0,\quad (-2su+tl)^2=1,\quad w=0.
\end{equation}

With these refined data, we observe that the $(2,2)$ entry of $\gamma^{tr}\gamma$ is 1, and thus we can employ the second identity of~\eqref{many} to conclude that the $(2,2)$-entry of $gg^{tr}$ is zero, i.e.,
\begin{equation}\label{extra6}
 r=0,\quad -tu-sl=0,
\end{equation}
In the case when $x=w=0$, we compare the $(2,3)$ entries of~\eqref{ddgg} to conclude
$$
yz=(-zt-sq)ty,\quad \text{so}\;\;z=-zt^2-stq 
$$
which, when incorporated with~\eqref{extra5}, arrives at
$$
z=-zt^2-stq=-zt^2-2s^2z=-(t^2+2s^2)z=-z,\quad\text{so}\;\; z=q=0.
$$ 
But then the $(2,2)$ entry of~\eqref{ddgg} gives
$$
y^2=(zt+sq)^2=0,
$$
a contradiction.

Therefore, $x\neq 0$ is the only possibility, where $w=r=0$ as verified above. We now have the simplified data
\begin{equation}\label{refined}
\aligned
d&=\begin{pmatrix}x&y&z\\0&0&u\end{pmatrix},\quad &
g&=\begin{pmatrix}-p/\sqrt{2}& -zt-sq&yt\\0&0&0\\0&0&-s\end{pmatrix},\\
\beta&=\begin{pmatrix}p&q&0\\0&l&0\\0&0&1\end{pmatrix},\quad&
\gamma^{tr}&=\begin{pmatrix} -\sqrt{2}x&-2sy&0\\0&0&\pm 1\\0&-t&0\end{pmatrix}.
\endaligned
\end{equation}
Accordingly,~\eqref{BG}  simplifies to
$$
2x^2=p^2,\quad pq=2\sqrt{2}sxy,\quad q^2+l^2=4s^2y^2+t^2.
$$
Since $x\neq 0$, we incorporate~\eqref{extra5} and~\eqref{extra6} to solve these equations to obtain
\begin{equation}\label{extra7}
p=\pm\sqrt{2}x,\quad q=\pm 2sy, \quad z=\pm ty,\quad l^2=t^2,
\end{equation}
where the first three equalities share the same sign. We then employ the last equality of~\eqref{extra7} and the second equality of~\eqref{extra6} to see
$$
l={\pm t},\quad u={\mp s},
$$
which means that $l$ and $u$ must differ by a sign. However, since $l$ appears in the second column and $u$ appears in the third, we can certainly change the sign of the basis vector to change the sign of one of $l$ and $u$ 
without affecting the other, while keeping the values of $A_1,B_1$ and $C_1$, to arrange that $l$ and $u$ have the same sign. This is a contradiction.

Lastly, we disprove the case when $\alpha$ is of rank 4, where now
$\beta$ and $\gamma^{tr}$ are 1-by-3. It follows by~\eqref{bega} that $X$ and
$Y$ are 1-by-3. Write 
$$
X:=(a,b,c),\quad \beta:=(p,q,r),
$$
where $X$ is given in~\eqref{x1x2x3}. Then~\eqref{bega} gives, as above,

$$
Y=(-p/\sqrt{2},-tc-sq,tb-sr),\quad \gamma^{tr}=(-\sqrt{2}a,(t^2-1)b/s-tr,(t^2-1)c/s+tq).
$$
Meanwhile, $X^{tr}X=Y^{tr}Y$ and $\beta^{tr}\beta=\gamma\gamma^{tr}$
derive as above

\begin{eqnarray}\nonumber
\aligned
&a=\pm(-p/\sqrt{2}),\quad b=\pm(-tc-sq),\quad c=\pm(tb-sr),\\
&p=\pm(-\sqrt{2}a),\quad q=\pm((t^2-1)b/s-tr),\quad r=\pm((t^2-1)c/s+tq),
\endaligned
\end{eqnarray}
where the three equations in each of the two triples share the same sign. It follows that, by solving the linear system with the unknowns $a,b,c,p,q,r$, we obtain 
\begin{equation}\label{linear}
b=c=q=r=0,
\end{equation}
since $s\neq 1/\sqrt{2}$.
Then, by the third identity of~\eqref{many} we obtain

$$
2b_2b_2^{tr}=\begin{pmatrix}1-2a^2&0&0\\0&1&0\\0&0&1\end{pmatrix},
$$
so that we see from the $(1,1)$-entry of~\eqref{dggd} that

$$
(1-2a^2)a^2=0.
$$
If $a=0$, then $X=Y=0$, or rather $d_2=g_2=0$, so that by~\eqref{bega} $\beta_2=\gamma_2=0$,
which contradicts the first identity of~\eqref{many}. Hence, $2a^2=1$. But then
the first identity of~\eqref{many} results in

$$
1=p^2+q^2+r^2+2a^2+2b^2+2c^2=p^2+2a^2=p^2+1;
$$
we conclude that $p=0$, or rather $\beta=0$, so that $\gamma=0$ by $\beta^{tr}\beta=\gamma\gamma^{tr}$,
and so~\eqref{bega} gives $d_2=g_2=0$, a contradiction again. We are done with item (1).

\vspace{2mm}

To prove item (2), we assume that a generic linear combination of $B_1$ through $B_4$ is of rank 3. Then 
the linear combination
$$
B(\theta):=\cos(\theta) B_1 +\sin(\theta) B_2=\begin{pmatrix} 0&\sin(\theta) d_2\\\sin(\theta) b_2&\cos(\theta) I+\sin(\theta) c_2\end{pmatrix}
$$
is of rank 3 for a generic $\theta$, with $(B_1,C_1)$ and $(B_2,C_2)$ given in~\eqref{BC} and~\eqref{e1}, where now
$$
\sigma=I/\sqrt{2},\quad \Delta=0;
$$
in particular, the first, second, and fifth identities of~\eqref{conv} for $i=2,j=1$ assert
\begin{equation}\label{c2f2}
c_2=-(c_2)^{tr},\;\; f_2=-(f_2)^{tr},\;\;\delta_2=-(\delta_2)^{tr},\;\; \gamma_2^{tr}=-\sqrt{2}d_2,\;\; \beta_2=-\sqrt{2}g_2,
\end{equation}
The kernel of the 8-by-7 $B(\theta)$ is of dimension 4 for a generic $\theta$. Setting $(x,y)^{tr}$ for a kernel vector of $B(\theta)$, where $x$ is of size 1-by-4 and $y$ is of size 1-by-3, we solve to see 
$$
\sin(\theta) d_2 y=0,\quad \sin(\theta) b_2 x+(\cos(\theta)I/\sqrt{2}+\sin(\theta)c_2)y=0,
$$
from which we derive
$$
d_2(\cos(\theta)I/\sqrt{2}+\sin(\theta)c_2)^{-1}b_2x=0,\quad \forall x.
$$
It follows that
$$
0=d_2(\cos(\theta)I/\sqrt{2}+\sin(\theta)c_2)^{-1}b_2=\sum_{k=0}^\infty (-1)^k d_2 (c_2)^k b_2 x^k,\quad x=\sqrt{2}\tan(\theta)
$$
for a generic small $\theta$, which is equivalent to
\begin{equation}\label{d2b2}
d_2(c_2)^kb_2=0,\quad k=0,1,2,3,\cdots
\end{equation}
Likewise, by considering $C_2$ we obtain
\begin{equation}\label{g2b2}
g_2(f_2)^kb_2=0,\quad k=0,1,2,3,\cdots
\end{equation}

Let us first remove the case when $d_2$ is of rank $3$. Performing a row reduction on the matrix $B_2$, we can eliminate $c_2$ without changing $b_2$. It follows that $b_2=0$ because $B_2$ is of rank 3. But since a generic linear combination of $d_2,d_3,d_4$ is also of rank 3,
we see a generic linear combination, and hence all linear combinations of $b_2,b_3,b_4$ are zero. 

\vspace{2mm}

We may now assume that all linear combinations of $d_2,d_3,d_4$ (likewise, of $g_2,g_3,g_4$) are of rank at most 2.  
Assume the rank of $d_2$ is 2. 

If $c_2\neq 0$, performing row and column operations, without changing $B_1, C_1$, and $A_1$, we may assume 
$$
c_2=zJ,\quad J:=\begin{pmatrix}0&0&0\\0&0&1\\0&-1&0\end{pmatrix},\quad z\neq 0;
$$
this is possible because the spectral data $(\sigma,\Delta)=(I/\sqrt{2},0)$ now. 
We then perform a column operation on the last two columns without changing $A_1,B_1,C_1$ and $c_2$, so that
we may assume
\begin{equation}\label{eqn}
d_2=\begin{pmatrix}p&q&0\\0&r&u\\0&0&0\\0&0&0\\0&0&0\end{pmatrix},\quad
b_2=\begin{pmatrix}b_{11}&b_{12}&b_{13}&b_{14}\\b_{21}&b_{22}&b_{23}&b_{24}\\b_{31}&b_{32}&b_{33}&b_{34},
\end{pmatrix}.
\end{equation}
from which~\eqref{d2b2} for $k=0,1$ with $z\neq 0$ results in
$$
pb_{11}+qb_{21}=0,\quad rb_{21}+ub_{31}=0,\quad qb_{31}=0,\quad -ub_{21}+rb_{31}=0.
$$
Generically, we may always assume $p\neq 0$ (by performing row and column operations if necessary). We solve to see that 
$b_2=0$ by the fact that one of $r$ and $u$ is nonzero for $d_2$ to have rank 2. Since the choice of $n_2$ is generic, this says that $b_2=b_3=b_4=0$ if generic combinations of $c_2,c_3,c_4$ are not zero.
So now we may assume
$$
c_2=c_3=c_4=0,\quad\text{and likewise}\quad  f_2=f_3=f_4=0,
$$
and a generic combination of $b_2,b_3,b_4$ is nonzero, which we may assume is $b_2$, without loss of generality.

The rank of $g_2$ is also 2, because
the sixth identity of~\eqref{conv} for $i=j=2$ reads
\begin{equation}\label{d2only}
d_2^{tr}d_2=g_2^{tr}g_2,
\end{equation}
knowing that $c_2=f_2=0$.

Setting $k=0$ in~\eqref{d2b2} and~\eqref{g2b2}, we see that the column space of $b_2$ is identical with  the 1-dimensional kernel space of $d_2$ and of $g_2$. 
We may thus assume
$b_2$ is spanned by $(0,0,1)^{tr}$ and assume
\begin{equation}\label{22}
d_2=\begin{pmatrix}p&q&0\\0&r&0\\0&0&0\\0&0&0\\0&0&0\end{pmatrix},\quad  b_2=\begin{pmatrix}0&0&0&0\\0&0&0&0\\a&0&0&0\end{pmatrix},\quad \delta_2=\begin{pmatrix}0&x&y\\-x&0&w\\-y&-w&0\end{pmatrix}.
\end{equation} 
The first identity of~\eqref{conv} applied to $i=j=2$
gives
\begin{equation}\label{ggddbb}
\gamma_2\gamma_2^{tr}+\delta_2\delta_2^{tr}+2b_2b_2^{tr}=I;
\end{equation}
with the fourth identity of~\eqref{c2f2} one compares the $(1,3),(2,3),$ and $(3,3)$-entries to ensure 
$$
xy=xw=0,\quad 2a^2+y^2+w^2=1.
$$ 

If $x\neq 0$, then $y=w=0$ and $a^2=1/2$, from which we see the nonzero 2-by-2 block $d$ of $d_2$ 
satisfies 
$$
d^{tr}d=(1-x^2)I/2,\quad \text{so} \;\;q=0, \quad p^2=r^2=(1-x^2)/2,
$$ 
incorporating the fourth identity of~\eqref{c2f2} and~\eqref{ggddbb}. However, since the spectral data, 
which are $(\sigma,\Delta)=(I/\sqrt{2},0)$ by assumption, of $B_2$ are the nonzero eigenvalues of 
$d_2^{tr}d_2$ and $b_2b_2^{tr}$ in view of the fact that we can now derive
$$
B_2B_2^{tr}=\begin{pmatrix}d_2d_2^{tr}&0\\0&b_2b_2^{tr}\end{pmatrix},
$$
we therefore conclude that $(1-x^2)/2 =1/2$, i.e., $x=0$, a contradiction. So, $x=0$.


If either $y$ or $w$ is nonzero, we observe first that with $c_2=0$ the first identity of~\eqref{conv} for $i=j=2$ implies
\begin{equation}\label{alga}
\alpha_2\gamma_2^{tr}=-\beta_2\delta_2^{tr},
\end{equation}
which says, by reading the third columns on both sides while invoking the fourth and fifth identities of~\eqref{c2f2}, that
the first two columns of $g_2$ are linearly dependent with coefficients $y$ and $w$, whereas~\eqref{d2only} asserts that the third column of $g_2$ is zero. This forces $g_2$ to be of rank $\leq 1$, a contradiction.
Consequently, $x=y=w=0$ so that $\delta_2=0$.


Now that $c_2=f_2=\delta_2=0$, the same analysis in the proof of item (1) for the case when the ranks of $\alpha$ and $d_2$ are 2 lends its way verbatim to~\eqref{refined}, where $\beta$ and $\gamma$ are also of rank 2 in the case when $t=0$. But then 
$$
A_2A_2^{tr}=\begin{pmatrix}\alpha\alpha^{tr}&0&0\\0&\beta\beta^{tr}&0\\0&0&\gamma\gamma^{tr}\end{pmatrix}
$$
forces $A_2$ to have rank 6, so that the spectral data of $B_2$ cannot be $(\sigma,\Delta)=(I/\sqrt{2},0)$, which would result in the rank of $A_2$ being 5. This case does not occur. 

On the other hand, the same proof as in item (1) in the case when the rank of $\alpha$ is 4 gets us all the way through to the linear system above~\eqref{linear}, where our spectral data is now
$(\sigma,\Delta)=(I/\sqrt{2},0)$. It is easily checked that
\begin{equation}\label{reference}
\beta\beta^{tr}=\gamma^{tr}\gamma=1,\quad XX^{tr}=YY^{tr}=1.
\end{equation}
Now,
$A_2^{tr}A_2$ is of rank 6
with eigenvalue 1 counted six times, four times from $\alpha$ and once each from $\beta$ and $\gamma$,  and 0 counted twice, so that $B_2B_2^{tr}$ is of rank 2 with eigenvalue $1/2$ counted twice. This again contradicts our spectral data assumption $(\sigma,\Delta)=(I/\sqrt{2},0)$.
This case does not occur either.

\vspace{2mm}

Next, we assume generic linear combinations of $d_2,d_3,d_4$ is of rank 1 and $b_2\neq 0$. We know by a symmetric reasoning that $g_2$ has rank $\leq 1$. Assume $c_2\neq 0$. Notation as in~\eqref{eqn}, we remark that the setup in the preceding case is still valid with 
$$
r=u=0
$$
now. We manipulate essentially the same to yield that if $q\neq 0$, then $b_{31}=0$ and $pb_{11}+qb_{21}=0,$ so that $b_2$ is of rank 1 as $b_2\neq 0$. But then the matrix
$$
B_2=\begin{pmatrix}0&d_2\\b_2&c_2\end{pmatrix}
$$ 
will be of rank 4, where the last row of $c_2$ (that of $b_2$ is 0) annihilates $q$ and $r$ of $d_2$ in a row operation, This is a contradiction. Hence, $q=0$, from which it follows that $b_{1j}=0$, i.e., the first row of $b_2$ is zero. Observe now we have
$$
d_2c_2=0,\quad c_2=zJ,
$$
so that  we calculate
$$
B_2B_2^{tr}=\begin{pmatrix}d_2d_2^{tr}&d_2c_2^{tr}\\c_2d_2^{tr}&c_2c_2^{tr}+b_2b_2^{tr}\end{pmatrix}=
\begin{pmatrix}d_2d_2^{tr}&0\\0&c_2c_2^{tr}+b_2b_2^{tr}\end{pmatrix};
$$
therefore, the spectral data dictates that we have 
$$
c_2c_2^{tr}+b_2b_2^{tr}=\begin{pmatrix}0&0\\0&z^2I\end{pmatrix}+\begin{pmatrix}0&0\\0&bb^{tr}\end{pmatrix},\quad b_2:=\begin{pmatrix}0\\b\end{pmatrix},\quad p^2=1/2,
$$
where $I$ of size 2-by-2, and $b$ of size 2-by-3 satisfies
\begin{equation}\label{zb}
z^2I+bb^{tr}=I/2.
\end{equation}
Hence, the identity
$$
\gamma_2\gamma_2^{tr}+\delta_2\delta_2^{tr}+2(b_2b_2^{tr}+c_2c_2^{tr})=I,
$$
obtained by the first identity of~\eqref{conv} for $i=j=2$, 
translates into
$$
\gamma_2\gamma_2^{tr}+\delta_2\delta_2^{tr}=\begin{pmatrix}1&0&0\\0&0&0\\0&0&0\end{pmatrix}.
$$
As a consequence, $\delta_2\delta_2^{tr}=0$
because $p^2=1/2$ and $\gamma_2=-\sqrt{2}d_2$. That is, 
$$
\delta_2=0.
$$ 
With this the first identity of~\eqref{conv} now gives
$$
\alpha_2\gamma_2^{tr}=-d_2c_2^{tr}=0,
$$
which implies that the first column (and hence the first row) of $\alpha_2=0$. Incorporating this into $p^2=1/2$ and
\begin{equation}\label{agd}
\alpha_2\alpha_2^{tr}+2g_2g_2^{tr}+2d_2d_2^{tr}=I
\end{equation}
obtained by the first identity of~\eqref{conv}, we conclude that the first column and row of $g_2g_2^{tr}$ are zero. 
That is, the first row of $g_2$ is zero; moreover, comparing the $(1.1)$-entries and knowing $p^2=1$, we see that the first column and row of $\alpha_2$ are zero since it is skew-symmetric.
Thus we can perform column and row operations, respecting $A_1,B_1,C_1,d_2$ and $c_2$, such that
$$
g_2=\begin{pmatrix}0&0&0\\\theta&\epsilon&0\\0&0&0\\0&0&0\\0&0&0\end{pmatrix}.
$$
Now, since
\begin{equation}\label{2g2}
2g_2^{tr}g_2+2b_2b_2^{tr}+2f_2f_2^{tr}=I
\end{equation}
obtained by the second identity of~\eqref{conv} for $i=j=2$ with $\delta_2=0$, we find that the $(1,3)$- and $(2,3)$-entries of $f_2f_2^{tr}$ are zero. That is,
$$
eg=eh=0,\quad f_2:=\begin{pmatrix}0&e&l\\-e&0&h\\-l&-h&0\end{pmatrix}.
$$
If $e\neq 0$, then $l=h=0$, so that inserting the first equality of~\eqref{zb} into~\eqref{2g2} to compare the $(3,3)$-entries we obtain $z=0$, a contradiction to $c_2\neq 0$. Thus $e=0$. We derive, by the second identity of~\eqref{conv} for $i=j=2$,
$$
\alpha_2^{tr} g_2=-\sqrt{2}g_2f_2^{tr},
$$
where the $(2,3)$-entry of the right hand side is a linear combination of the $(2,1)$- and $(2,2)$-entries of $g_2$ with coefficients $l$ and $h$ and all other entries are zero, whereas the $(2,3)$-entry of the the left hand side is zero. It follows that
$$
g_2f_2^{tr}=0=\alpha_2^{tr}g_2,
$$    
from which we conclude that the second, in addition to the first, column and  row of $\alpha_2$ are zero.
Thus, the second identity of~\eqref{conv} derives
$$
\alpha\alpha^{tr}=I,\quad \alpha_2=\begin{pmatrix}0&0\\0&\alpha\end{pmatrix},
$$
because both $d_2d_2^{tr}$ and $g_2g_2^{tr}$ are nontrivial only at the upper left 2-by-2 block, where $\alpha$ is of size 3-by-3 and skew-symmetric, which is absurd. As a result, $c_2=0$. 

Now that $c_2=0$, we employ the sixth identity of~\eqref{conv}, which gives
$$
d_2^{tr}d_2=g_2^{tr}g_2+f_2^{tr}f_2,
$$
to observe that $g_2$ cannot be of rank 0, or else the left hand side is of rank 1 whereas the right hand side is of rank either 0 or 2. That is, $g_2$ must be of rank 1 as well, so that
exactly the same parallel argument, replacing $d_2$ by $g_2$, establishes $f_2=0$. With now $c_2=f_2=\delta_2=0$, the  same arguments in the paragraph containing~\eqref{reference} results in a contradiction. 
This case does not occur.

Lastly, it is impossible that both $d_2=g_2=0$; for otherwise $\beta_2=\gamma_2=0$.  The first identity of~\eqref{conv} then asserts that the 5-by-5 skew-symmetric $\alpha_2$ satisfies $\alpha_2\alpha_2^{tr}=I$, which is absurd.

\end{proof}

\section{$M_{+}$ is generically $4$-null}\label{4n}

\begin{lemma}\label{la} Let $(m_{+},m_{-})=(7,8)$. Away from points of Condition A on $M_{+}$,
suppose
$$
\sup_{\lambda\in{\mathcal Q}_6} r_{\lambda}\geq 5.
$$
Then there is a choice of
$p_0,\cdots,p_5$ for
the codimension $2$ estimate~\eqref{est}
to go through. In particular, $V_0,\cdots,V_5$ are irreducible and
$p_0,\cdots,p_6$ form a regular sequence.
\end{lemma}

\begin{proof} Recall the {\em a priori} codimension 2 estimate~\eqref{est}, which is

\begin{equation}\label{eest}
8=m_{-}\geq 2k+1-j-c_j,
\end{equation}
where ${\mathcal L}_j$ and $c_j$ are defined in~\eqref{L} and~\eqref{cod}. 
It verifies that the codimension 2 estimate
goes through for $k\leq 3$ and any choice of $p_0,\cdots,p_3$.

For $k=4$, the estimate goes through for $j\geq 1$. However, since
$M_{+}$ away from points of Condition A is not 0-null, item (2) of Corollary~\ref{inde} implies
that for
$k=4$, ${\mathcal L}_0$ is of 
codimension at least 1 in ${\mathcal Q}_3$ (i.e., $c_0\geq 1$), because
by the corollary there must be a $\lambda\in{\mathcal Q}_3$ for which
$r_\lambda\neq 0$; therefore, the codimension 2
estimate goes through, for any choice of $p_0\cdots p_4$. In particular, $V_0,\cdots,V_4$ are irreducible
and any choice of $p_0,\cdots,p_5$ form a regular sequence.

For $k=5$, we pick $p_0,\cdots,p_5$ such that

\begin{equation}\label{sup}
\sup_{\lambda\in{\mathcal Q}_4} r_\lambda\geq 5.
\end{equation}
Note that~\eqref{eest}, which is now

\begin{equation}\label{estim}
8\geq 11-j-c_j,
\end{equation}
implies that the codimension 2 estimate automatically goes through
for $j\geq 3$. 

In general, for $j\leq 4$, ${\mathcal L}_j\subset{\mathcal Q}_4$ is not
generic by~\eqref{sup}, so that $c_j\geq 1$. Hence,~\eqref{estim} also takes care of the codimension 2 estimate for
$j=2$. Moreover, since by Lemma~\ref{1null},  $M_{+}$ is not $j$-null for
$j=1$, the refined codimension 2 estimate~\eqref{eee}, which is

\begin{equation}\label{In}
8=m_{-}\geq 2k-j-c_j,
\end{equation}
is satisfied for $j=1,k=5$ and $c_j\geq 1$; so, the codimension 2 estimate
goes through for $j=1$ as well.

For $j=0$,~\eqref{In} is ineffective as its right hand side is 9 with $c_j\geq 1$; we need to cut down one more dimension
from its right hand side. That is, more fundamentally we must effectively cut
${\mathscr S}_\lambda,\lambda\in{\mathcal L}_0,$ for
generic $\lambda\in{\mathcal L}_0$.

Note, however, 
notation as in Convention~\ref{convention}, since $r_{\lambda}=0$ for $\lambda\in{\mathcal L}_0$, we have $B_{{\tilde 1}}=C_{\tilde{1}}=0$ and $A_{{\tilde 1}}=Id$ in~\eqref{A}
for $S_{{\tilde 1}}$.
It follows that $p_{\tilde{0}}=0$ cuts ${\mathscr S}_\lambda$ in the variety

\begin{equation}\label{free}
\{(x,\pm \sqrt{-1}x,z):\sum_{\alpha=1}^8 (x_\alpha)^2=0\}.
\end{equation}
We may assume $(B_{\tilde{2}},C_{\tilde{2}})$ of $S_{\tilde{2}}$ is nonzero away from points of Condition A.
Since $z$ is a free variable in~\eqref{free},
$p_{\tilde{2}}=0$ cuts ${\mathscr S}_\lambda$ to result in the equation
with nontrivial $z$-terms:

\begin{eqnarray}\label{00}
\aligned
0=p_{\tilde{2}}=\sum_{\alpha=1,p=1}^{8,7}(S_{\alpha p}^2\pm\sqrt{-1}T_{\alpha p}^2)x_\alpha z_p.
\endaligned
\end{eqnarray}
Hence by Lemma~\ref{import} in Appendix I, 
$p_{\tilde{2}}=0$ introduces a nontrivial cut into
${\mathscr S}_\lambda$ to reduce the dimension estimate by 1, and more importantly the variety ${\mathscr F}_2$ cut out by 
$p_{\tilde{0}}=p_{\tilde{2}}=0$
in~\eqref{free} and~\eqref{00} is irreducible.
Indeed, we have seen before that this gives~\eqref{In}. 

To cut one more dimension, we remark that one of the pairs
$(B_{\tilde{3}},C_{\tilde{3}})$, $(B_{\tilde{4}},C_{\tilde{4}})$, and $(B_{\tilde{5}},C_{\tilde{5}})$ is
nonzero, to be in accordance with item (2) of Corollary~\ref{inde}. Hence we may assume
none of them are zero by a generic rotation
of the basis elements $n_{\tilde{3}},n_{\tilde{4}},n_{\tilde{5}}$;
note that, with this, the variety ${\mathscr F}_i$ cut out by
$p_{\tilde{0}}=p_{\tilde{i}}=0,3\leq i\leq 5,$ is also irreducible for the same reason
as ${\mathscr F}_2$.

When ${\mathscr F}_2$ and ${\mathscr F}_j$ are
distinct for some $j=3,4,5$. Then ${\mathscr F}_2\cap {\mathscr F}_j$
is of one dimension lower, i.e.,
$p_{\tilde{0}}=p_{\tilde{2}}=p_{\tilde{j}}=0$ cuts down one more dimension in
${\mathscr S}_\lambda$ by Lemma~\ref{import} in Appendix I, so that the right hand side of~\eqref{In} is dropped
by 1 and so the codimension 2 estimate goes through.

We must then rule out the possibility that ${\mathscr F}_k,2\leq k\leq 5,$ are
all identical, or equivalently, that $p_{\tilde{j}},j=3,4,5,$ restricted to ${\mathscr S}_\lambda$
are all constant multiples of $p_{\tilde{2}}$. That is,

\begin{equation}\label{dep}
S_{\alpha p}^i\pm\sqrt{-1}T_{\alpha p}^i
=c_i(S_{\alpha p}^2\pm\sqrt{-1}T_{\alpha p}^2),
\end{equation}
for some nonzero complex numbers $c_i,3\leq i\leq 5.$

Write
$$
c_i=a_i+\sqrt{-1}b_i
$$
for some real numbers $a_i,b_i$. Then we obtain

\begin{equation}\label{S}
\aligned 
&S^3_{\alpha p}=a_3S^2_{\alpha p}-b_3T^2_{\alpha p},\quad
T^3_{\alpha p}=b_3S^2_{\alpha p}+a_3T^2_{\alpha p},\\
&S^4_{\alpha p}=a_4S^2_{\alpha p}-b_4T^2_{\alpha p},\quad
T^4_{\alpha p}=b_4S^2_{\alpha p}+a_4T^2_{\alpha p},\\
&S^5_{\alpha p}=a_5S^2_{\alpha p}-b_5T^2_{\alpha p},\quad
T^5_{\alpha p}=b_5S^2_{\alpha p}+a_5T^2_{\alpha p}.
\endaligned
\end{equation}
Choose a nonzero solution $(x,y,z),x^2+y^2+z^2=1,$ to
\begin{equation}\label{xyz}
a_3x+a_4y+a_5z=0,\quad b_3x+b_4y+b_5z=0.
\end{equation}
Then it is easily seen that

\begin{equation}\label{combo}
xS^3_{\alpha p}+yS^4_{\alpha p}+zS^5_{\alpha p}=0.
\quad xT^3_{\alpha p}+yT^4_{\alpha p}+zT^5_{\alpha p}=0.
\end{equation}
That is, the shape operator $S_n:=xS_{\tilde{3}}+yS_{\tilde{4}}+zS_{\tilde{5}}$
has the property that its $B$ and $C$ blocks are identically
zero. So we may now
assume the $B$ and $C$ blocks of $S_{\tilde{5}}$ are zero.

We may now ignore
the above $a_5$ and $b_5$ in~\eqref{xyz}. Any nonzero solution $(x,y)$ that solves
the second
equation of~\eqref{xyz} implies that there is a real number $c$, namely,
$c=a_3x+a_4y$, such that the $B$ and $C$ blocks of $S_n:=xS_{\tilde{3}}+yS_{\tilde{4}}$ are $c$ times of $B_{\tilde{2}}$ and $C_{\tilde{2}}$,
respectively. But then $S_{n'}$, where $n'=(n_{\tilde{2}}-cn)/\sqrt{1+c^2}$,
has the property that the $B$ and $C$ blocks of 
$S_{n'}$ are zero. This means that we can now assume that the $B$ and $C$ blocks of
$S_{\tilde{4}}$ are zero, with possible new $S_{\tilde{2}}$ and $S_{\tilde{3}}$ out of the Gram-Schmidt
process. It follows that neither $(B_{\tilde{2}},C_{\tilde{2}})$
nor $(B_{\tilde{3}},C_{\tilde{3}})$ are zero
to not to violate item (2) of Corollary~\ref{inde}.

We are now led to the conclusion that if an irreducible component
${\mathcal C}$
of ${\mathcal L}_0$ is such that, the codimension 2 estimate is not true
for all $\lambda\in {\mathcal C}$, then each $\lambda\in{\mathcal C}$ is contained in
one and only one quadric ${\mathcal Q}_2\subset {\mathcal C}$, which is the set of 2-planes in the
4-dimensional Euclidean space spanned by ${\tilde n}_0,{\tilde n}_1,{\tilde n}_4,{\tilde n}_5$ given in the
preceding two
paragraphs, where $\lambda$ is the 2-plane spanned by ${\tilde n}_0,{\tilde n}_1$; in fact, this 4-dimensional linear space is characterized by that the shape operators $S_n$ of 
all unit $n$ in it share a common kernel (the Condition A for them). 
However, any two ${\mathcal Q}_2$ in ${\mathcal C}$ of dimension at most 3 in ${\mathcal Q}_4$ will intersect in
at least $2+2-3=1$ dimensional worth of points by a dimension count, so that
each of these points of intersection is contained in more than one
${\mathcal Q}_2$ in ${\mathcal C}$. This is a contradiction.
The contradiction implies that the codimension 2 estimate is true for at least one,
and hence, for generic $\lambda\in{\mathcal C}$.
\end{proof}

From now on, we assume that the isoparametric hypersurface is {\em not} the one constructed by Ozeki and Takeuchi,
and, by Lemma~\ref{la}, away from points of Condition A on $M_{+}$,  that $p_0,\cdots,p_5$ form a regular
sequence and $p_0=\cdots=p_5=0$ carves out an irreducible variety $V_5$. It follows
that $p_0,\cdots,p_6$ form a regular sequence for any choice of $p_6$~\cite[Corollary 1, p. 6]{Chiq}.
By~\eqref{sup}, we also have

$$
\sup_{\lambda\in{\mathcal Q}_5} r_{\lambda}\geq 5.
$$
We know the codiemsnion 2 estimate~\eqref{In} can no longer go through
for $k=6$; or else $p_0,\cdots,p_7$ would be a regular sequence and
the isoparametric hypersurface would be the one constructed by Ozeki and Takeuchi~\cite[Proposition 4, p. 11]{Chiq}.
Let us understand how and where the codimension 2 estimate fails in this case.

For $k=6$, when $(m_{+},m_{-})=(7,8)$, we record that the {\em a priori} codimension 2
estimate~\eqref{eest} is now

\begin{equation}\label{est1}
8=m_{-}\geq 13-j-c_j.
\end{equation}
So clearly it holds when $j\geq 4$ since $c_j\geq 1$ for $j\leq 4$.

For $j=3$, the codimension 2 estimate goes through as well as long as
$c_j\geq 2$. So in the following we assume $c_j=1$. We claim that the
condition in Lemma~\ref{3null} is satisfied
so that Lemma~\ref{sharper} allows us to employ the
refined codimension 2
estimate~\eqref{In}, which is now,

\begin{equation}\label{est2}
8=m_{-}\geq 12-j-c_j,
\end{equation}
to conclude that the codimension 2 estimate goes through with $j=3$ and
$c_j=1$.
To prove the claim, it suffices to establish the following Lemma.

\begin{lemma}\label{LA} Let ${\mathcal C}$ be
an irreducible component of ${\mathcal L}_j$. Suppose ${\mathcal C}$ is of codimension $1$ in ${\mathcal Q}_5$ {\rm (}i.e., $c_j=1${\rm)}. Then there is
a $\lambda\in{\mathcal C}$, which is the $2$-plane spanned by ${\tilde n}_0,{\tilde n}_1$,
such that there is an ${\tilde n}_2$ perpendicular to ${\tilde n}_0,{\tilde n}_1$ for which $B_{\tilde 2}$
is of rank at least $5$.
\end{lemma}

\begin{proof} Let $S^6$ be the unit sphere in the linear space spanned by $n_0,\cdots,n_6$. Consider the incidence space
$$
{\mathcal I}=\{({\tilde n},\lambda)\in S^6\times {\mathcal C}:\; {\tilde n}\perp {\tilde n}_0,{\tilde n}_1;\;
\lambda=\text{span}({\tilde n}_0,{\tilde n}_1)\}
$$
with the projection $\pi_1$ and $\pi_2$ onto the first and second factors, respectively.
${\mathcal I}$ is (real) 12-dimensional because for each
$\lambda=\text{span}({\tilde n}_0,{\tilde n}_1),$ the set $\pi_2^{-1}(\lambda)$ is the unit
4-sphere in the span of ${\tilde n}_2,\cdots,{\tilde n}_6$ perpendicular to ${\tilde n}_0,{\tilde n}_1$.

We show that $\pi_1$ is surjective. For each ${\tilde n}$ in the image of $\pi_1$, the set
$\pi_1^{-1}({\tilde n})$ consists of all
$({\tilde n},\lambda),\lambda=\text{span}({\tilde n}_0,{\tilde n}_1)\in{\mathcal C},$
such that ${\tilde n}\perp {\tilde n}_0,{\tilde n}_1$; therefore, $\pi_1^{-1}({\tilde n})$ is the
intersection of ${\mathcal C}$
and the variety ${\mathcal G}\simeq {\mathcal Q}_4$ of oriented
2-planes in ${\tilde n}^{\perp}\simeq{\mathbb R}^6$ with ${\tilde n}$ in the span of $n_0,\cdots,n_6$
and so $\pi_1^{-1}({\tilde n})={\mathcal G}\cap{\mathcal C}$ is (complex) 3-dimensional.
As a result, $\pi_1({\mathcal I})$ is (real) 6-dimensional contained in $S^6$ and so
$\pi_1$ is surjective.

We can now pick a generic ${\tilde n}\in S^6$ whose associated ${\mathcal G}\cap{\mathcal C}$ recovers ${\tilde n}_0,{\tilde n}_1$ and designate this ${\tilde n}$ to be ${\tilde n}_2$.
Then $B_{\tilde 2}$ of $S_{{\tilde n}_2}$ assumes generic rank $\geq 5$.
\end{proof}

In view of the preceding lemma, if there is a $\lambda\in{\mathcal L}_3$ whose spectral data satisfy the condition in item (1) of Lemma~\ref{3null}, then the codimension 2 estimate goes through since the normal basis cannot be 3-null. 

Otherwise, the spectral data of all $\lambda\in{\mathcal L}_3$ satisfy the condition in item (2) of Lemma~\ref{3null}.
Now, pick a generic point
$\lambda\in{\mathcal C}$ spanned by $n_{\tilde{0}},n_{\tilde{1}}$.
Let $S_{\tilde{0}}$ and $S_{\tilde{1}}$ be normalized
as in~\eqref{BC} and~\eqref{A} and extend them to
$S_{\tilde{0}}\cdots,S_{\tilde{6}}$. Consider the $S^5\subset{\mathcal Q}_5$ given by
$[1:\lambda_1:\cdots:\lambda_6]$, where $\lambda_1,\cdots,\lambda_6$ are
purely imaginary. Note that $\lambda=[1:\sqrt{-1}:0:\cdots:0]$ in
$S^5\cap{\mathcal C}$. Now,
\begin{equation}\label{dim}
\dim(S^5\cap{\mathcal C})\geq 5+8-10=3,
\end{equation}
where 10 is the real dimension of ${\mathcal Q}_5$.

 This dimension estimate implies that the closure $\Lambda$ of the irreducible component of  $S^5\cap{\mathcal C}$ containing $\lambda$ coincides with the unit 3-sphere
of the span of $\tilde{n}_1,{\tilde n}_4,\tilde{n}_5,\tilde{n}_6$. This is because by the concluding paragraph of Remark~\ref{importantremark}, the closure of
the irreducible component of $S^5\cap{\mathcal C}$ containing $\tilde{n}_1$ is a sphere whose generic $\tilde{n}$ 
is $3$-null. Thus,~\eqref{dim} implies that 
there are at least three such independent $\tilde{n}$, so that there are exactly three such independent $\tilde{n}$, namely, $\tilde{n}_4,\tilde{n}_5,{\tilde n}_6$
for $\tilde{n}_1,\tilde{n}_4,\tilde{n}_5,\tilde{n}_6$ to bound a 3-sphere, because  $\tilde{n}_2$ is not 3-null since otherwise by item (2) of Lemma~\ref{3null} the rank of $B_{\tilde{2}}$ would be 3, contradicting its being $\geq 5$ as said in Lemma~\ref{LA}, and, consequently, 
$\tilde{n}_3$ is not 3-null either by virtue of~\eqref{dep}. But then item (2) of Lemma~\ref{3null} implies that all linear combinations of $B_{\tilde{4}},B_{\tilde{5}},$ and $B_{\tilde{6}}$ are of the form in~\eqref{e1} with the $b$-block zero. It follows by item (2) of Lemma~\ref{3null} that a generic point of the quadric
${\mathcal Q}_3$, defined to be the set 
of 2-planes in the span of $\tilde{n}_0,\tilde{n}_1,\tilde{n}_4,\tilde{n}_5,\tilde{n}_6$, is contained in ${\mathcal C}$, and moreover, 
this ${\mathcal Q}_3$ is the unique 3-quadric containing $\lambda$ in the closure of ${\mathcal C}$ (because $\Lambda=S^3$).

But then, we can take a generic combination of $B_{\tilde 2},\cdots,B_{\tilde 6}$ , which is of rank 5, and call it 
$B_{2'}$ with normal direction $n_2'$. We then go through the same argument as above to conclude that we can come up with normal vectors $n_4',n_5',n_6'$ such that ${\tilde n}_0,{\tilde n}_1,n_4',n_5',n_6'$ generate a ${\mathcal Q}_3^{'}$ contained in the closure of ${\mathcal C}$ different from the above ${\mathcal Q}_3$, both containing $\lambda$. This contradicts the uniqueness of such ${\mathcal Q}_3$.


 

For $j=2$, Lemma~\ref{LA} enforces item (1) of Lemma~\ref{2null},
so that Lemma~\ref{sharper} allows us to 
warrant the validity of~\eqref{est2},
where the right hand side is $\leq 8$; with $c_j\geq 2$ the codimension 2 estimate
holds. Henceforth, we assume $c_j=1$ and so ${\mathcal C}\subset {\mathcal Q}_5$ given in Lemma~\ref{LA} is of (complex) dimension 4. The right hand side of~\eqref{est2} is 9;
we need to cut down one more dimension for the codimension 2 estimate to go
through. We spell out more details.

For $\lambda\in{\mathcal L}_2$, $p_{\tilde{0}}=p_{\tilde{1}}=0$
cuts ${\mathscr S}_\lambda$ in the variety (see Lemma~\ref{sharper})
$$
\{(X_1,X_2,Y_1,Y_2,Z_1,Z_2)\},
$$
where $X_1=(x_1,\cdots,x_{6})$ satisfies 

\begin{equation}\label{I'}
\sum_{\alpha=1}^{6} x_\alpha^2=0,
\end{equation}
$X_1=\pm\sqrt{-1}Y_1$, $X_2=-Y_2$ and $Z_2$ depends linearly on $X_2$.
Moreover, for $2\leq l\leq 6$,

\begin{equation}\label{II'}
p_{l^*}=\sum_{\alpha=1, p=1}^{6,5}
(S^l_{\alpha p}+\pm\sqrt{-1} T^l_{\alpha p})x_\alpha z_p +\; {\rm other\; terms}.
\end{equation}
We may assume the displayed sum is nontrivial for $l=2$ since 2-nullity is impossible by item (1) of Lemma~\ref{2null}.
~\eqref{I'} and~\eqref{II'} imply that $p_{\tilde{0}}=p_{\tilde{2}}=0$
cuts down one more dimension in ${\mathscr S}_\lambda$ to carve out an irreducible variety
${\mathcal F}_2$ by Lemma~\ref{import2} in Appendix I, so that the lower
bound in~\eqref{est2}, which is now 9, is achieved.

To cut down one more dimension to reach 8 on the right hand side of~\eqref{est2}, observe that if ${\mathcal F}_j$, the
irreducible variety of ${\mathscr S}_\lambda$ cut out by
$p_{\tilde{0}}=p_{\tilde{j}}=0,3\leq j\leq 6,$ is distinct from ${\mathcal F}_2$,
then one more dimension cut can be achieved by Lemma~\ref{import2} in Appendix I, so
that the codimension 2 estimate holds.

So now, we must rule out the case that all ${\mathcal F}_j,3\leq j\leq 6,$ are identical with
${\mathcal F}_2$. Suppose they were all identical. It would then follow by a similar argument as in~\eqref{dep} through~\eqref{combo} in Lemma~\ref{la} that
the displayed part of $p_{\tilde{4}},p_{\tilde{5}},p_{\tilde{6}}$ in~\eqref{II'} are all zero.
We could then employ the same arguments immediately following Lemma~\ref{LA} as for $j=3$, with obvious modifications while invoking item (2) of Lemma~\ref{2null}, to reach a contradiction. Thus, generic $\lambda\in{\mathcal C}$
satisfies the codimension 2 estimate.

For $j=1$, Lemma~\ref{1null}
allows us to apply Lemma~\ref{sharper} to obtain~\eqref{est2},
whose right hand side is
10 obtained by setting $p_{\tilde{0}}=p_{\tilde{2}}=0$ as usual.

Now, not all $p_{\tilde{j}},j\geq 3$ are multiples of $p_{\tilde{2}}$ when restricted to
${\mathscr S}_\lambda$; for otherwise,
we can argue exactly as in~\eqref{S},~\eqref{xyz} and~\eqref{combo} to obtain $p_{\tilde{6}}=0$ when restricted to ${\mathscr S}_\lambda$
so that the basis element $\tilde{n}_6$ is 1-null, which is impossible by Lemma~\ref{1null}.
So we may assume $p_{\tilde{2}}$ and $p_{\tilde{3}}$ are linearly independent when
restricted to
${\mathscr S}_\lambda.$ Then employing the same arguments one more time
we can conclude that we may assume $p_{\tilde{2}},p_{\tilde{3}},p_{\tilde{4}}$ are linearly
independent when restricted to ${\mathscr S}_\lambda$. Lemma~\ref{import1}
in Appendix I then enables us to further cut down 2 more dimensions from the right hand side
of~\eqref{est2}, so that the codimension 2 estimate holds.

Lastly, for $j=0$, no bases being 0-null lets us utilize~\eqref{est2}
whose right hand side is 11. We may assume
$p_{{\tilde 2}},p_{{\tilde 3}},p_{{\tilde 4}}$ (understood to be restricted to
${\mathscr S}_\lambda$ in the following) are
independent to be in accordance with item (2) of Corollary~\ref{inde}. For otherwise, a nontrivial linear combination of each of the triples  $(\tilde{n}_2,\tilde{n}_3,\tilde{n}_4)$,  $({\tilde n}_2,{\tilde n}_ 3, {\tilde n}_5)$, and $({\tilde n}_2,{\tilde n}_3,{\tilde n}_6)$ would result in three independent normal directions $n_1',n_2',n_3'$ for which the $B$ and $C$ blocks of the corresponding shape operators $S_{n_1'},S_{n_2'},S_{n_3'}$ are zero to violate Corollary~\ref{inde}.
 If $p_{{\tilde 5}}$ and $p_{{\tilde 6}}$ (understood to be restricted to
${\mathscr S}_\lambda$
in the following)
are both dependent on $p_{{\tilde 2}},p_{{\tilde 3}},p_{{\tilde 4}}$, then as before after a basis change
we may assume $p_{{\tilde 5}}$ and $p_{{\tilde 6}}$ are zero. However, this implies that,
as in the ending arguments in Lemma~\ref{la},
through $\lambda$ there is a unique ${\mathcal Q}_2$ in the irreducible component
${\mathcal C}$ of ${\mathcal L}_0$ where $\lambda$ belongs. Since
$\dim({\mathcal C})\leq 4$, we see as before that any two such quadrics
have nonempty intersection in ${\mathcal C}$, a contradiction. Hence, we may assume
that $p_{{\tilde 2}},\cdots,p_{{\tilde 5}}$ are linearly independent. Lemma~\ref{import} in Appendix I implies that
$p_{{\tilde 3}}=p_{{\tilde 4}}=p_{{\tilde 5}}=0$ now cuts down three more dimensions from the right hand side
of~\eqref{est2}. That is, the codimension 2 estimate goes through.

It follows that
the codimension 2 estimate holds for $k=6$ if the generic rank of
$r_\lambda\geq 5$ for $\lambda\in{\mathcal Q}_5$; 
the isoparametric hypersurface is thus the one constructed by Ozeki and
Takeuchi. This is impossible. So, we conclude the following.

\begin{lemma}\label{r} Let $(m_{+},m_{-})=(7,8)$. Suppose the isoparametric
hypersurface is not the one constructed by Ozeki and Takeuchi. Away from points of Condition A on $M_{+}$, given an orthonormal pair $(n_0,n_1)$ of
normal vectors of $M_{+}$, let $S_{n_0}$ and $S_{n_1}$ be normalized as
in~\eqref{0a},~\eqref{BC} and~\eqref{A}. 
Then the rank of the $B_1$ {\rm (}and $C_1${\rm )} of $S_{n_1}$ is $\leq 4$ for any choice of $n_0$. 
\end{lemma}

\begin{proof} Away from points of Condition A, suppose there is a unit normal pair $(n_0,n_1)$ of $M_{+}$ for
which the $B_1$ of $S_{n_1}$ is of rank $\geq 5$. Extend $n_0,n_1$ to an orthonormal
basis $n_0,\cdots,n_7$. The analysis in Lemma~\ref{la} and what follows it shows that
the isoparametric hypersurface is the one constructed by Ozeki and Takeuchi, which is a contradiction.
We conclude that the rank of $B_1$ is $\leq 4$ for any choice of $n_0$.
\end{proof}

Note that by Corollary~\ref{COR}, a generic normal basis is respectively
4-null, 3-null, or 2-null if the generic rank is 4, 3, or 2.

We will in fact establish that the generic rank is 4 in the next section in Corollary~\ref{CA}.

\section{Mirror points}\label{mirror} Let $x_0\in M_{+}$ and let $n_0,n_a, a=1,\cdots,m_{+},$ be a normal basis of $M_{+}$ at $x_0$.
Set
$$
x_0^\#:=n_0,\quad n_0^\#:=x_0.
$$
Of fundamental importance is that $x_0^\#$ is also a point on $M_{+}$ with the normal space 
${\mathbb R}n_0^\#\oplus E_0,$ 
where $E_0$ is the 0-eigenspace of the shape operator $S_{n_0}$ at $x_0$, whose basis vectors are denoted by
$e_p,p=1,\cdots,m_{+}$.
The 0-eigenspace of the shape operator $S_{n_0^\#}$ at $x_0^\#$ is spanned by $n_a,a=1,\cdots,m_{+}$. Moreover, $S_{n_0}$ at $x_0$ and 
$S_{n_0^\#}$ at $x_0^\#$ share the same $(+1)-$ and $(-1)$-eigenspaces 
$E_{+}$ and $E_{-}$, whose basis vectors are denoted by $e_\alpha$ and $e_\mu,1\leq \alpha,\mu\leq m_{-}$, respectively.  

Referring to~\eqref{0a}, where 
\begin{equation}\label{gOOd}
A_a:=\begin{pmatrix}S^a_{\alpha\mu}\end{pmatrix},\quad B_a:=\begin{pmatrix}S^a_{\alpha p}\end{pmatrix},\quad C_a:=\begin{pmatrix}S^a_{\mu p}\end{pmatrix}.
\end{equation}
Let the counterpart matrices at $x_0^{\#}$ and their blocks
be denoted by the same notation with an additional \#.
Then, for $p=1,\cdots,m_{+},$

\begin{equation}\label{gooD}
A_p^\#:=\begin{pmatrix}S^p_{\alpha\mu}\end{pmatrix},\quad B_p^{\#}=\begin{pmatrix}S^a_{\alpha p}\end{pmatrix},
\quad C_p^{\#}=-\begin{pmatrix}S^a_{\mu p}\end{pmatrix}.
\end{equation}

We call $x_0^{\#}$ the ``mirror point'' of $x_0$ on $M_{+}$.

Similarly, set
\begin{equation}\label{*}
x_0^*:=(x_0+n_0)/\sqrt{2},\quad n_0^*:=(x_0-n_0)/\sqrt{2}.
\end{equation}
$x_0^*$ is a point on $M_{-}$ and $n_0^*$ is normal to $M_{-}$ at $x_0^*$. The normal 
space to $M_{-}$ at $x_0^*$ is
${\mathbb R}n_0^*\oplus E_{+}.$ 
Furthermore, the $(+1)$-eigenspace $E_{+}^*$
of the shape operator $S_{n_0^*}$ is spanned by $n_1,\cdots,n_{m_1}$, the $(-1)$-eigenspace $E_{-}^*$ of $S_{n_0^*}$ is $E_0$, and
the $0$-eigenspace $E_0^*$ of $S_{n_0^*}$ is $E_{-}$.

Referring to~\eqref{0a}, let the counterpart matrices at $x_0^*$ and their blocks
be denoted by the same notation with an additional *.
Then, for $\alpha=1,\cdots,m_{-},$

\begin{equation}\label{Good}
A_\alpha^*=-\sqrt{2}\begin{pmatrix}S^a_{\alpha p}\end{pmatrix},\quad
B_\alpha^*=-1/\sqrt{2}\begin{pmatrix}S^a_{\alpha\mu}\end{pmatrix},\quad
C_\alpha^*=-1/\sqrt{2}\begin{pmatrix}S^p_{\alpha\mu}\end{pmatrix}.
\end{equation}
(Likewise, there are counterpart matrices when we replace $\alpha$ by $\mu$ at the points $(x_0^*)^\#\in M_{-}$.)

We call $x_0^*$ the ``mirror point'' of $x_0$ on $M_{-}$. See~\cite[p. 144]{Chi},~\cite[p. 474]{Chiq4} for more details.

\begin{corollary}\label{CA} Notation as above, we may assume
\begin{eqnarray}\label{good}
\aligned
&A_\alpha^*=\begin{pmatrix} 0&0\\0&\cdot\end{pmatrix},\quad B_\alpha^*=\begin{pmatrix} \cdot&0\\0&\cdot\end{pmatrix},\quad
C_\alpha^*=\begin{pmatrix} \cdot&0\\0&\cdot\end{pmatrix},\quad 1\leq \alpha\leq 4;\\
&A_\alpha^*=\begin{pmatrix} 0&\cdot\\\cdot&\cdot\end{pmatrix},\quad B_\alpha^*=\begin{pmatrix} 0&\cdot\\\cdot&\cdot\end{pmatrix},\quad
C_\alpha^*=\begin{pmatrix} 0&\cdot\\\cdot&\cdot\end{pmatrix}, \quad 5\leq\alpha\leq 8.
\endaligned
\end{eqnarray}
In particular, Lemma~$\ref{r}$ can be improved to $4$-nullity.
\end{corollary}

\begin{proof} By Lemma~\ref{r}, a generic choice of $x$ and $x^\#$ can only be $r$-null for $1\leq r\leq 4$, so that the upper left $(8-r)$-by-$(7-r)$ 
block of $B_p^\#$ and $C_p^\#$ are zero for $1\leq p\leq 7$. That is,
\begin{equation}\label{SA}
S^a_{\alpha p}=S^a_{\mu p}=0, \quad 1\leq\alpha,\mu\leq 8-r,1\leq a\leq 7-r,\forall p=1,\cdots,7.
\end{equation}
In other words,
\begin{equation}\label{Ba}
B_a=\begin{pmatrix}0&0\\\beta_a&\gamma_a\end{pmatrix},\quad 1\leq a\leq 7-r,
\end{equation}
where the columns are indexed by $p$ and the upper left block is of size $(8-r)$-by-$(7-r)$. (Likewise, $C_a$ are of the same form.)

We normalize $A_1$ and $B_1$ as in~\eqref{BC} and~\eqref{A}. 
The proof of Corollary~\ref{inde} implies that 

\begin{equation}\label{Aa}
A_a=\begin{pmatrix} z_a&0\\0&\cdot\end{pmatrix},\quad 2\leq a\leq 7-r,
\end{equation}
where the upper left block is of size $(8-r)$-by-$(8-r)$ with 

\begin{equation}\label{pa}
z_a=-z_a^{tr},\quad z_az_b+z_bz_a=-2\delta_{ab}I,\quad 2\leq a,b\leq 7-r.
\end{equation} 
That is, we have a Clifford $C_{6-r}$-module ${\mathbb R}^{8-r}$ for $1\leq r\leq 4$; this is possible only when $r=4$. In particular, generic points
of $M_{+}$ are 4-null.  

With $r=4$ in place, note that, by~\eqref{gOOd} and~\eqref{Good},~\eqref{Aa} is equivalent to

$$
B_\alpha^*=\begin{pmatrix} h_\alpha&0\\k_\alpha&\cdot\end{pmatrix},\;\;\alpha\leq 4; \quad B_\alpha^*=\begin{pmatrix} 0&\cdot\\\cdot&\cdot\end{pmatrix},\;\;5\leq\alpha\leq 8
$$
for some $h_\alpha,k_\alpha$. Now the 4-nullity at $x$ is 

\begin{equation}\label{ba}
B_a=\begin{pmatrix} 0&\cdot\\\cdot&\cdot\end{pmatrix}, \quad \forall a=1,\cdots,7.
\end{equation}
That is,

\begin{equation}\label{Sa}
S^a_{\alpha p}=0,\quad 1\leq\alpha\leq 4,\quad 1\leq p\leq 3,\quad \forall a=1,\cdots,7.
\end{equation}
Putting~\eqref{SA} and~\eqref{Sa} together, we obtain
$$
A_\alpha^*=\begin{pmatrix} 0&0\\0&\cdot\end{pmatrix},\quad 1\leq \alpha\leq 4.
$$


That the upper left corner of $A^*_\alpha$ is zero for $\alpha\geq 5$ is equivalent to that the lower left block of $B_a$ in~\eqref{Ba} 
is zero for $1\leq a\leq 3.$ 
To show the latter, item (1) of Corollary~\ref{inde} implies that there is a matrix $B_j$, for some $j\geq 4,$ of the form

\begin{equation}\label{bj}
B_j=\begin{pmatrix}0&d\\b&c\end{pmatrix},\quad d_{4\times 4}\neq 0.
\end{equation}
Consider 
$$
E:=uB_1+vB_2+wB_j=\begin{pmatrix}0&wd\\v\beta+wb&\sigma+v\gamma+wc\end{pmatrix},\quad u^2+v^2+w^2=1,
$$
where we suppress the index 2 for $B_2$ in~\eqref{Ba}. $E$ is of rank at most 4, and is of rank 4 for $u$ close to 1, so that the equation
$$
\begin{pmatrix}0&wd\\v\beta+wb&u\sigma+v\gamma+wc\end{pmatrix}\begin{pmatrix}x\\y\end{pmatrix}=0,
$$
is of dimension 3 for $u$ close to 1. This amounts to 
$$
wdy=0,\quad (v\beta+wb)x+(u\sigma+v\gamma+wc)y=0.
$$
Since $\sigma+v\gamma+wc$ is invertible for $u$ close to 1, we can solve $y$ in terms the 3-dimensional $x$ and insert it into $dy=0$ 
(for small $w\neq 0$) to yield
$$
d(u\sigma+v\gamma+wc)^{-1}(v\beta+wb)=0,
$$
whose Taylor expansion reads

\begin{equation}\label{taylor}
d(I-(v'\sigma^{-1}\gamma+w'\sigma^{-1}c)+(v\sigma^{-1}\gamma+w\sigma^{-1}c)^2-\cdots)\sigma^{-1}(v'\beta+w'b)=0,
\end{equation}
where $v'=v/u$ and $w'=w/u$, from which we can extract

\begin{equation}\label{dsig}
d\sigma^{-1}\beta=0.
\end{equation}
That is, the column space of $\sigma^{-1}\beta$ is in the kernel of $d$. We thus conclude that

\begin{equation}\label{col}
\text{the column space of}\; \sigma^{-1}\beta\subset\cap_{j=i}^7\;\text{kernel}(d_j),  
\end{equation}
where 
$$
B_j:=\begin{pmatrix} 0&d_j\\b_j&c_j\end{pmatrix}.
$$

We claim that $\cap_{j=i}^7\;\text{kernel}(d_j)$ is of dimension at most 1. To this end, suppose the intersection is of dimension $l$. Reparametrizing, we 
may assume the first $l$ columns of $d_j$ are zero for all $j=1,\cdots,7$, which amounts to 
$$
S^a_{\alpha p}=0,\quad 1\leq\alpha\leq 4,\quad 4\leq p\leq 3+l,\quad \forall a=1,\cdots, 7.
$$ 
This is equivalent to 
$$
B_p^\#=\begin{pmatrix} 0&0\\\cdot&\cdot\end{pmatrix},\quad p=4,\cdots,3+l,
$$ 
where the 0 rows are of size 4-by-7. On the other hand,~\eqref{ba} is equivalent to
$$
B_p^\#=\begin{pmatrix} 0&0\\\cdot&\cdot\end{pmatrix},\quad p=1,2,3.
$$
Therefore, normalizing $B_1^\#$ as in~\eqref{BC}, we have that the top four rows of $B_j,2\leq j\leq 3+l,$ are zero. But then Corollary~\ref{inde} implies that
$l\leq 1$, because only Clifford $C_2$, when $l=0$, and Clifford $C_3$, when $l=1$, can act on ${\mathbb R}^4$. This proves the claim.
 
When $l=0$, we have $\beta=0$ by~\eqref{col}, i.e., the lower left block of $B_a$ in~\eqref{Ba} 
is zero for $1\leq a\leq 3.$ 

We can thus assume that generically $l=1$ over $M_{+}$. This is equivalent to saying, by considering generic $x$ and $x^\#$, that
there is an index $a\geq 4$, say, $a=4$, and an index $p\geq 4$, say, $p=4$, such that 
$$
S^{a=4}_{\alpha\, p}=S^a_{\alpha\, p=4}=0,\quad 1\leq \alpha\leq 4,\quad \forall a, p=1,\cdots,7.
$$   
That is, for each $\alpha\leq 4,$ the first four columns and rows of the 7-by-7 matrix $A_\alpha^*$ in~\eqref{good} are zero, i.e.,

\begin{equation}\label{aa}
A_\alpha^*=\begin{pmatrix}0&0\\0&\delta_\alpha\end{pmatrix}, \quad 1\leq\alpha\leq 4,
\end{equation}
where $\delta_\alpha$ is of size 3-by-3. 

Note that in~\eqref{pa} we may assume
that $z_1, z_2$ and $z_3$ are respectively the matrix representation of the quaternionic multiplication of the basis elements 
${\bf i},{\bf j}$ and ${\bf k}$ on the left of ${\mathbb H}$; in doing so we do not assume $z_1=I$ so that the representation will be notationally more consistent, and it will not 
affect the subsequent arguments. Accordingly ,we have
$$
z_1=\begin{pmatrix}0&-1&0&0\\1&0&0&0\\0&0&0&-1\\0&0&1&0\end{pmatrix},\quad z_2=\begin{pmatrix}0&0&-1&0\\0&0&0&1\\1&0&0&0\\0&-1&0&0\end{pmatrix},
\quad z_3=\begin{pmatrix}0&0&0&-1\\0&0&-1&0\\0&1&0&0\\1&0&0&0\end{pmatrix},
$$
according to which
\begin{eqnarray}\nonumber 
\aligned
&h_1=\begin{pmatrix}0&-1&0&0\\0&0&-1&0\\0&0&0&-1\end{pmatrix}/\sqrt{2},\quad h_2=\begin{pmatrix}1&0&0&0\\0&0&0&1\\0&0&-1&0\end{pmatrix}/\sqrt{2},\\
&h_3=\begin{pmatrix}0&0&0&-1\\1&0&0&0\\0&1&0&0\end{pmatrix}/\sqrt{2},\quad h_4=\begin{pmatrix}0&0&1&0\\0&-1&0&0\\1&0&0&0\end{pmatrix}/\sqrt{2}.
\endaligned
\end{eqnarray}

Moreover, we have (in $B_\alpha^*$)

\begin{equation}\nonumber
h_\alpha k_\alpha^{tr}=0, \quad h_\alpha h_\alpha^{tr}=I/2,
\quad \alpha\leq 4,
\end{equation}
by the first identity of~\eqref{conv} when we set $i=j=\alpha$, 
where $h_\alpha$ is of size $3$-by-$4$ and $k_\alpha$ is of size $4$-by-$4$, 
from which we read off that the only possibly nonzero column of $k_\alpha$ is the $\alpha$th one, i.e.,
$$
k_\alpha =\begin{pmatrix} \epsilon^\alpha_{jk}\delta_{k\alpha}\end{pmatrix},\quad 1\leq \alpha,j,k\leq 4.
$$ 
Now the the first identity of~\eqref{conv} applied to $1\leq\alpha\neq\beta\leq 4$ gives
$$
h_\alpha k_\beta^{tr}+h_\beta k_\alpha^{tr}=0,
$$
which implies the four possibly nonzero columns are all identical, i.e,

\begin{equation}\label{identical}
\epsilon^1_{j1}=\epsilon^2_{j2}=\epsilon^3_{j3}=\epsilon^4_{j4}, \quad 1\leq j\leq 4.
\end{equation}
By performing a coordinate change on the $a$-indexes, $4\leq a\leq 7$, indexing the rows of $B_a^*$, we may assume only the first components of these four columns are possibly nonzero, i.e.,

\begin{equation}\label{only}
\epsilon^1_{j1}=\epsilon^2_{j2}=\epsilon^3_{j3}=\epsilon^4_{j4}=0, \quad 2\leq j\leq 4.
\end{equation}
The same holds for $C_\alpha^*,\alpha\leq 4,$ as well by changing the $p$-indexes, $4\leq p\leq 8$. In fact, the sixth identity of~\eqref{conv} 
implies that we may further assume
the nonzero entries of these columns for both $B_\alpha^*$ and $C_\alpha^*$ are identical. 

Now, by the first identity of~\eqref{conv} with $i=j=\alpha\leq 4$, we derive
\begin{equation}\label{Balpha}
B_\alpha^*(B_\alpha^*)^{tr}= \begin{pmatrix}I/2&0\\0&D_\alpha\end{pmatrix},\quad D_\alpha=\begin{pmatrix} 1/2&0\\0&e_\alpha\end{pmatrix},
\end{equation}
where $e_\alpha$ is of size 3-by-3, in light of~\eqref{aa}. Thus we may rearrange indexes 
(see~\cite[Lemma 49, p. 64]{CCJ}) to assume 

\begin{equation}\label{ABC}
A_\alpha^*=\begin{pmatrix}0&0\\0&\delta_\alpha\end{pmatrix},\quad B_1^*=C_1^*=\begin{pmatrix}0&I/\sqrt{2}&0\\0&0&\sqrt{D}\end{pmatrix}, \quad \alpha\leq 4,
\end{equation}
where $\sqrt{D}$ is diagonal of the form 

\begin{equation}\label{D}
\sqrt{D}=\text{diag}(1/\sqrt{2},1/\sqrt{2},b,b),
\end{equation}
given the spectral data $(\sigma,\Delta)$ since $\delta_\alpha$ is of size 3-by-3, where $I$ is 3-by-3.

Suppose $\sqrt{D}$ is nonsingular. $\delta_1$ is skew-symmetric as it is part of $\Delta$. But then because
nonsingularity of $D$ is a generic condition, it follows that each linear combination of $\delta_\alpha,\alpha\leq 4,$ is
skew-symmetric of size 3-by-3 when suitably normalized, 
which implies 
that generic linear combinations of $\delta_\alpha$ are of rank 2, from which we see, for a generic point $c:=(c_1,\cdots,c_4) \in S^3$,
$$
\delta_c:=c_1\delta_1+\cdots+c_4\delta_4,
$$ 
that there is a unique $c'$ on $S^2$ which is the eigen direction of $\delta_c$ with eigenvalue 0.

Without loss of generality,
let us assume the map 
$$
F: S^3\rightarrow S^2,\quad c\rightarrow c'
$$ 
is surjective (more precisely, the domain and target spaces of  $F$ are projective spaces, though this does not create a problem); if $F$ is not surjective the preimage will be of even larger dimension to our advantage. Then the closure ${\mathcal C}$ of the preimage $F^{-1}(c')$ is a 1-dimensional circle,
because 
for $c\in F^{-1}(c')$, 
the plane perpendicular to $c'$, which is an eigenspace of $\delta_c(\delta_c)^{tr}$, is fixed, from 
which we conclude that there is a unique point $c_0$ on ${\mathcal C}$ for which $\delta_{c_0}=0$, because the spectral data stipulate that all $\delta_c$
for $c\in F^{-1}(c')$ be of the same form as $\delta_4$ below.
This means that we have an $S^2$ worth of $\delta_{c_0}$, one for each $c'$, which are
identically zero, so that we may assume 
$$
\delta_1=\delta_2=\delta_3=0, \quad \delta_4=\begin{pmatrix}0&0&0\\0&0&\tau\\0&-\tau&0\end{pmatrix}
$$
for some $0<\tau\leq 1/\sqrt{2}$. But then this implies that the first five columns and rows of $A_\alpha,\alpha\leq 4,$ are zero, which contradicts $l\leq 1$.

On the other hand, suppose generic $D_\alpha$ is singular, then
$$
D_\alpha\sim\text{diag}(1/2,1/2,1/2,0),\quad \text{or}\;\; \text{diag}(1/2,1/2,0,0).
$$
If it is the former case, then $\delta_c$ has a 2-dimensional eigenspace with eigenvalue 0. 
Let us denote by $c'$ the direction that is perpendicular
to the 2-dimensional 0-eigenspace of $\delta_c$;
the spectral data stipulate that $\delta_c$ be of the form
$$
\delta_c=\begin{pmatrix}0&0&0\\0&0&0\\0&0&x\end{pmatrix},\quad x^2=1.
$$
We are done by the same reasoning as in the nonsingular case.
If it is the latter, then $e_c$, whose components are given in the second matrix in~\eqref{Balpha}, serves the role of $\delta_c$ in the former case,
from which we conclude that there are $e_c=0$, contradicting the given nonzero spectral data.

In conclusion, $l=0$ generically. That is, 
the lower left block of $B_a=0$ in~\eqref{Ba} for $1\leq a\leq 3$, or equivalently, the upper left corner of $A_\alpha^*=0$ for $\alpha\geq 5$, for a generic choice of $x$ and $x^\#$. 

We will show in Corollary~\ref{qJ} below that the lower left blocks of $B^*_\alpha$ (and $C^*_\alpha$), $\alpha\leq 4,$ are zero.

\end{proof}

\begin{remark}\label{rkk} Intrinsically, in the preceding corollary, let $N^*\simeq{\mathbb H}\subset E_{+}$ be the kernel of $B_1^{tr}$, let
$V_0^*\simeq{\mathbb H}\subset E_{-}^*$ be the kernel of $C_1^{tr}$, let $V_{-}^*\simeq\; \text{Im}({\mathbb H})\subset E_{-}^*$ be the kernel of $B_1$, 
and let 
$V_{+}^*\simeq\; \text{Im}({\mathbb H})\subset E_{+}^*$
be the kernel of $B_1^\#$. Then these four spaces parametrize the upper left blocks of the matrices in the corollary, where $N^*$ is parametrized by $1\leq \alpha\leq 4$,
$V_0^*$ by $1\leq \mu\leq 4$, $V_{+}^*$ by $1\leq a\leq 3$, and $V_{-}^*$ by $1\leq p\leq 3$. 

\end{remark}

\begin{corollary} Notation as in the preceding remark, let 

\begin{equation}\label{V}
V:=V_{+}^*\oplus V_{-}^*\oplus V_0^*\subset E_{+}^*\oplus E_{-}^*\oplus E_0^*:=E.
\end{equation}
Let $p_j^*|_V$ and $q_j^*|_V,0\leq j\leq m_{-}=8$ be the components of the second and third fundamental forms of $M_{-}$ at $x^*$ evaluated on $V$, where the indexes 
$1\leq j\leq 4$ range through $N^*$, and as always $j=0$ indexes the components corresponding to $n_0^*$. Then we have

\begin{eqnarray}\label{3}
\aligned
&p_j^*|_V=0,\quad j\geq 5,\\
&q_j^*|_V=0,\quad 0\leq j\leq 4.
\endaligned
\end{eqnarray}
\end{corollary}

\begin{proof} The first identity follows from the vanishing of the upper left blocks of the last three matrices in the statement of Corollary~\ref{CA}. 

The second follows from the normal covariant derivative of the second fundamental form $S^*$ at $x^*\in M_{-}$

\begin{equation}\label{cov}
\sum_k(S^*)^b_{ij;k}\,\omega^k=d(S^*)^b_{ij}-\sum_t(S^*)^b_{tj}\,\theta^t_i-\sum_t(S^*)^b_{it}\,\theta^t_j,
\end{equation}
where $(S^*)^b_{ij;k}$ are the components of $q_b^*$, we assume the normal connection is zero at the point of calculation, and $\omega^j$ and $\theta^j_i$ are the coframe and connection forms. 

We indicate 
one calculation for illustration. Let indexes $i,j\leq 3$ and $k\leq 4$ denote respectively those for $E_{+}^*,E_{-}^*$ and $E_0^*$. Then 
for $1\leq b\leq 4$, the right hand side of ~\eqref{cov}
is zero by the vanishing blocks of the first matrix in~\eqref{good}, knowing that $(S^*)^b_{uv}=0$ whenever
$u$ and $v$ index the same eigenspace and that $\theta^k_i$ and $\theta^k_j$ vanish on $E_0^*$
(see~\cite[(4.18), p. 14]{CCJ} for how to calculate $\theta^j_i$ in general).

On the other hand, the cubic polynomial
\begin{equation}\label{q*0}
q_0^*|_V=\sum_{p\leq 4, i,j\leq 3} (S^*)^p_{ij}\, z_p\, x_i\, y_j=0,
\end{equation}
where $p$ indexes the corresponding normal directions at $(x^*)^\#$, the mirror point of $x^*$ on $M_{-}$, and $i,j\leq 3$ index the $E_{+}^*$ and $E_{-}^*$, 
respectively. The vanishing of the identity
follows from that of the upper left block of the first matrix in~\eqref{good} when we replace $\alpha$ by $\mu$.
\end{proof} 
 
\begin{corollary} Let ${\bf 1},{\bf i},{\bf j},{\bf k}$ be the standard basis in ${\mathbb H}$. Write 
$$
v=x\oplus y\oplus z
$$ 
respecting the direct sum of $V$ in~\eqref{V}, and write
$$
p^*:= p_1^*|_V \,{\bf 1}+p_2^*|_V\, {\bf i}+p_3^*|_V\, {\bf j}+p_4^*|_V\, {\bf k}.
$$ 
Then
\begin{equation}\label{p}
p^*(v,v)=-\sqrt{2}(xz+y\circ z),
\end{equation}
where $y\circ z=yz$ or $zy$ (quaternion multiplication).

\end{corollary}

\begin{proof} This follows from~\eqref{pa} and the corresponding identity for the matrix $A_p^\#,1\leq p\leq 3.$ 
See \cite[Remark 1, p. 140, and Proposition 1, p. 146]{Chi} for more details.
\end{proof}

\begin{corollary}\label{qJ} $q^*_j|_V=0,\forall j.$ In particular, the lower left blocks of $B_\alpha^*$ and $C_\alpha^*,\alpha\leq 4,$ in~\eqref{good} are zero.
\end{corollary}

\begin{proof} By the identity~\cite[(3-8), p. 530]{OT}
$$
16(\sum_{a=0}^8 (q^*_a)^2)=16G(\sum_i u_i^2)-\langle \nabla G,\nabla G\rangle,
$$
where $G:=\sum_{a=0}^8 (p_a^*)^2$ and $u_i$ parametrize the tangential directions at $x^*$. A straightforward calculation by the first identity in~\eqref{3},~\eqref{identical},
and~\eqref{only} gives
\begin{equation}\label{16}
\aligned
16(\sum_{a=0}^8(q^*_a|_V)^2)&=16G|_V(|x|^2+|y|^2+|z|^2)-\langle\nabla (G|_V),\nabla (G|_V)\rangle\\
&-4c^2(\sum_{a=1}^4 p_a^*|_V\, z_a)^2,
\endaligned
\end{equation}
where $x,y,z$ are given in the preceding corollary, $c=(S^*)^a_{5a},1\leq a\leq 4$, and the factor 4 comes from the contribution of the $(5,a)$-entries, which are equal in value, of both 
$B_\alpha^*$ and $C_\alpha^*,\alpha\leq 4,$ in~\eqref{good} (see also~\eqref{identical} and~\eqref{only}).

In~\eqref{p}, if 
\begin{equation}\label{p^}
p^*(v,v)=-\sqrt{2}(xz+zy),
\end{equation}
then the sum of the first two terms on the right hand side of~\eqref{p} vanishes, because it is exactly equal to the normed square of the third fundamental
form of the {\bf homogeneous} isoparametric hypersurface with multiplicity pair $(3,4)$, which is zero. But then~\eqref{16} implies that
$c=0$ and $q_a|_V=0$ for all $0\leq a\leq 8.$ 

On the other hand, if
\begin{equation}\label{q*}
p^*(v,v)=-\sqrt{2}(xz+yz),
\end{equation}
then the sum of the first two terms on the right hand side of~\eqref{16} is
$$
|xy-yx|^2|z|^2,
$$
since it is the normed square of the third fundamental form of the {\bf inhomogeneous} isoparametric hypersurface with multiplicity pair $(3,4)$. Setting $x=y$ in~\eqref{q*},
we obtain once more that $c=0$, because $p^*(v,v)=-2\sqrt{2}xz$ makes the last term on the right hand side of~\eqref{16} nonzero if $c\neq 0$, which is impossible.

In particular,  
the lower left blocks of $B_\alpha^*$ and $C_\alpha^*,\alpha\leq 4,$ in~\eqref{good} are zero.

Now that
$$
16(\sum_{a=0}^8(q^*_a|_V)^2)=|xy-yx|^2|z|^2
$$
in the latter case, we see by the second identity of~\eqref{3} that
\begin{equation}\label{17}
16(\sum_{a=5}^8(q^*_a|_V)^2)=|xy-yx|^2|z|^2. 
\end{equation}
We will derive a contradiction. First, observe that~\eqref{17} implies that $q_a^*|_V,a\geq 5$, are all multilinear in $x,y,z$
and in fact after a coordinate change of $z$ we may assume 
\begin{equation}\label{eq100}
q^*_5|_V {\bf 1}+q^*_6|_V {\bf i} +q^*_7|_V {\bf j} +q^*_8|_V {\bf k}=(xy-yx)z.
\end{equation}
This is because setting $x=y$ in~\eqref{17}, we see each $q^*_a|_V,a\geq 5,$ is skew-symmetric in $x$ and $y$ and linear in $z$, so that $q^*_a|_V$
are of the form 
$$
q^*_a|_V = (x_2y_3-x_3y_2)\sum_b c^a_{1b}z_b+(x_3y_1-x_1y_3)\sum_b c^a_{2b}z_b+(x_1y_2-x_2y_1)\sum_b c^a_{3b}z_b,
$$
for $1\leq b\leq 4,5\leq a\leq 8,$ where
$$
xy-yx=(x_2y_3-x_3y_2){\bf i}+(x_3y_1-x_1y_3){\bf j}+(x_1y_2-x_2y_1){\bf k}.
$$
The right hand side of~\eqref{17} then asserts that the three 4-by-4 matrices $\begin{pmatrix}c^a_{ib}\end{pmatrix},1\leq i\leq 3,1\leq a,b\leq 4,$
form a Clifford system, and hence there follows~\eqref{eq100}.

So now, 
\begin{equation}\label{q*a}
\aligned
&q^*_a = \langle (xy-yx)z,f_a\rangle\\ 
&+ \text{{\em terms that involve at least one variable beyond those of}}\; x,y,z,
\endaligned
\end{equation}
for $a\geq 5$, where 
$$
(f_5,f_6,f_7,f_8):=({\bf 1},{\bf i},{\bf j},{\bf k}),
$$
while for $a\geq 5$,
$$
p^*_a\;\text{{\em has no terms with only variables of}}\; x,y,z, 
$$
by the first identity in~\eqref{3}. Meanwhile, by the block form of $B_\alpha^*,\alpha\leq 4,$ in~\eqref{good} we see
\begin{eqnarray}\nonumber
\aligned
&p_a^*\;\text{{\em consists of terms with only variables of}}\; x,z\; (\text{or}\; y,z)\;\text{{\em and of terms}}\\
&\text{{\em with only variables beyond those of}}\;x,y,z,
\endaligned
\end{eqnarray}
for $1\leq a\leq 4$. Therefore, from the identity~\cite[(3-7), p. 529]{OT}
\begin{equation}\label{pq}
\sum_{a=0}^8 p^*_a q^*_a=0,     
\end{equation}
we deduce, when we set 
$$
(e_1,e_2,e_3,e_4):=({\bf 1},{\bf i},{\bf j},{\bf k})
$$ 
and substitute~\eqref{q*a}, that
\begin{equation}\label{sum}
\sum_{a=5}^8\langle (e_b e_c-e_c e_b)e_p,f_a\rangle\, S^a_{b, c'}=0,
\end{equation}
where we set $x=e_b,y=e_c,z=e_p, 2\leq b,c\leq 4,1\leq p\leq 4$, and $c'\geq 5$, and
$\begin{pmatrix} S^a_{bc'}\end{pmatrix}$ represents the upper right block of $A^*_a,a\geq 5,$ in~\eqref{good}. Here, we also make use of the fact
that for $1\leq i\leq 4$, $q_i^*$ has no terms involving both variables of $x$ and $y$, while $q_0^*$ has no terms involving both $x$ and $z$ (or $y$ and $z$),  together with a third variable beyond $x,y,z$ in either case,  so that it is not a possibility to cancel the left hand side of~\eqref{sum} by the first five terms in~\eqref{pq};
this follows
from~\eqref{cov},~\eqref{q*0} without the restriction to $V$,  and the matrix types in~\eqref{good}. Consequently, we derive
$$
S^a_{b,c'}=0,\quad a,c'\geq 5, b\leq 4,
$$
and likewise, 
$$
S^a_{b', c}=0,\quad a,b'\geq 5,c\leq 4;
$$
that is, the only possibly nonzero blocks of $A^*_\alpha,\alpha\geq 5$, in~\eqref{good} are at the lower right corner.
$$
A^*_\alpha=\begin{pmatrix}0&0\\0&w_\alpha\end{pmatrix},\quad \alpha\geq 5.
$$ 
But then~\eqref{cov} establishes that
$$
q^*_a|_V=0,\quad a\geq 5.
$$
This is a contradiction to~\eqref{q*a}. 

Hence, we conclude that only~\eqref{p^} is valid, and thus $q^*_a|_V=0$ for all $0\leq a\leq 8$.
\end{proof}
\begin{corollary} Let $M$ be an isoparametric hypersurface with multiplicity pair $(m_{+},m_{-})=(7,8)$ not constructed by Ozeki
and Takeuchi. Given any point $p\in M$ with its unit normal 
$n$ and any vector $v$ at $p$
tangent to a curvature surface (which is a sphere) of dimension $7$, there is a $16$-dimensional Euclidean space passing through $p,n$ and $v$ such that it cuts
$M$ in a homogeneous isoparametric hypersurface with multiplicity pair $(m_{+},m_{-})=(3,4)$ in the $15$-dimensional sphere. 
\end{corollary}

\begin{proof} Notation as above, the $16$-dimensional Euclidean space is just ${\mathbb R}x^*\oplus{\mathbb R}n^*\oplus V$, where $x^*$ and $n^*$ are given in~\eqref{*} and
$V$ is given in~\eqref{V}, whose existence is generically established in the preceding theorem, where $p$ and $n$ span the same plane as $x^*$ and $n_0^*$,
or as $x$ and $n_0$, and $v$ is the vector $n_1$ in the normal basis $n_0,n_1,\cdots,n_7$ at the focal point $x$ with the normalization given in~\eqref{BC} and~\eqref{A}. 
Taking limit, the existence of the $16$-dimensional Euclidean space is established everywhere.

\end{proof} 

The preceding corollary points to that the isoparametric hypersurface should be one of the two constructed by Ferus, Karcher, and M\"{u}nzner. We will prove in the next section that this is indeed the case.

\section{The hypersurface is one constructed by Ferus, Karcher, and M\"{u}nzner}\label{sec7}
When both $x$ and $x^\#$ are generic in $M_{+}$ with the chosen $4$-nullity bases as specified in Remark~$\ref{rkk}$, it is more convenient to consider the conversion of~\eqref{good} from $x^*$ to $x$ to obtain

\begin{equation}\label{mtx}
\aligned
&A_a=\begin{pmatrix}z_a&0\\0&w_a\end{pmatrix},\quad B_a=\begin{pmatrix}0&0\\0&c_a\end{pmatrix},\quad C_a=\begin{pmatrix}0&0\\0&f_a\end{pmatrix},\quad 1\leq a\leq 3,\\
&A_a=\begin{pmatrix} 0&\beta_a\\\gamma_a&\delta_a\end{pmatrix},\quad B_a=\begin{pmatrix}0&d_a\\b_a&c_a\end{pmatrix},\quad C_a=\begin{pmatrix} 0&g_a\\b_a&f_a\end{pmatrix},\quad 4\leq a\leq 7.
\endaligned
\end{equation}
Observe that the matrices $\begin{pmatrix}\sqrt{2}c_a&w_a\end{pmatrix},\, 1\leq a\leq 3,$ form a 
Clifford multiplication of type $[3,4,8]$.
$$
F:{\mathbb R}^3\times {\mathbb R}^4\rightarrow{\mathbb R}^8,\quad F(e_a,f_\alpha)=\;\text{the}\;\alpha\text{th row of}\;\begin{pmatrix}\sqrt{2}c_a&w_a\end{pmatrix}.
$$
This is the starting point of our remaining task to pinpoint the characteristic features of the undetermined blocks of the matrices in~\eqref{mtx}. In~\cite{CW}, we have classified the orthogonal multiplications of type $[3,4,8]$, which we will apply to understand~\eqref{mtx}.

\begin{lemma}\label{ALG} Given four $4$-by-$3$ matrices $b_i,4\leq i\leq 7$, consider the linear combinations
$$
b(x):=x_1b_4+\cdots+x_4b_7.
$$
Suppose the first column of $b(x)$ is 
$$
x=\begin{pmatrix} x_1&x_2&x_3&x_4\end{pmatrix}^{tr}
$$ 
(more generally, suppose the four components of the first column are linearly independent linear polynomials), and suppose generic\; $b(x)$ is of rank $=2$. Then we may assume, e.g., the third columns of $b_i,4\leq i\leq 7,$ are zero after a simultaneous column operation, i.e., the three column vectors of
$b_i$ are subject to the same linear constraint for all $4\leq i\leq 7$.
\end{lemma}

\begin{proof} This follows from the fact that the Koszul complex
$$
0\longrightarrow R\stackrel{x\wedge}\longrightarrow \Lambda^1 R^4\stackrel{x\wedge}\longrightarrow\Lambda^2R^4\stackrel{x\wedge}\longrightarrow\Lambda^3R^4\stackrel{x\wedge}\longrightarrow \Lambda^4 R^4\rightarrow 0,
$$
where $R:={\mathbb R}[x_1,x_2,x_3,x_4]$ is the polynomial ring in four variables and  $x\wedge$ means taking the wedge product against $x$, is a free resolution. The assumption that $b(x)$ is generically of rank 2 means that the wedge product of second column $v_2$ and third column $v_3$ of $b(x)$ lives in the kernel of
$$
\longrightarrow\Lambda^2R^4\stackrel{x\wedge}\longrightarrow\Lambda^3R^4,\quad v_2\wedge v_3\mapsto x\wedge(v_2\wedge v_3)=0,
$$
so that either $v_2\wedge v_3=0$, in which case they differ by a constant multiple, or $v_2\wedge v_3=x\wedge w$ for some $w\in R^4$, so that we may assume the first two columns of b(x) are both $x$ up to a constant multiple.
\end{proof}
 
\begin{remark} When the generic rank of $b(x)$ is $1$, it is clear that two column vectors of $b(x)$ are constant multiples of the remaining one because all entries are linear.
\end{remark}

\begin{corollary}\label{cccor} Assume the isoparametric hypersurface is not of the type constructed by Ozeki and Takeuchi. Away from points of Condition A in $M_{+}$, let $(n_0,n_1)$ be $4$-null with the decomposition in~\eqref{good} $($expressed over $M_{-}$ with the conversion to the corresponding data over $M_{+}$ by ~\eqref{gOOd},~\eqref{gooD},~\eqref{Good}$)$.
Then for $4\leq a\leq 7$ over $M_{+}$, the generic linear combination of the $4$-by-$3$ matrices $b_a$ in
$$
B_a=\begin{pmatrix}0&d_a\\b_a&c_a\end{pmatrix}
$$
is of rank $\leq 2$, so that by Lemma~$\ref{ALG}$ we may assume $b_a,4\leq a\leq 7,$ share a common
zero column. As a consequence, the spectral data $(\sigma,\Delta)$ is such that $\sigma=sI$ for some $s>0$.
\end{corollary}

\begin{proof} At generic $x$ and $x^\#$ in $M_{+}$ with 4-nullity, $b_a$ cannot be all zero for $4\leq a\leq 7$ at $x$. Otherwise, translated to the data at $x^\#$ by~\eqref{gOOd} and~\eqref{gooD}, the matrices $B_p^\#,1\leq p\leq 3,$ which are
of the form 
\begin{equation}\label{sharp}
B_p^\#=\begin{pmatrix}0&0\\0&c_p^\#\end{pmatrix},
\end{equation}
would be such that $c_p^\#=0, 1\leq p\leq 3$, which contradicts the 4-nullity of $B_1^\#$. 

Suppose, e.g., $b_4$ is of rank 3. Since
\begin{equation}\label{d4b4}
d_4\sigma^{-1}b_4=0,
\end{equation}
which holds by an analysis similar to the one following~\eqref{taylor}, $d_4$ is perpendicular to the 3-dimensional column space of $\sigma^{-1} b_4$. Hence by row operations without changing the spectral data in the normalized $B_1$, we may assume the only nonzero row of $d_4$ is the first one. 

We claim that $c_4=f_4$. To prove the claim, observe that we have
$$ 
\sigma (c_4-f_4)=-(c_4-f_4)^{tr}\sigma,\quad b_4^{tr}(c_4-f_4)=0,
$$
which are~\eqref{h} and the first equation of~\eqref{c-f}, which together with the fact that $b_4$ is of rank 3 force $c_4-f_4=0$.
It follows that
$$
d_4^{tr}d_4=g_4^{tr}g_4
$$
by the second equation of~\eqref{c-f}, so that $g_4$ is of the same rank as $d_4$, which is $\leq 1$. Now the formula
$
A_4A_4^{tr}+2B_4B_4^{tr}=I
$
gives
$$
\beta_4\beta_4^{tr}+2d_4d_4^{tr}=I,
$$
where as usual
$$
A_4=\begin{pmatrix} 0&\beta_4\\\gamma_4&\delta_4\end{pmatrix},
$$
so that $\beta_4\beta_4^{tr}=I-2d_4d_4^{tr}$ is diagonal of rank at least 3 since the only nonzero row of $d_4$ is the first one. But then the identity
$$
g_4\sigma^{-1}=d_4\sigma^{-1}\Delta-\beta_4,
$$
which is~\eqref{b}, gives that $g_4$ is of rank at least 3. This is a contradiction.

It follows that the generic rank of linear combination $b(x):=x_1b_4+\cdots+x_4b_7$ is $\leq 2$, so that by Lemma~\ref{ALG} we may assume a fixed column of $b_4,\cdots,b_7$ is identically zero. Note that the condition in Lemma~\ref{ALG} that the four components of the first column are linearly independent linear polynomials is satisfied, because when viewed at $x^\#$ the first columns of $b_4,\cdots,b_7$ are, respectively, the first, second, third, and fourth columns of $c_1^\#$, i.e.,
$$
c_1^\#=\begin{pmatrix}\sigma_1&0&0&0\\0&\sigma_1&0&0\\0&0&\sigma_2&0\\0&0&0&\sigma_2\end{pmatrix},
$$
in~\eqref{sharp}, Similarly, the second (vs. third) columns of $b_4,\cdots,b_7$ are the respective columns of $c_2^\#$ (vs. $c_3^\#$). Therefore, when viewed at $x^\#$, we conclude by Lemma~\ref{ALG} that one of the $c_p^\#$, and so the corresponding $B_p^\#$, $p=2,3,$ is identically zero, which we have seen in the example in Section~\ref{subsec}. It follows from~\cite[Sections 4, 5]{CW} that 
$\sigma_1=\sigma_2=\sigma=sI$ for some $s>0$.
\end{proof} 

\begin{remark} We summarize before we proceed further. When both $x$ and $x^\#$ are generic in $M_{+}$ with the chosen $4$-nullity bases as specified in Remark~$\ref{rkk}$, we have~\eqref{mtx}
where, interchanging $x$ and $x^\#$ by symmetry, we may assume
$$ 
c_1=sI,\quad c_3=0.
$$
The second item of Corollary~$\ref{inde}$ then implies that all $d_a\neq 0,4\leq a\leq 7,$ because now $B_3=0$ and the first four rows of $B_1$ and $B_2$ are zero. As a result of $C_3^{tr}C_3=B_3^{tr}B_3$ we obtain $C_3=0$, so that a similar situation holds for $C_a,1\leq a\leq 7$ as well. 

Moreover, the third columns of the four $4$-by-$3$ matrices $b_4,\cdots,b_7$ are zero in accordance with $c_3^\#=0$; in fact, we know by~\cite[Section 5]{CW} that since 
 $c_2^\#$ is of the form
 \begin{equation}\label{NQE}
 c_2^\#=a\,Id+b\begin{pmatrix}I&0\\0&\pm I\end{pmatrix},\quad I=\begin{pmatrix}0&-1\\1&0\end{pmatrix},\;\; b\neq 0,
 \end{equation}
with $c_1^\#=s\, Id$ and $c_3^\#=0$, they can be convert to the data
\begin{equation}\label{EQQQ}
b_4=\begin{pmatrix}s&a&0\\0&b&0\\0&0&0\\0&0&0\end{pmatrix}, \quad b_5=\begin{pmatrix}0&-b&0\\s&a&0\\0&0&0\\0&0&0\end{pmatrix},\quad b_6=\begin{pmatrix}0&0&0\\0&0&0\\s&a&0\\0&\pm b&0\end{pmatrix},\quad b_7=\begin{pmatrix}0&0&0\\0&0&0\\0&\mp b&0\\s&a&0\end{pmatrix}
\end{equation}
at $x$, whose linear combinations are of generic rank $2$. 

In particular, a glance at 
$B_a,1\leq a\leq 7,$ in~\eqref{mtx} shows that their third columns are all zero, or equivalently, that there is a common kernel vector for all the shape operators $S_n$ for all $n$. 
\end{remark}

\begin{corollary}\label{corollary10} Let $(m_{+},m_{-})=(7,8)$. Assume the isoparametric hypersurface is not the one constructed by Ozeki and Takeuchi. Then at each point of\; $M_{+}$ the intersections of the kernels of all the shape operators is nontrivial, which is generically of dimension $1$.
\end{corollary}

\begin{proof} The conclusion of the preceding remark establishes the existence of such a common eigenvector for generic points of $M_{+}$, and so the existence is true everywhere by taking limit. Generically the dimension of this common eigenspace must be 1-dimensional because generic linear combinations of $b_4,\cdots,b_7$ is of rank 2 as said in the preceding remark.
\end{proof}

\begin{remark} The preceding corollary gives us a clear geometric picture. Namely, when the isoparametric hypersurface with multiplicities $(m_{+},m_{-})=(7,8)$ is not the one constructed by Ozeki and Takeuchi, consider the quadric ${\mathcal Q}_6$ of oriented $2$-planes in the normal space at a generic point $x\in M_{+}$. We know a generic element $(n_0,n_1)$ in ${\mathcal Q}_6$ is $4$-null, or equivalently, the intersection $V$ of the kernels of $S_{n_0}$ and $S_{n_1}$ is $3$-dimensional. By the preceding corollary, there is a nonzero unit vector $v\in V$ common to all kernels of the shape operators at $x$. We choose an orthonormal basis $e_1,e_2,e_3=v$ spanning $V$.
When viewed at the mirror point $x^\#=n_0\in M_{+}$, $e_1,e_2,e_3$ are converted to three normal basis vectors of which the three matrices $c_1^\#,c_2^\#,c_3^\#$ given in~\eqref{mtx} are of the form $c_1^\#=s\,Id, c_3^\#=0$, and $c_2^\#$ is given in~\eqref{NQE}.

By a symmetric reasoning, all this holds true as well at $x$ when both $x$ and $x^\#$ are generic.
\end{remark}







\begin{corollary} A generic linear combination 
$$
d(x):=x_1d_4+\cdots+x_4d_7
$$ 
of $d_4,\cdots,d_7$ is of rank $\leq 2$. In particular, we may assume the last two rows of $d(x)$ are zero. 
\end{corollary}

\begin{proof} $b(x)$ is of generic rank 2 by the preceding corollary, which is explicitly given in~\eqref{EQQQ}. On the other hand, similar to~\eqref{d4b4}, we have
\begin{equation}\label{dxbx}
d(x)b(x)=0
\end{equation}
(and similarly $g(x)b(x)=0$), knowing now $\sigma=sI$, so that each row $r_i(x),1\leq i\leq 4,$ of $d(x)$ annihilates $b(x)$. Hence, it must be that
$$
r_i(x)=\begin{pmatrix}x_1&x_2&x_3&x_4\end{pmatrix}M_i,
$$
where $M_i$ is a skew-symmetric constant matrix, because the first column of $b(x)$ is $\begin{pmatrix}x_1&x_2&x_3&x_4\end{pmatrix}$, which is a regular sequence~\cite[(5), p. 7]{Chiq},~\cite[p. 93]{Chi5}. On the other hand, the same sort of relation must hold true for the second column of $b(x)$ as well. That is,
$$
r_i(x)=\begin{pmatrix} x_1&x_2&x_3&x_4\end{pmatrix}\Gamma(\Gamma^{-1}M_i),
$$
where $\Gamma^{-1}M_i$ is skew-symmetric,
$$
\Gamma:=
\begin{pmatrix}a&b&0&0\\
-b&a&0&0\\
0&0&a&\pm b \\0&0&\mp b&a\end{pmatrix},
$$
and $\begin{pmatrix}x_1&x_2&x_3&x_4\end{pmatrix}\Gamma$ is the second column of $b(x)$ transposed in light of~\eqref{EQQQ}. 
It follows that
$$
M_i=\begin{pmatrix}0&U\\-U^{tr}&0\end{pmatrix},\quad U:=\begin{pmatrix}u&v\\-v&u\end{pmatrix}.
$$
Therefore, all four rows of $d(x)$ are linearly spanned by the two vectors
\begin{equation}\label{xyzw}
\aligned
&\begin{pmatrix}-x_3&-x_4&x_1&x_2\end{pmatrix},\\
&\begin{pmatrix}
-x_4&x_3&-x_2&x_1
\end{pmatrix}.
\endaligned
\end{equation}

. 
\end{proof}




\begin{corollary} With the condition that the last two rows of $d(x)$ are zero, we may assume the first two rows of $g(x)$ are zero.
\end{corollary}

\begin{proof} We know $\sigma=sI$ and 
$$
\Delta=\begin{pmatrix}\tau J&0\\0&\tau J\end{pmatrix},\quad J=\begin{pmatrix}0&1\\-1&0\end{pmatrix},\quad \tau=\sqrt{1-2s^2}
$$
for some $s$. By~\eqref{mtx} and the fact that
$$
A(x)A(x)^{tr}+2B(x)B(x)^{tr}=I,
$$
where $A(x)=x_4A_4+\cdots+x_4A_7$ and likewise for $B(x)$, it follows by comparing the upper left block of the involved matrices that we obtain
\begin{equation}\label{bx}
\beta(x)\beta^{tr}(x)+2d(x)d(x)^{tr}=I.
\end{equation}
We employ
\begin{equation}\label{bex}
\beta(x)=s^{-1}(d(x)\Delta-g(x)), 
\end{equation}
which is~\eqref{b}, to derive
$$
\aligned
s^2\beta(x)\beta(x)^{tr}&=(d(x)\Delta-g(x))(d(x)\Delta-g(x))^{tr}\\
&=\tau^2d(x)d(x)^{tr}+g(x)g(x)^{tr}-(d(x)\Delta g(x)^{tr}-g(x)\Delta d(x)^{tr}),
\endaligned
$$
so that with $\tau^2=1-2s^2$ and~\eqref{bx} we obtain
$$
s^2I=d(x)d(x)^{tr}+g(x)g(x)^{tr}-(d(x)\Delta g(x)^{tr}-g(x)\Delta d(x)^{tr}), 
$$
where the lower right 2-by-2 blocks of all the matrices on the right, except for $g(x)g(x)^{tr}$, are zero because the last two rows of $d(x)$ are zero. Therefore, the lower right 2-by-2 block of $g(x)g(x)^{tr}$ is $s^2I$, which means that the last two rows of $g(x)$ are linearly independent. We can accordingly do row reductions to annihilate the first two rows of $g(x)$ by the last two while performing the same row reduction on $d(x)$ to not to change the spectral data, where in fact $d(x)$ is not affected by the row reduction since its last two rows are zero.

\end{proof}

\begin{corollary} The spectra data are $(\sigma,\Delta)=(1/\sqrt{2}I,0)$.
\end{corollary}

\begin{proof} Employing that $d(x)$ and $g(x)$ are of the form
$$
d(x)=\begin{pmatrix}d_1(x)&d_2(x)\\0&0\end{pmatrix},\quad g(x)=\begin{pmatrix}0&0\\g_1(x)&g_2(x)\end{pmatrix},
$$
by the preceding corollary, we employ~\eqref{bx} and~\eqref{bex} to arrive at
$$
\aligned
&d_1(x)d_1(x)^{tr}+d_2(x)d_2(x)^{tr}=g_1(x)g_1(x)^{tr}+g_2(x)g_2(x)^{tr}=s^2I,\\ &\tau(d_1(x)Jg_1(x)^{tr}+d_2(x)Jg_2(x)^{tr})=0, \quad x_1^2+\cdots x_4^2=1.
\endaligned
$$
However, since $d_1(x)$ are in terms of $x_3,x_4$ and $d_2(x)$ are in terms of $x_1,x_2$, and likewise for $g_1(x)$ and $g_2(x)$, there must hold, by homogenizing,
\begin{equation}\label{final}
\aligned
&d_1(x)d_1(x)^{tr}=s^2 (x_3^2+x_4^2),\quad d_2(x)d_2(x)^{tr}=s^2(x_1^2+x_2^2).\\
&\tau d_1(x)Jg_1(x)^{tr}=0=\tau d_2(x)Jg_2(x)^{tr}.
\endaligned
\end{equation}
That is, 
\begin{equation}\label{x1-x4}
d_1=sU\begin{pmatrix}-x_3&-x_4\\x_4&-x_3\end{pmatrix},\quad 
d_2=sU\begin{pmatrix}x_1&x_2\\-x_2&x_1\end{pmatrix}
\end{equation}
for some 2-by-2 orthogonal matrix $U;$ 
by the same token,
\begin{equation}\label{x4-x1}
g_1=sW\begin{pmatrix}-x_3&-x_4\\x_4&-x_3\end{pmatrix},\quad
g_2=sW\begin{pmatrix}x_1&x_2\\-x_2&x_1\end{pmatrix}
\end{equation}
with $W$ orthogonal, which we substitute into the third equality of~\eqref{final} to derive
$$
0=\tau U\begin{pmatrix}0&x_3^2+x_4^2\\-(x_3^2+x_4^2)&0\end{pmatrix}W^{tr}.
$$
This is possible only when $\tau=0$, i.e., when the spectral data $(\sigma,\Delta)=(I/\sqrt{2},0)$.
\end{proof}

\begin{corollary} Notation as in~\eqref{mtx}, we have\; $c_a=f_a,1\leq a\leq 7,$ and hence 
$\delta_a,1\leq a\leq 7,$  are skew-symmetric. 
\end{corollary}

\begin{proof} Let us first handle the case when $4\leq a\leq 7$. 
We know $c_a-f_a$ is skew-symmetric by~\eqref{h} because the spectral data are $(\sigma,\Delta)=(I/\sqrt{2},0)$ now. Moreover,
$$
(c_a-f_a)^{tr}b_a=0
$$ 
by~\eqref{c-f}. Hence linear combinations of $c_a-f_a,4\leq a\leq 7,$ i.e.,
$$
h(x):=x_1(c_4-f_4)+\cdots+x_4(c_7-f_7),
$$ 
satisfies 
$$
h(x)b(x)=0
$$ 
and so the first row of $h(x)$ is a linear combination of the vectors in~\eqref{xyzw}. However, since $h(x)$ is skew-symmetric, the first component of the first row of $h(x)$ is zero. Consequently,
the entire first row of $h(x)$ is, and similarly, all rows of $h(x)$ are zero. That is, $c_a=f_a$ for all $4\leq a\leq 7$.

For $1\leq a\leq 3$, the first columns of $b_4,\cdots,b_7$ at $x$ are placed in order to form the first matrix $c_1^\#$ and $f_1^\#$ at $x^\#$, the second columns to form $c_2^\#$ and $f_2^\#$,  the third to form $c_3^\#$ and $f_3^\#$, and vice versa. It follws that $c_a=f_a,1\leq a\leq 3,$ because they are both generated by aligning the columns of  $b_4^\#,\cdots,b_7^\#$.

That $\delta_a$ is skew-symmetric follows from~\eqref{ee} and $\Delta=0$. Lastly,
$$d_a^{tr}d_a=g_a^{tr}g_a$$ follows from the second identity in~\eqref{c-f}.
\end{proof}

We are in a position to prove the classification theorem.

\begin{theorem} Let $(m_{+},m_{-})=(7,8)$. Assume the isoparametric hypersurface is not the one constructed by Ozeki and Takeuchi. 
Then the hypersurface is one of the two constructed by Ferus, Karcher, and M\"{u}nzner.
\end{theorem}

\begin{proof} Referring to~\eqref{good}, we will show there is a {\bf Clifford frame}~\cite[(8.1)-(8.4), p. 28]{CCJ} on the unit normal bundle of $M_{-}$. 

Recall the tangent bundle ${\mathcal T}$ of the unit bundle ${\mathcal UN}$ of $M_{-}$ naturally splits into the vertical part ${\mathcal V}$ and and the horizontal part ${\mathcal H}$, and ${\mathcal H}$ further splits into three subspaces which, at $(x^*,n^*)\in{\mathcal UN}$ sitting over $x^*\in M_{-}$, are the horizontal lift of the three eigenspaces of the shape operator $S_{n^*}$ at $x^*$ with eigenvalues $0, 1, -1$, respectively, i.e.,
$$
{\mathcal T}={\mathcal V}\oplus {\mathcal E}_0^*\oplus {\mathcal E}_{+}^*\oplus {\mathcal E}_{-}^*,
$$
where the basis elements of ${\mathcal V}, {\mathcal E}_0^*, {\mathcal E}_{+}^*, {\mathcal E}_{-}^*$ are indexed by subscripts $\alpha, \mu, a, p, $ where $1\leq\alpha,\mu\leq 8,1\leq a, p\leq 7$, so that a typical one is denoted, respectively, by $e_\alpha,e_\mu,e_a,e_p$ in the corresponding range with dual frame $\omega^\alpha,\omega^\mu,\omega^a,\omega^p$ and connection forms $\theta^i_j$ with $i,j$ ranging over all possible indexes; for a specific index in a range, we will denote it by, e.g., $e_{\alpha=5}, \theta^{a=6}_{\mu=5},$ etc. Write
\begin{equation}\label{FF}
\theta^i_j=\sum_k F^i_{jk}\omega^k.
\end{equation}
We know~\cite[(2.9), p. 9]{CCJ} $F^i_{jk}=0$ whenever exactly two indexes fall in the same $\alpha,\mu,a$, or $p$ range. 

A Clifford frame is one on ${\mathcal T}$ 
that satisfies 
\begin{equation}\label{cliff}
\aligned
&A_\alpha^*=A_\mu^*,\\
&(a,\mu)\;\text{entry of }\; B_\alpha^*=-(a,\alpha)\;\text{entry of}\; B_\mu^*,\\
&(p,\mu)\;\text{entry of }\; C_\alpha^*=-(p,\alpha)\;\text{entry of}\; C_\mu^*,\\
&\theta^i_j-\theta^{i'}_{j'}=\sum_k L^i_{jk} (\omega^k+\omega^{k'})
\endaligned
\end{equation}
for some smooth functions $L^i_{jk}$, where $i,j,k$ are in the $\alpha$ index range and $i',j',k'$ are in the $\mu$ index range with the same respective index values (i.e., $i$ indicates $\alpha=i$ and $i'$ indicates $\mu=i$, etc.) 

It was shown in~\cite{CCJ} that a Clifford frame characterizes an isoparametric hypersurfaces constructed by Ozeki-Takeuchi and Ferus-Karcher-M\"{u}nzner. Moreover, it is shown in~\cite{Ch} that a Clifford frame is the same as a distribution ${\mathcal D}$ over ${\mathcal T}$ given by
$$
{\mathcal D}={\mathcal F}\oplus{\mathcal E}_{+}^*\oplus{\mathcal E}_{-}^*,
$$
where ${\mathcal F}\subset {\mathcal V}\oplus  {\mathcal E}_0^*$ is the graph of an orthogonal bundle map
$$
Q: {\mathcal E}_0^*\rightarrow {\mathcal V},
$$
where we define 
\begin{equation}\label{Qsign}
e_{\alpha=j}:=-Q(e_{\mu=j}),\quad 1\leq j\leq 8,
\end{equation}
to set up an orthonormal basis for ${\mathcal V}$
corresponding to a given one in ${\mathcal E}_0^*$.

Furthermore, in~\cite{Ch} it was shown that the first three equations in~\eqref{cliff} mean that the distribution ${\mathcal D}$ is involutive and each of its leaves induces an isometry of $M_{-}$ that extends, by the last equation of~\eqref{cliff} which means that the forms on its left hand side annihilate the distribution ${\mathcal D}$, to an ambient isometry so that the isoparametric hypersurface is one of the two constructed by Ferus, Karcher, and M\"{u}nzner. 

 Converted to the language of the unit bundle of $M_{+}$ at $(x,n)$ instead, where the shape operator $S_{n}$ has the eigenspaces $E_0,E_{+},E_{-}$,
the first three equations of~\eqref{cliff} say, in view of~\eqref{gOOd},~\eqref{gooD},~\eqref{Good},~\eqref{good}, that
there is an orthogonal map $Q$ that identifies the $j$th basis vector $e_{\mu=j}\in E_{-}$ with $-e_{\alpha=j}\in E_{+}$ so that  
\begin{equation}\label{EQ}
\aligned
&B_a=C_a,\; \forall a,\\
&A_a\;\text{is skew-symmetric},\;\forall a,\\
&A_a^\#\;\text{is skew-symmetric},\;\forall a.
\endaligned
\end{equation}

The first item of~\eqref{EQ} is true. Indeed~\eqref{x1-x4} and~\eqref{x4-x1} mean that if we perform orthogonal row operations by $U$ and $W$ we may assume 
$$
d_1(x)=g_1(x),\quad d_2(x)=g_2(x).
$$  
That is, if we define the bundle map $Q$ that swaps the first (last) two $\mu$-rows of $g(x)$ in $C_a$ with the last (first) two $\alpha$-rows of $d(x)$ in $B_a$ and leaves all remaining four rows of $B_a$ and $C_a$ unchanged, then $B_a=C_a$ via the identification $Q$ (i.e., we may assume $d_a=g_a$ via $Q$).

It suffices to establish the second item of~\eqref{EQ}. Now $\delta_a$ is skew-symmetric by the preceding corollary. $z_a,1\leq a\leq 3$ are skew-symmetric since $z_a,1\leq a\leq 3,$ generate the Clifford algebra $C_3$ by~\eqref{mtx}, while the upper left blocks of $A_a, 4\leq a\leq 7$ are zero. The nature of $Q$ does not change the skew-symmetry of these blocks.

Next, with $d_a=g_a$ via $Q$ in place, we derive from~\eqref{b} and~\eqref{d} (with $\Delta=0$) that we have
$\beta_a=\gamma_a^{tr}.$ 
However, we can now change the sign of the last four $\alpha$-rows and $\mu$-columns of $A_a$ without affecting the skew-symmetry of $\delta_a$ and the property $d_a=g_a,c_a=f_a$, so that now
$$
\beta_a=-\gamma_a^{tr},\quad 1\leq a\leq 7.
$$
That is, $A_a$ is now skew-symmetric for all $1\leq a\leq 7$ with this modified $Q$.

It remains to establish the last item of~\eqref{cliff}, knowing that the first three equations are true via $Q$. By~\cite[Lemma 2, p. 11]{Chiq}, the last item holds true if either $\alpha=i$ or $\alpha=j$ indexes a basis vector in the image of the map
\begin{equation}\label{H}
H:{\mathcal E}_{+}^*\oplus {\mathcal E}_{-}^*\rightarrow {\mathcal E}_0^*,\quad (e_a,e_p)\mapsto \sum_{\alpha}S^a_{\alpha p}e_\alpha,
\end{equation}
which is easily seen to be the direct sum of all $e_{\alpha=l}$ for $l\neq 3, 4$ (i.e., the 3rd and 4th rows of $B_a$ are zero for all $1\leq a\leq 7$). Thus, it suffices to show that the last item of~\eqref{cliff} is valid for $i=3,j=4$ in the $\alpha$-range.

The left hand side of the last equation in~\eqref{cliff} annihilates the vectors in ${\mathcal E}^*_{+}\oplus{\mathcal E}^*_{-}\subset {\mathcal D}$ because they are horizontal, so that, as said below~\eqref{FF}, $\theta^3_4$ and $\theta^{3'}_{4'}$ annihilate them since exactly $3$ and $4$ (respectively, $3'$ and $4'$) are in the same $\alpha$ (respectively, $\mu$) range. (It is understood that by $3$ we mean $\alpha=3$ and by $3'$ we mean $\mu=3$, etc.)

We show the left hand side of the last equation in~\eqref{cliff} annihilates ${\mathcal F}\subset{\mathcal D}$ as well. For
$$
v:=e_{l'}-e_l\in{\mathcal F},
$$ 
we calculate
\begin{equation}\label{4v}
\theta^3_4(v)=-\theta^3_4(e_l), \quad \theta^{3'}_{4'}(v)=\theta^{3'}_{4'}(e_{l'})
\end{equation}
again by what is said below~\eqref{FF}. 

Since the calculation is pointwise, we first look at the geometry before we proceed. For $x\in M_{+}$ and $n$ in the unit normal sphere to $M_{+}$ at $x$, the map
\begin{equation}\label{diffeo}
f:(x,n)\mapsto (x^*,n^*)=((x+n)/\sqrt{2},(x-n)/\sqrt{2})
\end{equation}
sets up a diffeomorphism between the normal bundles of $M_{+}$ and $M_{-}$. Fix a point $(x_0,n_0)$ in the unit normal bundle of $M_{+}$, consider two sets  
$$
S_{+}:=\{(x,n): x+n=x_0+n_0\}, \quad S_{-}:=\{(x,n):x-n=x_0-n_0 \}.
$$
$S_{\pm}$ are two 8-dimensional spheres. Indeed, taking derivative of $x\pm n=c$ with $c$ a constant, we have $dx\pm dn=0$, which means that a typical tangent space to $S_{\pm}$ is the eigenspace ${\mathcal E}_{\pm}$ at $(x,n)$, respectively. 

The diffeomorphism $f$ maps $S_{+}$ to a sphere whose tangent space at $(x_0^*,n_0^*)$ is ${\mathcal V}$, so that it is the fiber of the unit normal bundle of $M_{-}$ over $x_0^*$, and $f$ maps $S_{-}$ to a sphere whose tangent space at $(x_0^*,n_0^*)$ is the horizontal ${\mathcal E}_0^*$. Thus to calculate the quantities in~\eqref{4v}, it suffices to observe that~\eqref{H} gives us the information
$$
\dim(\bigcap_{a=1}^7\text{kernel}(B_a^{tr}))=2.
$$
This translates to $S_{+}$ to say that the tangent space to $S_{+}$ at $(x,n)$ is identified with $E_+$ of the second fundamental form $S_n$, in which there naturally sits a 2-dimensional plane that is the
intersection of all kernels of the $B_m^{tr}$-block of $S_m$ with $m$ perpendicular to $n$ at $x$, which form a 2-plane bundle ${\mathcal P}_{+}$ over $S_{+}$. By the same token there is a 2-plane bundle
${\mathcal P}_{-}$ over $S_{-}$ which comes from the intersection of all kernels of the $C_m^{tr}$-block of $S_m$ with $m$ perpendicular to $n$ at $x$. Now, the above fact that after swapping rows
we may assume $d_a=g_a,1\leq a\leq 7$, means that once we set up the coordinate system of the ambient Euclidean space by the eigenspace decomposition
$$
{\mathbb R}x\oplus{\mathbb R}n\oplus E_0\oplus E_{+}\oplus E_{-}
$$
of the shape operator $S_n$ at $x$ for $(x,n)\in S_{+}$, where the third and fourth rows of $B^{tr}_a=0$ for all $1\leq a\leq 7$, we may assume, after swapping the third and fourth rows with the first and second, that ${\mathcal P}_{+}$ and ${\mathcal P}_{-}$ are parametrized identically in the coordinates. That is, in the coordinates we can parametrize $S_{+}$ and $S_{-}$ via an isometry $\iota$ in which ${\mathcal P}_{+}$ is brought to ${\mathcal P}_{-}$. As a consequence, via the diffeomorphism $f$ in~\eqref{diffeo}, a local basis $(e_3,e_4)$ spanning ${\mathcal P}_{+}$ is converted to one around ${\mathcal V}$ at $(x_0^*,n_0^*)$, and local basis $(e_3',e_4')$ spanning ${\mathcal P}_{-}$ is converted to one on the sphere whose tangent space at $(x_0^*,n_0^*)$ is ${\mathcal E}^*_0$. Thus through the isometry $\iota$ we see that
$$
\theta^3_4=\langle de_3,e_4\rangle=\langle de_3',e_4'\rangle=\theta^{3'}_{4'},
$$
which gives~\eqref{4v}, remarking that there the extra sign is a result of the sign convention in our identification map $Q$ in~\eqref{Qsign}, whose choice is in agreement with that of an isoparametriic hypersurface constructed by Ferus, Karcher, and M\"{u}nzner. 

The four equations in~\eqref{cliff} are satisfied. Thus the isoparametric hypersurface is one of the two constructed by Ferus, Karcher, and M\"{u}nzner, if it is not the one constructed by Ozeki and Takeuchi.

\end{proof}

\vspace{10mm}

\begin{appendix}{\begin{center}APPENDIX I\end{center}}

We give certain codimension 2 estimates needed for imposing constraints on 1-, 2-, and 3-nullity in Section~\ref{constraints}.

\begin{lemma}\label{import} Consider\, ${\mathbb C}^{15}={\mathbb C}^8\oplus{\mathbb C}^7$
parametrized by $(x,z)$. Consider the homogeneous equations of degree $2$
$$
f_0:=\sum_{\alpha=1}^8 (x_\alpha)^2=0,\quad f_i:=\sum_{\alpha=1,p=1}^{8,7}\theta_{\alpha p}^i x_\alpha z_p=0,\quad i=1,2,3.
$$
Let $Z_k$ be the variety carved out by\, $0=f_0=\dots=f_k,0\leq k\leq 3$. Suppose
$f_1,f_2,f_3$ are linearly independent. Then $Z_k,0\leq k\leq 3,$ are
irreducible of codimension $k+1$.
For an $f_4$ of homogeneous degree $2$ linearly independent
from $f_0,f_1,f_2,f_3$, we have that $f_0,f_1,f_2,f_3,f_4$ form a regular sequence and so they carve out
a subvariety of codimension $5$.
\end{lemma}

\begin{proof} The singular set of $f_0$ consists of points of the form $(0,z)$.
Hence the codimension 2 estimate goes through for $Z_0$. Set
$$
R_0:=\begin{pmatrix}I&0\\0&0\end{pmatrix},\quad R_k:=\begin{pmatrix}
0&\theta_k\\\theta^{tr}_k&0\end{pmatrix}, k=1,2,3,
$$
where the identity matrix is of size 8-by-8 and $\theta_k$ is the 8-by-7 matrix
whose entries are $\theta_{\alpha p}^k$. As in~\eqref{sing}, we estimate the
dimension of the kernel of
$$
S:=c_0R_0+\cdots+c_kR_k
$$
with $[c_0;\cdots:c_k]\in{\mathbb C}P^k,k=1,2,3.$ For simplicity, we may assume
$c_0=1$. Then
$$
S:=\begin{pmatrix} I&\Theta_k:=\sum_{l=1}^kc_l\theta_l\\(\Theta_k)^{tr}&0\end{pmatrix},
$$
whose kernel elements $(x,z)^{tr}$ satisfies
$$
x+\Theta_k z=0,\quad (\Theta_k)^{tr}x=0.
$$
From this we see that
$$
(\Theta_k)^{tr}\Theta_k z=0,
$$
so that the dimension of $z$ is at most 6 for a generic choice of $[c_0:\cdots:c_k]$
(respectively, 7 for a nongeneric choice) because the independence of $p_1,p_2,p_3$ dictates that $\Theta_k$ is nonzero
for such a generic choice. Therefore, the fact that $x=-\Theta_k z$ implies that
the kernel dimension is at most 6 for a generic parameter $[c_0:\cdots:c_k]$
of dimension $k$. Hence the total dimension is at most $6+k$ (respectively,
$7+(k-1)=6+k$). On the other hand, $\dim(Z_k)-2\geq (15-k-1)-2=12-k$. Therefore,
the codimension 2 estimate goes through for  $Z_k,1\leq k\leq 3.$
\end{proof}

\begin{lemma}\label{import1} Consider\,
${\mathbb C}^{14}\simeq{\mathbb C}^7\oplus{\mathbb C}^7$
parametrized by $(x,z)$, and
consider the homogeneous equations of degree $2$
$$
f_0:=\sum_{\alpha=1}^7 (x_\alpha)^2=0,\quad
f_i:=\sum_{\alpha=1,p=1}^{7,7}\theta_{\alpha p}^i x_\alpha z_p+ z_7z_p\; {\rm terms}\;= 0                 
$$
for $i=1,2$. Let $Z_k$\, be the variety carved out by\,
$0=f_0=\cdots=f_k,0\leq k\leq 2$. Suppose\,
$\sum_{\alpha=1,p=1}^{7,6}\theta_{\alpha p}^i x_\alpha z_p,i=1,2,$
are linearly independent. Then $Z_k,0\leq k\leq 2,$ are
irreducible of codimension $k+1$.
For an $f_3$ of homogeneous degree $2$ linearly independent
from $f_0,f_1,f_2$, we have that $f_0,f_1,f_2,f_3$ form a regular sequence and so they carve out
a subvariety of codimension $4$.
\end{lemma}

\begin{proof} The singular set of $f_0$ consists of points of the form $(0,z)$.
Hence the codimension 2 estimate goes through for $V_0$. Set
$$
R_0:=\begin{pmatrix}I&0\\0&0\end{pmatrix},\quad R_k:=\begin{pmatrix}
0&\theta_k\\\theta^{tr}_k&\tau_k\end{pmatrix},\quad k=1,2,
$$
where $I$ is 7-by-7, the 7-by-7 $\theta_k$ is defined similarly as in the preceding lemma, and
$\tau_k$ is a 7 by 7 symmetric matrix whose only nonzero row and column are the last one
corresponding to the coefficients of the $z_7z_p$ terms of $f_k$.
Again we estimate the
dimension of the kernel of
$$
S:=c_0R_0+\cdots+c_kR_k
$$
with $[c_0;\cdots:c_k]\in{\mathbb C}P^k,k=1,2,3.$ For simplicity, we may assume
$c_0=1$. Then
$$
S:=\begin{pmatrix} I&\Theta_k:=\sum_{l=1}^kc_l\theta_l\\(\Theta_k)^{tr}&\Pi_k:=\sum_l c_l\tau_l\end{pmatrix},
$$
whose kernel elements $(x,z)^{tr}$ satisfies
$$
x+\Theta_k z=0,\quad (\Theta_k)^{tr}x+\Pi_k z=0.
$$
From this we see that
$$
((\Theta_k)^{tr}\Theta_k+\Pi_k) z=0,
$$
so that the dimension of $z$ is at most 6 for a generic choice of $[c_0:\cdots:c_k]$
(respectively, 7 for a nongeneric choice) because the independence of
$\sum_{\alpha=1,p=1}^{7,6}\theta_{\alpha p}^i x_\alpha z_p,i=1,2,$
dictates that the upper left 6-by-6 block of $(\Theta_k)^{tr}\Theta_k$ is nonzero
for such a generic choice. Therefore, the fact that $x=-\Theta_k z$ implies that
the kernel dimension is at most 6 for a generic parameter $[c_0:\cdots:c_k]$
of dimension $k$. Hence the total dimension is at most $6+k$ (respectively,
$7+(k-1)=6+k$). On the other hand, $\dim(Z_k)-2\geq (14-k-1)-2=11-k$. Therefore,
the codimension 2 estimate goes through for  $Z_k,0\leq k\leq 2.$
\end{proof}

\begin{lemma}\label{import2} By the same token, if over\,
${\mathbb C}^{13}={\mathbb C}^6\oplus{\mathbb C}^7$
we are given
$$
f_0:=\sum_{\alpha=1}^6 (x_\alpha)^2=0,\quad
f_i:=\sum_{\alpha=1,p=1}^{6,7}\theta_{\alpha p}^i x_\alpha z_p
+z_6z_p\;\text{\rm terms}+ z_7z_p\; {\rm terms}\;= 0,
$$
$1\leq i\leq 2$. Let $Z_k$ be the variety carved out by $0=f_0=\cdots=f_k,0\leq k\leq 2$. Suppose\; $\sum_{\alpha=1,p=1}^{6,5}\theta_{\alpha p}^i x_\alpha z_p,i=1,2,$ are linearly independent,
then the codimension $2$ estimate goes through for $k\leq 2$, and so
$Z_k,k\leq 2,$ are irreducible of codimension $k+1$.
For an $f_3$ of homogeneous degree $2$ linearly independent
from $f_0,f_1,f_2$, we have that $f_0,f_1,f_2,f_3$ form a regular sequence and so they carve out
a subvariety of codimension $4$.

\end{lemma}

\end{appendix}

\end{document}